\documentclass[numbers=enddot,12pt,final,onecolumn,notitlepage]{scrartcl}%
\usepackage[headsepline,footsepline,manualmark]{scrlayer-scrpage}
\usepackage[all,cmtip]{xy}
\usepackage{amssymb}
\usepackage{amsmath}
\usepackage{amsthm}
\usepackage{framed}
\usepackage{comment}
\usepackage{color}
\usepackage[breaklinks=true]{hyperref}
\usepackage[sc]{mathpazo}
\usepackage[T1]{fontenc}
\usepackage{tikz}
\usepackage{needspace}
\usepackage{tabls}
\providecommand{\U}[1]{\protect\rule{.1in}{.1in}}
\usetikzlibrary{arrows}
\newcounter{exer}

\numberwithin{exer}{section}
\theoremstyle{definition}
\newtheorem{theo}{Theorem}[section]
\newenvironment{theorem}[1][]
{\begin{theo}[#1]\begin{leftbar}}
{\end{leftbar}\end{theo}}
\newtheorem{lem}[theo]{Lemma}
\newenvironment{lemma}[1][]
{\begin{lem}[#1]\begin{leftbar}}
{\end{leftbar}\end{lem}}
\newtheorem{prop}[theo]{Proposition}
\newenvironment{proposition}[1][]
{\begin{prop}[#1]\begin{leftbar}}
{\end{leftbar}\end{prop}}
\newtheorem{defi}[theo]{Definition}
\newenvironment{definition}[1][]
{\begin{defi}[#1]\begin{leftbar}}
{\end{leftbar}\end{defi}}
\newtheorem{remk}[theo]{Remark}

\newtheorem{coro}[theo]{Corollary}
\newenvironment{corollary}[1][]
{\begin{coro}[#1]\begin{leftbar}}
{\end{leftbar}\end{coro}}
\newtheorem{conv}[theo]{Convention}

\newtheorem{quest}[theo]{Question}

\newtheorem{warn}[theo]{Warning}

\newtheorem{conj}[theo]{Conjecture}

\newtheorem{exam}[theo]{Example}

\newtheorem{exmp}[exer]{Exercise}

\let\sumnonlimits\sum
\let\prodnonlimits\prod
\let\cupnonlimits\bigcup
\let\capnonlimits\bigcap
\renewcommand{\sum}{\sumnonlimits\limits}
\renewcommand{\prod}{\prodnonlimits\limits}
\renewcommand{\bigcup}{\cupnonlimits\limits}
\renewcommand{\bigcap}{\capnonlimits\limits}
\newenvironment{verlong}{}{}
\newenvironment{vershort}{}{}

\excludecomment{verlong}
\includecomment{vershort}
\excludecomment{noncompile}

\ihead{On binomial coefficients modulo squares of primes}
\ohead{page \thepage}
\begin{document}

\title{On binomial coefficients modulo squares of primes}
\author{Darij Grinberg}
\date{
\today
}
\maketitle

\begin{abstract}
\textbf{Abstract.} We give elementary proofs for the
Apagodu-Zeilberger-Stanton-Amdeberhan-Tauraso congruences%
\begin{align*}
\sum_{n=0}^{p-1}\dbinom{2n}{n}  &  \equiv\eta_{p}\operatorname{mod}p^{2};\\
\sum_{n=0}^{rp-1}\dbinom{2n}{n}  &  \equiv\eta_{p}\sum_{n=0}^{r-1}\dbinom
{2n}{n}\operatorname{mod}p^{2};\\
\sum_{n=0}^{rp-1}\sum_{m=0}^{sp-1}\dbinom{n+m}{m}^{2}  &  \equiv\eta_{p}%
\sum_{m=0}^{r-1}\sum_{n=0}^{s-1}\dbinom{n+m}{m}^{2}\operatorname{mod}p^{2},
\end{align*}
where $p$ is an odd prime, $r$ and $s$ are nonnegative integers, and%
\[
\eta_{p}=%
\begin{cases}
0, & \text{if }p\equiv0\operatorname{mod}3;\\
1, & \text{if }p\equiv1\operatorname{mod}3;\\
-1, & \text{if }p\equiv2\operatorname{mod}3
\end{cases}
.
\]

\end{abstract}
\tableofcontents

\section{Introduction}

In this note, we prove that any odd prime $p$ and any $r,s\in\mathbb{N}$
satisfy%
\begin{align*}
\sum_{n=0}^{p-1}\dbinom{2n}{n}  &  \equiv\eta_{p}\operatorname{mod}%
p^{2}\ \ \ \ \ \ \ \ \ \ \left(  \text{Theorem \ref{thm.Z1}}\right)  ;\\
\sum_{n=0}^{rp-1}\dbinom{2n}{n}  &  \equiv\eta_{p}\sum_{n=0}^{r-1}\dbinom
{2n}{n}\operatorname{mod}p^{2}\ \ \ \ \ \ \ \ \ \ \left(  \text{Theorem
\ref{thm.Z2}}\right)  ;\\
\sum_{n=0}^{rp-1}\sum_{m=0}^{sp-1}\dbinom{n+m}{m}^{2}  &  \equiv\eta_{p}%
\sum_{m=0}^{r-1}\sum_{n=0}^{s-1}\dbinom{n+m}{m}^{2}\operatorname{mod}%
p^{2}\ \ \ \ \ \ \ \ \ \ \left(  \text{Theorem \ref{thm.Z3}}\right)  ,
\end{align*}
where%
\[
\eta_{p}=%
\begin{cases}
0, & \text{if }p\equiv0\operatorname{mod}3;\\
1, & \text{if }p\equiv1\operatorname{mod}3;\\
-1, & \text{if }p\equiv2\operatorname{mod}3
\end{cases}
.
\]
These three congruences are (slightly extended versions of) three of the
\textquotedblleft Super-Conjectures\textquotedblright\ (namely, 1,
1\textquotedblright\ and 4') stated by Apagodu and Zeilberger in
\cite{ApaZei16}\footnote{In the arXiv preprint version of \cite{ApaZei16}
(\href{https://arxiv.org/abs/1606.03351v2}{arXiv:1606.03351v2}), these
congruences appear as \textquotedblleft Super-Conjectures\textquotedblright%
\ 1, 1\textquotedblright\ and 5', respectively.}. Our proofs are more
elementary than previous proofs by Stanton \cite{Stanto16} and Amdeberhan and
Tauraso \cite{AmdTau16}.

\subsection{Binomial coefficients}

Let us first recall the definition of binomial coefficients:\footnote{We use
the notation $\mathbb{N}$ for the set $\left\{  0,1,2,\ldots\right\}  $.}

\begin{definition}
\label{def.binom}Let $n\in\mathbb{N}$ and $m\in\mathbb{Z}$. Then, the
\textit{binomial coefficient} $\dbinom{m}{n}$ is a rational number defined by%
\[
\dbinom{m}{n}=\dfrac{m\left(  m-1\right)  \cdots\left(  m-n+1\right)  }{n!}.
\]

\end{definition}

\begin{verlong}
This definition is a particular case of the definition given in
\cite[Definition 3.1]{detnotes}.
\end{verlong}

\begin{verlong}
For the sake of convenience, let us slightly extend the notion of binomial coefficients:
\end{verlong}

\begin{definition}
\label{def.binom.negative}Let $n$ be a negative integer. Let $m\in\mathbb{Z}$.
Then, the \textit{binomial coefficient} $\dbinom{m}{n}$ is a rational number
defined by $\dbinom{m}{n}=0$.
\end{definition}

\begin{vershort}
(This is the definition used in \cite{GKP} and \cite{lucas}. Some authors
follow other conventions instead.)
\end{vershort}

\begin{verlong}
This convention is the one used by Graham, Knuth and Patashnik in \cite{GKP},
and also by myself in \cite{lucas}. Other authors use other conventions instead.
\end{verlong}

The following proposition is well-known (see, e.g., \cite[Proposition
1.9]{lucas}):

\begin{proposition}
\label{prop.binom.int}We have $\dbinom{m}{n}\in\mathbb{Z}$ for any
$m\in\mathbb{Z}$ and $n\in\mathbb{Z}$.
\end{proposition}

Proposition \ref{prop.binom.int} shows that $\dbinom{m}{n}$ is an integer
whenever $m\in\mathbb{Z}$ and $n\in\mathbb{Z}$. We shall tacitly use this
below, when we study congruences involving binomial coefficients.

One advantage of Definition \ref{def.binom.negative} is that it makes the
following hold:

\begin{verlong}
\begin{proposition}
\label{prop.binom.newton.1}For any $n\in\mathbb{Z}$, we have%
\[
\left(  1+X\right)  ^{n}=\sum_{m\in\mathbb{N}}\dbinom{n}{m}X^{m}%
\]
in the ring $\mathbb{Z}\left[  \left[  X\right]  \right]  $ of formal power series.
\end{proposition}
\end{verlong}

\begin{proposition}
\label{prop.binom.newton}For any $n\in\mathbb{Z}$ and $m\in\mathbb{Z}$, the
binomial coefficient $\dbinom{n}{m}$ is the coefficient of $X^{m}$ in the
formal power series $\left(  1+X\right)  ^{n}\in\mathbb{Z}\left[  \left[
X\right]  \right]  $. (Here, the coefficient of $X^{m}$ in any formal power
series is defined to be $0$ when $m$ is negative.)
\end{proposition}

\subsection{Classical congruences}

The behavior of binomial coefficients modulo primes and prime powers is a
classical subject of research; see \cite[\S 2.1]{Mestro14} for a survey of
much of it. Let us state two of the most basic results in this subject:

\begin{theorem}
\label{thm.lucas}Let $p$ be a prime. Let $a$ and $b$ be two integers. Let $c$
and $d$ be two elements of $\left\{  0,1,\ldots,p-1\right\}  $. Then,%
\[
\dbinom{ap+c}{bp+d}\equiv\dbinom{a}{b}\dbinom{c}{d}\operatorname{mod}p.
\]

\end{theorem}

Theorem \ref{thm.lucas} is known under the name of \textit{Lucas's theorem},
and is proven in many places (e.g., \cite[\S 2.1]{Mestro14} or \cite[Proof of
\S 4]{Hausne83} or \cite[proof of Lucas's theorem]{AnBeRo05} or \cite[Exercise
5.61]{GKP}) at least in the case when $a$ and $b$ are nonnegative integers.
The standard proof of Theorem \ref{thm.lucas} in this case uses generating
functions (specifically, Proposition \ref{prop.binom.newton}); this proof
applies (mutatis mutandis) in the general case as well. See \cite[Theorem
1.11]{lucas} for an elementary proof of Theorem \ref{thm.lucas}.

Another fundamental result is the following:

\begin{theorem}
\label{thm.p2cong}Let $p$ be a prime. Let $a$ and $b$ be two integers. Then,%
\[
\dbinom{ap}{bp}\equiv\dbinom{a}{b}\operatorname{mod}p^{2}.
\]

\end{theorem}

Theorem \ref{thm.p2cong} is a known result, perhaps due to Charles Babbage. It
appears with proof in \cite[Theorem 1.12]{lucas}; again, many sources prove it
for nonnegative $a$ and $b$ (for example \cite[Exercise 1.14 \textbf{c}%
]{Stanley-EC1} or \cite[Exercise 5.62]{GKP}). Notice that if $p\geq5$, then
the modulus $p^{2}$ can be replaced by $p^{3}$ or (depending on $a$, $b$ and
$p$) by even higher powers of $p$; see \cite[(22) and (23)]{Mestro14} for the
details. See also \cite[Lemma 2.1]{SunTau11} for another strengthening of
Theorem \ref{thm.p2cong}.

\subsection{The three modulo-$p^{2}$ congruences}

\begin{definition}
For any $p\in\mathbb{Z}$, we define an integer $\eta_{p}\in\left\{
-1,0,1\right\}  $ by%
\[
\eta_{p}=%
\begin{cases}
0, & \text{if }p\equiv0\operatorname{mod}3;\\
1, & \text{if }p\equiv1\operatorname{mod}3;\\
-1, & \text{if }p\equiv2\operatorname{mod}3
\end{cases}
.
\]

\end{definition}

Notice that $\eta_{p}$ is the so-called \textit{Legendre symbol} $\left(
\dfrac{p}{3}\right)  $ known from number theory.

We are now ready to state three conjectures by Apagodu and Zeilberger, which
we shall prove in the sequel. The first one is \cite[Super-Conjecture
1]{ApaZei16}:\footnote{To be precise (and boastful), our Theorem \ref{thm.Z1}
is somewhat stronger than \cite[Super-Conjecture 1]{ApaZei16}, since we only
require $p$ to be odd (rather than $p\geq5$). Of course, in the case of
Theorem \ref{thm.Z1}, this extra generality is insignificant, since it just
adds the possibility of $p=3$, in which case Theorem \ref{thm.Z1} can be
checked by hand. However, for Theorems \ref{thm.Z2} and \ref{thm.Z3} further
below, we gain somewhat more from this generality.}

\begin{theorem}
\label{thm.Z1}Let $p$ be an odd prime. Then,%
\[
\sum_{n=0}^{p-1}\dbinom{2n}{n}\equiv\eta_{p}\operatorname{mod}p^{2}.
\]

\end{theorem}

The next one (\cite[Super-Conjecture 1\textquotedblright]{ApaZei16}) is a generalization:

\begin{theorem}
\label{thm.Z2}Let $p$ be an odd prime. Let $r\in\mathbb{N}$. Set%
\[
\alpha_{r}=\sum_{n=0}^{r-1}\dbinom{2n}{n}.
\]
Then,%
\[
\sum_{n=0}^{rp-1}\dbinom{2n}{n}\equiv\eta_{p}\alpha_{r}\operatorname{mod}%
p^{2}.
\]

\end{theorem}

Theorem \ref{thm.Z1} and Theorem \ref{thm.Z2} both have been proven by Dennis
Stanton \cite{Stanto16} using Laurent series (in the case when $p\geq5$), and
by Liu \cite[(1.3)]{Liu16} using harmonic numbers. We shall reprove them elementarily.

The third conjecture that we shall prove is \cite[Super-Conjecture
5']{ApaZei16}:

\begin{theorem}
\label{thm.Z3}Let $p$ be an odd prime. Let $r\in\mathbb{N}$ and $s\in
\mathbb{N}$. Set
\[
\epsilon_{r,s}=\sum_{m=0}^{r-1}\sum_{n=0}^{s-1}\dbinom{n+m}{m}^{2}.
\]
Then,
\[
\sum_{n=0}^{rp-1}\sum_{m=0}^{sp-1}\dbinom{n+m}{m}^{2}\equiv\eta_{p}%
\epsilon_{r,s}\operatorname{mod}p^{2}.
\]

\end{theorem}

A proof of Theorem \ref{thm.Z3} has been found by Amdeberhan and Tauraso, and
was outlined in \cite[\S 6]{AmdTau16}; we give a different, elementary proof.

\section{The proofs}

\subsection{Identities and congruences from the literature}

Before we come to the proofs of Theorems \ref{thm.Z1}, \ref{thm.Z2} and
\ref{thm.Z3}, let us collect various well-known results that will prove useful.

The following properties of binomial coefficients are well-known (see, e.g.,
\cite[\S 3.1]{detnotes} and \cite[\S 1]{lucas}):

\begin{proposition}
\label{prop.binom.00}We have $\dbinom{m}{0}=1$ for every $m\in\mathbb{Z}$.
\end{proposition}

\begin{verlong}
Proposition \ref{prop.binom.00} is a particular case of \cite[Proposition 3.3
\textbf{(a)}]{detnotes}.
\end{verlong}

\begin{proposition}
\label{prop.binom.0}We have $\dbinom{m}{n}=0$ for every $m\in\mathbb{N}$ and
$n\in\mathbb{N}$ satisfying $m<n$.
\end{proposition}

\begin{verlong}
Proposition \ref{prop.binom.0} is \cite[Proposition 3.6]{detnotes}.
\end{verlong}

\begin{proposition}
\label{prop.binom.symm}We have $\dbinom{m}{n}=\dbinom{m}{m-n}$ for any
$m\in\mathbb{N}$ and $n\in\mathbb{N}$ satisfying $m\geq n$.
\end{proposition}

\begin{verlong}
Proposition \ref{prop.binom.symm} is \cite[Proposition 3.8]{detnotes}.
\end{verlong}

\begin{proposition}
\label{prop.binom.mm}We have $\dbinom{m}{m}=1$ for every $m\in\mathbb{N}$.
\end{proposition}

\begin{verlong}
Proposition \ref{prop.binom.mm} is \cite[Proposition 3.9]{detnotes}.
\end{verlong}

\begin{proposition}
\label{prop.binom.upper-neg}We have%
\[
\dbinom{m}{n}=\left(  -1\right)  ^{n}\dbinom{n-m-1}{n}%
\]
for any $m\in\mathbb{Z}$ and $n\in\mathbb{N}$.
\end{proposition}

\begin{verlong}
Proposition \ref{prop.binom.upper-neg} is a particular case of
\cite[Proposition 3.16]{detnotes}.
\end{verlong}

\begin{verlong}
The following fact is the classical recurrence relation for the binomial coefficients:
\end{verlong}

\begin{proposition}
\label{prop.binom.rec.neg}We have%
\[
\dbinom{m}{n}=\dbinom{m-1}{n-1}+\dbinom{m-1}{n}%
\]
for any $m\in\mathbb{Z}$ and $n\in\mathbb{Z}$.
\end{proposition}

\begin{verlong}
See \cite[Proposition 1.10]{lucas} for the straightforward proof of
Proposition \ref{prop.binom.rec.neg}.
\end{verlong}

\begin{proposition}
\label{prop.vandermonde.consequences}For every $x\in\mathbb{Z}$ and
$y\in\mathbb{Z}$ and $n\in\mathbb{N}$, we have
\[
\dbinom{x+y}{n}=\sum_{k=0}^{n}\dbinom{x}{k}\dbinom{y}{n-k}.
\]

\end{proposition}

Proposition \ref{prop.vandermonde.consequences} is the so-called
\textit{Vandermonde convolution identity}, and is a particular case of
\cite[Theorem 3.29]{detnotes}.

\begin{corollary}
\label{cor.ps2.2.S.-1}For each $n\in\mathbb{N}$, we have%
\[
\sum_{i=0}^{n-1}\left(  -1\right)  ^{i}\dbinom{n-1-i}{i}=\left(  -1\right)
^{n}\cdot%
\begin{cases}
0, & \text{if }n\equiv0\operatorname{mod}3;\\
-1, & \text{if }n\equiv1\operatorname{mod}3;\\
1, & \text{if }n\equiv2\operatorname{mod}3
\end{cases}
.
\]

\end{corollary}

Corollary \ref{cor.ps2.2.S.-1} is \cite[Corollary 8.63]{detnotes}. Apart from
that, Corollary \ref{cor.ps2.2.S.-1} can be easily derived from \cite[\S 5.2,
Problem 3]{GKP}, \cite[Identity 172]{BenQui03} or \cite{BenQui08}.

Another simple identity (sometimes known as the ``absorption identity'') is
the following:

\begin{proposition}
\label{prop.binom.absorb}Let $n\in\mathbb{Z}$ and $k\in\mathbb{Z}$. Then,
$k\dbinom{n}{k}=n\dbinom{n-1}{k-1}$.
\end{proposition}

Proposition~\ref{prop.binom.absorb} appears in \cite[(5.6)]{GKP}, and is
easily proven just from the definition of binomial coefficients.

\begin{verlong}
\begin{proof}
[Proof of Proposition \ref{prop.binom.absorb}.]We are in one of the following
three cases:

\textit{Case 1:} We have $k<0$.

\textit{Case 2:} We have $k=0$.

\textit{Case 3:} We have $k>0$.

Let us first consider Case 1. In this case, we have $k<0$. Thus, $k$ is a
negative integer (since $k\in\mathbb{Z}$). Hence, Definition
\ref{def.binom.negative} yields $\dbinom{n}{k}=0$. On the other hand,
$k-1<k<0$, so that $k-1$ is a negative integer (since $k-1\in\mathbb{Z}$).
Hence, Definition \ref{def.binom.negative} yields $\dbinom{n-1}{k-1}=0$.
Comparing $k\underbrace{\dbinom{n}{k}}_{=0}=0$ with $n\underbrace{\dbinom
{n-1}{k-1}}_{=0}=0$, we obtain $k\dbinom{n}{k}=n\dbinom{n-1}{k-1}$.
Proposition \ref{prop.binom.absorb} is thus proven in Case 1.

Let us now consider Case 2. In this case, we have $k=0$. Thus, $k-1=0-1=-1$,
so that $k-1$ is a negative integer (since $-1$ is a negative integer). Hence,
Definition \ref{def.binom.negative} yields $\dbinom{n-1}{k-1}=0$. Comparing
$\underbrace{k}_{=0}\dbinom{n}{k}=0$ with $n\underbrace{\dbinom{n-1}{k-1}%
}_{=0}=0$, we obtain $k\dbinom{n}{k}=n\dbinom{n-1}{k-1}$. Proposition
\ref{prop.binom.absorb} is thus proven in Case 2.

Let us finally consider Case 3. In this case, we have $k>0$. Hence, $k\geq1$
(since $k$ is an integer), so that $k-1\in\mathbb{N}$. Thus, the definition of
$\dbinom{n-1}{k-1}$ yields%
\begin{align*}
\dbinom{n-1}{k-1}  &  =\dfrac{\left(  n-1\right)  \left(  n-2\right)
\cdots\left(  \left(  n-1\right)  -\left(  k-1\right)  +1\right)  }{\left(
k-1\right)  !}\\
&  =\dfrac{\left(  n-1\right)  \left(  n-2\right)  \cdots\left(  n-k+1\right)
}{\left(  k-1\right)  !}%
\end{align*}
(since $\left(  n-1\right)  -\left(  k-1\right)  +1=n-k+1$). Also,
$k\in\mathbb{N}$ (since $k>0$ and $k\in\mathbb{Z}$). Hence, the definition of
$\dbinom{n}{k}$ yields%
\[
\dbinom{n}{k}=\dfrac{n\left(  n-1\right)  \cdots\left(  n-k+1\right)  }%
{k!}=\dfrac{n\left(  n-1\right)  \cdots\left(  n-k+1\right)  }{k\cdot\left(
k-1\right)  !}%
\]
(since $k!=k\cdot\left(  k-1\right)  !$ (because $k>0$)).

Now, comparing%
\begin{align*}
&  n\underbrace{\dbinom{n-1}{k-1}}_{=\dfrac{\left(  n-1\right)  \left(
n-2\right)  \cdots\left(  n-k+1\right)  }{\left(  k-1\right)  !}}\\
&  =n\dfrac{\left(  n-1\right)  \left(  n-2\right)  \cdots\left(
n-k+1\right)  }{\left(  k-1\right)  !}=\dfrac{1}{\left(  k-1\right)
!}\underbrace{n\cdot\left(  \left(  n-1\right)  \left(  n-2\right)
\cdots\left(  n-k+1\right)  \right)  }_{=n\left(  n-1\right)  \cdots\left(
n-k+1\right)  }\\
&  =\dfrac{1}{\left(  k-1\right)  !}n\left(  n-1\right)  \cdots\left(
n-k+1\right)  =\dfrac{n\left(  n-1\right)  \cdots\left(  n-k+1\right)
}{\left(  k-1\right)  !}%
\end{align*}
with%
\[
k\underbrace{\dbinom{n}{k}}_{=\dfrac{n\left(  n-1\right)  \cdots\left(
n-k+1\right)  }{k\cdot\left(  k-1\right)  !}}=k\dfrac{n\left(  n-1\right)
\cdots\left(  n-k+1\right)  }{k\cdot\left(  k-1\right)  !}=\dfrac{n\left(
n-1\right)  \cdots\left(  n-k+1\right)  }{\left(  k-1\right)  !},
\]
we obtain $k\dbinom{n}{k}=n\dbinom{n-1}{k-1}$. Proposition
\ref{prop.binom.absorb} is thus proven in Case 3.

Thus, Proposition \ref{prop.binom.absorb} has been proven in all three Cases
1, 2 and 3. Since these three Cases cover all possibilities, we thus conclude
that Proposition \ref{prop.binom.absorb} always holds. This completes its proof.
\end{proof}
\end{verlong}

Finally, we need the following result from elementary number theory:

\begin{theorem}
\label{thm.powersums.p}Let $p$ be a prime. Let $k\in\mathbb{N}$. Assume that
$k$ is not a positive multiple of $p-1$. Then,%
\[
\sum_{l=0}^{p-1}l^{k}\equiv0\operatorname{mod}p.
\]

\end{theorem}

Theorem \ref{thm.powersums.p} is proven, e.g., in \cite[Theorem 3.1]{lucas}
and (in a slightly rewritten form) in \cite[Theorem 1]{MacSon10}.

\subsection{Variants and consequences of Vandermonde convolution}

We are now going to state a number of identities that are restatements or
particular cases of the Vandermonde convolution identity (Proposition
\ref{prop.vandermonde.consequences}). We begin with the following one:

\begin{corollary}
\label{cor.vandermonde.2}Let $u\in\mathbb{Z}$ and $l\in\mathbb{N}$ and
$w\in\mathbb{N}$. Then,%
\[
\sum_{m=0}^{l}\dbinom{u}{w+m}\dbinom{l}{m}=\dbinom{u+l}{w+l}.
\]

\end{corollary}

\begin{vershort}
\begin{proof}
[Proof of Corollary \ref{cor.vandermonde.2}.]Proposition
\ref{prop.vandermonde.consequences} (applied to $x=u$, $y=l$ and $n=w+l$)
yields%
\begin{align*}
\dbinom{u+l}{w+l}  &  =\sum_{k=0}^{w+l}\dbinom{u}{k}\dbinom{l}{w+l-k}\\
&  =\sum_{k=0}^{w-1}\dbinom{u}{k}\underbrace{\dbinom{l}{w+l-k}}%
_{\substack{=0\\\text{(by Proposition \ref{prop.binom.0}}\\\text{(since
}l<w+l-k\text{ (because }k<w\text{)))}}}+\sum_{k=w}^{w+l}\dbinom{u}{k}%
\dbinom{l}{w+l-k}\\
&  \ \ \ \ \ \ \ \ \ \ \left(
\begin{array}
[c]{c}%
\text{here, we have split the sum at }k=w\text{,}\\
\text{since }0\leq w\leq w+l
\end{array}
\right) \\
&  =\underbrace{\sum_{k=0}^{w-1}\dbinom{u}{k}0}_{=0}+\sum_{k=w}^{w+l}%
\dbinom{u}{k}\dbinom{l}{w+l-k}=\sum_{k=w}^{w+l}\dbinom{u}{k}\dbinom{l}%
{w+l-k}\\
&  =\sum_{m=0}^{l}\dbinom{u}{w+m}\underbrace{\dbinom{l}{w+l-\left(
w+m\right)  }}_{\substack{=\dbinom{l}{l-m}=\dbinom{l}{m}\\\text{(by
Proposition \ref{prop.binom.symm})}}}\\
&  \ \ \ \ \ \ \ \ \ \ \left(  \text{here, we have substituted }w+m\text{ for
}k\text{ in the sum}\right) \\
&  =\sum_{m=0}^{l}\dbinom{u}{w+m}\dbinom{l}{m}.
\end{align*}
This proves Corollary \ref{cor.vandermonde.2}.
\end{proof}
\end{vershort}

\begin{verlong}
\begin{proof}
[Proof of Corollary \ref{cor.vandermonde.2}.]We have $l\in\mathbb{N}$, so that
$l\geq0$. We have $w\in\mathbb{N}$, so that $w\geq0$. Thus, $0\leq w\leq w+l$
(since $l\geq0$). Hence, $w+l\geq0$, so that $w+l\in\mathbb{N}$ (since
$w+l\in\mathbb{Z}$).

Each $k\in\left\{  0,1,\ldots,w-1\right\}  $ satisfies%
\begin{equation}
\dbinom{l}{w+l-k}=0 \label{pf.cor.vandermonde.2.1}%
\end{equation}
\footnote{\textit{Proof of (\ref{pf.cor.vandermonde.2.1}):} Let $k\in\left\{
0,1,\ldots,w-1\right\}  $. Thus, $k\leq w-1<w$, so that $w>k$. Hence,
$\underbrace{w}_{>k}+l-k>k+l-k=l$, so that $l<w+l-k$. Also, clearly,
$k\in\left\{  0,1,\ldots,w-1\right\}  \subseteq\mathbb{Z}$, so that
$w+l-k\in\mathbb{Z}$ and thus $w+l-k\in\mathbb{N}$ (since $w+l-k>l\geq0$).
Thus, Proposition \ref{prop.binom.0} (applied to $l$ and $w+l-k$ instead of
$m$ and $n$) yields $\dbinom{l}{w+l-k}=0$. This proves
(\ref{pf.cor.vandermonde.2.1}).}.

Each $m\in\left\{  0,1,\ldots,l\right\}  $ satisfies%
\begin{equation}
\dbinom{l}{w+l-\left(  w+m\right)  }=\dbinom{l}{m}
\label{pf.cor.vandermonde.2.2}%
\end{equation}
\footnote{\textit{Proof of (\ref{pf.cor.vandermonde.2.2}):} Let $m\in\left\{
0,1,\ldots,l\right\}  $. Thus, $m\leq l$, so that $l\geq m$. Also,
$m\in\left\{  0,1,\ldots,l\right\}  \subseteq\mathbb{N}$. Now, Proposition
\ref{prop.binom.symm} (applied to $l$ and $m$ instead of $m$ and $n$) yields
$\dbinom{l}{m}=\dbinom{l}{l-m}$. But $w+l-\left(  w+m\right)  =l-m$. Thus,
$\dbinom{l}{w+l-\left(  w+m\right)  }=\dbinom{l}{l-m}=\dbinom{l}{m}$. This
proves (\ref{pf.cor.vandermonde.2.2}).}.

Proposition \ref{prop.vandermonde.consequences} (applied to $x=u$, $y=l$ and
$n=w+l$) yields%
\begin{align*}
\dbinom{u+l}{w+l}  &  =\sum_{k=0}^{w+l}\dbinom{u}{k}\dbinom{l}{w+l-k}\\
&  =\sum_{k=0}^{w-1}\dbinom{u}{k}\underbrace{\dbinom{l}{w+l-k}}%
_{\substack{=0\\\text{(by (\ref{pf.cor.vandermonde.2.1}))}}}+\sum_{k=w}%
^{w+l}\dbinom{u}{k}\dbinom{l}{w+l-k}\\
&  \ \ \ \ \ \ \ \ \ \ \left(  \text{here, we have split the sum at }k=w\text{
(since }0\leq w\leq w+l\text{)}\right) \\
&  =\underbrace{\sum_{k=0}^{w-1}\dbinom{u}{k}0}_{=0}+\sum_{k=w}^{w+l}%
\dbinom{u}{k}\dbinom{l}{w+l-k}=\sum_{k=w}^{w+l}\dbinom{u}{k}\dbinom{l}%
{w+l-k}\\
&  =\sum_{m=0}^{l}\dbinom{u}{w+m}\underbrace{\dbinom{l}{w+l-\left(
w+m\right)  }}_{\substack{=\dbinom{l}{m}\\\text{(by
(\ref{pf.cor.vandermonde.2.2}))}}}\\
&  \ \ \ \ \ \ \ \ \ \ \left(  \text{here, we have substituted }w+m\text{ for
}k\text{ in the sum}\right) \\
&  =\sum_{m=0}^{l}\dbinom{u}{w+m}\dbinom{l}{m}.
\end{align*}
This proves Corollary \ref{cor.vandermonde.2}.
\end{proof}
\end{verlong}

Let us also state another corollary of Proposition
\ref{prop.vandermonde.consequences}:

\begin{corollary}
\label{cor.vandermonde.1}Let $x\in\mathbb{Z}$ and $y\in\mathbb{N}$ and
$n\in\mathbb{Z}$. Then,%
\[
\dbinom{x+y}{n}=\sum_{i=0}^{y}\dbinom{x}{n-i}\dbinom{y}{i}.
\]

\end{corollary}

See \cite[Corollary 2.2]{lucas} for a proof of Corollary
\ref{cor.vandermonde.1}.

\begin{lemma}
\label{lem.A1a}Let $u\in\mathbb{Z}$ and $w\in\mathbb{N}$ and $l\in\mathbb{N}$.
Then,%
\[
\dbinom{u+2l}{w+l}=\dbinom{u}{w}\dbinom{2l}{l}+\sum_{i=1}^{l}\left(
\dbinom{u}{w+i}+\dbinom{u}{w-i}\right)  \dbinom{2l}{l-i}.
\]

\end{lemma}

\begin{vershort}
\begin{proof}
[Proof of Lemma \ref{lem.A1a}.]Corollary \ref{cor.vandermonde.1} (applied to
$x=u$, $y=2l$ and $n=w+l$) yields%
\begin{align*}
\dbinom{u+2l}{w+l}  &  =\sum\limits_{i=0}^{2l}\dbinom{u}{w+l-i}\dbinom{2l}%
{i}=\sum\limits_{i=-l}^{l}\dbinom{u}{w+i}\dbinom{2l}{l-i}\\
&  \ \ \ \ \ \ \ \ \ \ \left(  \text{here, we have substituted }l-i\text{ for
}i\text{ in the sum}\right) \\
&  =\sum\limits_{\substack{i\in\left\{  -l,-l+1,\ldots,l\right\}  ;\\i\neq
0}}\dbinom{u}{w+i}\dbinom{2l}{l-i}+\dbinom{u}{w}\dbinom{2l}{l}%
\end{align*}
(here, we have split off the addend for $i=0$ from the sum). Hence,%
\begin{align*}
\dbinom{u+2l}{w+l}-\dbinom{u}{w}\dbinom{2l}{l}  &  =\sum
\limits_{\substack{i\in\left\{  -l,-l+1,\ldots,l\right\}  ;\\i\neq0}%
}\dbinom{u}{w+i}\dbinom{2l}{l-i}\\
&  =\sum\limits_{i=1}^{l}\dbinom{u}{w+i}\dbinom{2l}{l-i}+\sum\limits_{i=-l}%
^{-1}\dbinom{u}{w+i}\dbinom{2l}{l-i}\\
&  \ \ \ \ \ \ \ \ \ \ \left(
\begin{array}
[c]{c}%
\text{here, we have split the sum into two:}\\
\text{one for \textquotedblleft positive }i\text{\textquotedblright\ and one
for \textquotedblleft negative }i\text{\textquotedblright}%
\end{array}
\right) \\
&  =\sum\limits_{i=1}^{l}\dbinom{u}{w+i}\dbinom{2l}{l-i}+\sum\limits_{i=1}%
^{l}\dbinom{u}{w-i}\underbrace{\dbinom{2l}{l+i}}_{\substack{=\dbinom{2l}%
{l-i}\\\text{(by Proposition \ref{prop.binom.symm})}}}\\
&  \ \ \ \ \ \ \ \ \ \ \left(
\begin{array}
[c]{c}%
\text{here, we have substituted }-i\text{ for }i\\
\text{in the second sum}%
\end{array}
\right) \\
&  =\sum\limits_{i=1}^{l}\dbinom{u}{w+i}\dbinom{2l}{l-i}+\sum\limits_{i=1}%
^{l}\dbinom{u}{w-i}\dbinom{2l}{l-i}\\
&  =\sum_{i=1}^{l}\left(  \dbinom{u}{w+i}+\dbinom{u}{w-i}\right)  \dbinom
{2l}{l-i}.
\end{align*}
In other words,%
\[
\dbinom{u+2l}{w+l}=\dbinom{u}{w}\dbinom{2l}{l}+\sum_{i=1}^{l}\left(
\dbinom{u}{w+i}+\dbinom{u}{w-i}\right)  \dbinom{2l}{l-i}.
\]
This proves Lemma \ref{lem.A1a}.
\end{proof}
\end{vershort}

\begin{verlong}
\begin{proof}
[Proof of Lemma \ref{lem.A1a}.]We have $l\in\mathbb{N}$, so that $l\geq0$.
Hence, $-l\leq0$, so that $-l\leq0\leq l$ (since $l\geq0$).

Every $i\in\left\{  1,2,\ldots,l\right\}  $ satisfies%
\begin{equation}
\dbinom{2l}{l+i}=\dbinom{2l}{l-i} \label{pf.lem.A1a.symm}%
\end{equation}
\footnote{\textit{Proof of (\ref{pf.lem.A1a.symm}):} Let $i\in\left\{
1,2,\ldots,l\right\}  $. Thus, $1\leq i\leq l$. Hence, $l+\underbrace{i}_{\leq
l}\leq l+l=2l$, so that $2l\geq l+i$. Also, $i\in\left\{  1,2,\ldots
,l\right\}  \subseteq\mathbb{N}$, so that $l+i\in\mathbb{N}$ (since
$l\in\mathbb{N}$). Hence, Proposition \ref{prop.binom.symm} (applied to $2l$
and $l+i$ instead of $m$ and $n$) yields $\dbinom{2l}{l+i}=\dbinom
{2l}{2l-\left(  l+i\right)  }=\dbinom{2l}{l-i}$ (since $2l-\left(  l+i\right)
=l-i$). This proves (\ref{pf.lem.A1a.symm}).}.

Corollary \ref{cor.vandermonde.1} (applied to $x=u$, $y=2l$ and $n=w+l$)
yields%
\begin{align}
\dbinom{u+2l}{w+l}  &  =\sum\limits_{i=0}^{2l}\dbinom{u}{w+l-i}\dbinom{2l}%
{i}=\sum\limits_{i=-l}^{l}\underbrace{\dbinom{u}{w+l-\left(  l-i\right)  }%
}_{\substack{=\dbinom{u}{w+i}\\\text{(since }w+l-\left(  l-i\right)
=w+i\text{)}}}\dbinom{2l}{l-i}\nonumber\\
&  \ \ \ \ \ \ \ \ \ \ \left(  \text{here, we have substituted }l-i\text{ for
}i\text{ in the sum}\right) \nonumber\\
&  =\sum\limits_{i=-l}^{l}\dbinom{u}{w+i}\dbinom{2l}{l-i}\nonumber\\
&  =\sum\limits_{i=-l}^{0-1}\dbinom{u}{w+i}\dbinom{2l}{l-i}+\sum
\limits_{i=0}^{l}\dbinom{u}{w+i}\dbinom{2l}{l-i} \label{pf.lem.A1a.1}%
\end{align}
(here, we have split the sum at $i=0$ (since $-l\leq0\leq l$)).

But $0-1=-1$. Hence,%
\begin{align}
&  \sum\limits_{i=-l}^{0-1}\dbinom{u}{w+i}\dbinom{2l}{l-i}\nonumber\\
&  =\sum\limits_{i=-l}^{-1}\dbinom{u}{w+i}\dbinom{2l}{l-i}=\sum\limits_{i=1}%
^{l}\underbrace{\dbinom{u}{w+\left(  -i\right)  }}_{\substack{=\dbinom{u}%
{w-i}\\\text{(since }w+\left(  -i\right)  =w-i\text{)}}}\underbrace{\dbinom
{2l}{l-\left(  -i\right)  }}_{\substack{=\dbinom{2l}{l+i}\\\text{(since
}l-\left(  -i\right)  =l+i\text{)}}}\nonumber\\
&  =\sum\limits_{i=1}^{l}\dbinom{u}{w-i}\underbrace{\dbinom{2l}{l+i}%
}_{\substack{=\dbinom{2l}{l-i}\\\text{(by (\ref{pf.lem.A1a.symm}))}}%
}=\sum\limits_{i=1}^{l}\dbinom{u}{w-i}\dbinom{2l}{l-i}. \label{pf.lem.A1a.2}%
\end{align}

On the other hand, $l\in\left\{  0,1,\ldots,l\right\}  $ (since $l\in
\mathbb{N}$). Hence, we can split off the addend for $i=0$ from the sum
$\sum\limits_{i=0}^{l}\dbinom{u}{w+i}\dbinom{2l}{l-i}$. We thus obtain%
\begin{align}
&  \sum\limits_{i=0}^{l}\dbinom{u}{w+i}\dbinom{2l}{l-i}\nonumber\\
&  =\sum\limits_{i=1}^{l}\dbinom{u}{w+i}\dbinom{2l}{l-i}+\underbrace{\dbinom
{u}{w+0}}_{=\dbinom{u}{w}}\underbrace{\dbinom{2l}{l-0}}_{=\dbinom{2l}{l}%
}\nonumber\\
&  =\sum\limits_{i=1}^{l}\dbinom{u}{w+i}\dbinom{2l}{l-i}+\dbinom{u}{w}%
\dbinom{2l}{l}. \label{pf.lem.A1a.3}%
\end{align}

Now, (\ref{pf.lem.A1a.1}) becomes%
\begin{align*}
\dbinom{u+2l}{w+l}  &  =\underbrace{\sum\limits_{i=-l}^{0-1}\dbinom{u}%
{w+i}\dbinom{2l}{l-i}}_{\substack{=\sum\limits_{i=1}^{l}\dbinom{u}{w-i}%
\dbinom{2l}{l-i}\\\text{(by (\ref{pf.lem.A1a.2}))}}}+\underbrace{\sum
\limits_{i=0}^{l}\dbinom{u}{w+i}\dbinom{2l}{l-i}}_{\substack{=\sum
\limits_{i=1}^{l}\dbinom{u}{w+i}\dbinom{2l}{l-i}+\dbinom{u}{w}\dbinom{2l}%
{l}\\\text{(by (\ref{pf.lem.A1a.3}))}}}\\
&  =\underbrace{\sum\limits_{i=1}^{l}\dbinom{u}{w-i}\dbinom{2l}{l-i}%
+\sum\limits_{i=1}^{l}\dbinom{u}{w+i}\dbinom{2l}{l-i}}_{\substack{=\sum
\limits_{i=1}^{l}\dbinom{u}{w+i}\dbinom{2l}{l-i}+\sum\limits_{i=1}^{l}%
\dbinom{u}{w-i}\dbinom{2l}{l-i}\\=\sum_{i=1}^{l}\left(  \dbinom{u}%
{w+i}+\dbinom{u}{w-i}\right)  \dbinom{2l}{l-i}}}+\dbinom{u}{w}\dbinom{2l}{l}\\
&  =\sum_{i=1}^{l}\left(  \dbinom{u}{w+i}+\dbinom{u}{w-i}\right)  \dbinom
{2l}{l-i}+\dbinom{u}{w}\dbinom{2l}{l}\\
&  =\dbinom{u}{w}\dbinom{2l}{l}+\sum_{i=1}^{l}\left(  \dbinom{u}{w+i}%
+\dbinom{u}{w-i}\right)  \dbinom{2l}{l-i}.
\end{align*}
This proves Lemma \ref{lem.A1a}.
\end{proof}
\end{verlong}

\begin{lemma}
\label{lem.vandermonde.3}Let $p\in\mathbb{N}$. Let $c\in\mathbb{Z}$. Let
$l\in\left\{  0,1,\ldots,p-1\right\}  $. Then,%
\[
\dbinom{cp+2l}{l}=\sum\limits_{k=0}^{p-1}\dbinom{cp+l}{k}\dbinom{l}{k}.
\]

\end{lemma}

\begin{vershort}
\begin{proof}
[Proof of Lemma \ref{lem.vandermonde.3}.]Corollary \ref{cor.vandermonde.1}
(applied to $x=cp+l$, $y=l$ and $n=l$) yields%
\begin{align*}
\dbinom{cp+l+l}{l}  &  =\sum_{i=0}^{l}\dbinom{cp+l}{l-i}\dbinom{l}{i}%
=\sum\limits_{k=0}^{l}\dbinom{cp+l}{k}\underbrace{\dbinom{l}{l-k}%
}_{\substack{=\dbinom{l}{k}\\\text{(by Proposition \ref{prop.binom.symm})}}}\\
&  \ \ \ \ \ \ \ \ \ \ \left(  \text{here, we have substituted }k\text{ for
}l-i\text{ in the sum}\right) \\
&  =\sum\limits_{k=0}^{l}\dbinom{cp+l}{k}\dbinom{l}{k}.
\end{align*}

Comparing this with%
\begin{align*}
\sum\limits_{k=0}^{p-1}\dbinom{cp+l}{k}\dbinom{l}{k}  &  =\sum\limits_{k=0}%
^{l}\dbinom{cp+l}{k}\dbinom{l}{k}+\sum\limits_{k=l+1}^{p-1}\dbinom{cp+l}%
{k}\underbrace{\dbinom{l}{k}}_{\substack{=0\\\text{(by Proposition
\ref{prop.binom.0}}\\\text{(applied to }m=l\text{ and }n=k\text{)}%
\\\text{(since }l<k\text{))}}}\\
&  \ \ \ \ \ \ \ \ \ \ \left(  \text{here, we have split the sum at
}k=l\text{, since }0\leq l\leq p-1\right) \\
&  =\sum\limits_{k=0}^{l}\dbinom{cp+l}{k}\dbinom{l}{k}+\underbrace{\sum
\limits_{k=l+1}^{p-1}\dbinom{cp+l}{k}0}_{=0}=\sum\limits_{k=0}^{l}%
\dbinom{cp+l}{k}\dbinom{l}{k},
\end{align*}
we obtain $\sum\limits_{k=0}^{p-1}\dbinom{cp+l}{k}\dbinom{l}{k}=\dbinom
{cp+l+l}{l}=\dbinom{cp+2l}{l}$. This proves Lemma \ref{lem.vandermonde.3}.
\end{proof}
\end{vershort}

\begin{verlong}
\begin{proof}
[Proof of Lemma \ref{lem.vandermonde.3}.]We have $l\in\left\{  0,1,\ldots
,p-1\right\}  \subseteq\mathbb{N}$.

Each $k\in\left\{  0,1,\ldots,l\right\}  $ satisfies%
\begin{equation}
\dbinom{l}{l-k}=\dbinom{l}{k} \label{pf.lem.vandermonde.3.1}%
\end{equation}
\footnote{\textit{Proof of (\ref{pf.lem.vandermonde.3.1}):} Let $k\in\left\{
0,1,\ldots,l\right\}  $. Thus, $0\leq k\leq l$. Hence, $l\geq k$. Also,
$k\in\left\{  0,1,\ldots,l\right\}  \subseteq\mathbb{N}$. Hence, Proposition
\ref{prop.binom.symm} (applied to $l$ and $k$ instead of $m$ and $n$) yields
$\dbinom{l}{k}=\dbinom{l}{l-k}$. This proves (\ref{pf.lem.vandermonde.3.1}).}.

Each $k\in\left\{  l+1,l+2,\ldots,p-1\right\}  $ satisfies%
\begin{equation}
\dbinom{l}{k}=0 \label{pf.lem.vandermonde.3.2}%
\end{equation}
\footnote{\textit{Proof of (\ref{pf.lem.vandermonde.3.2}):} Let $k\in\left\{
l+1,l+2,\ldots,p-1\right\}  $. Thus, $k\geq l+1>l\geq0$ (since $l\in\left\{
0,1,\ldots,p-1\right\}  $). Hence, $k\in\mathbb{N}$ (since $k\in\left\{
l+1,l+2,\ldots,p-1\right\}  \subseteq\mathbb{Z}$).
\par
Also, $l<k$ (since $k>l$) and $l\in\mathbb{N}$. Thus, Proposition
\ref{prop.binom.0} (applied to $l$ and $k$ instead of $m$ and $n$) yields
$\dbinom{l}{k}=0$. This proves (\ref{pf.lem.vandermonde.3.2}).}.

We have $l\in\left\{  0,1,\ldots,p-1\right\}  $, so that $0\leq l\leq p-1$.

We have $l\in\mathbb{N}$. Hence, Corollary \ref{cor.vandermonde.1} (applied to
$x=cp+l$, $y=l$ and $n=l$) yields%
\begin{align*}
\dbinom{cp+l+l}{l}  &  =\sum_{i=0}^{l}\dbinom{cp+l}{l-i}\dbinom{l}{i}%
=\sum\limits_{k=0}^{l}\underbrace{\dbinom{cp+l}{l-\left(  l-k\right)  }%
}_{\substack{=\dbinom{cp+l}{k}\\\text{(since }l-\left(  l-k\right)
=k\text{)}}}\underbrace{\dbinom{l}{l-k}}_{\substack{=\dbinom{l}{k}\\\text{(by
(\ref{pf.lem.vandermonde.3.1}))}}}\\
&  \ \ \ \ \ \ \ \ \ \ \left(  \text{here, we have substituted }k\text{ for
}l-i\text{ in the sum}\right) \\
&  =\sum\limits_{k=0}^{l}\dbinom{cp+l}{k}\dbinom{l}{k}.
\end{align*}

Comparing this with%
\begin{align*}
\sum\limits_{k=0}^{p-1}\dbinom{cp+l}{k}\dbinom{l}{k}  &  =\sum\limits_{k=0}%
^{l}\dbinom{cp+l}{k}\dbinom{l}{k}+\sum\limits_{k=l+1}^{p-1}\dbinom{cp+l}%
{k}\underbrace{\dbinom{l}{k}}_{\substack{=0\\\text{(by
(\ref{pf.lem.vandermonde.3.2}))}}}\\
&  \ \ \ \ \ \ \ \ \ \ \left(  \text{here, we have split the sum at }k=l\text{
(since }0\leq l\leq p-1\text{)}\right) \\
&  =\sum\limits_{k=0}^{l}\dbinom{cp+l}{k}\dbinom{l}{k}+\underbrace{\sum
\limits_{k=l+1}^{p-1}\dbinom{cp+l}{k}0}_{=0}=\sum\limits_{k=0}^{l}%
\dbinom{cp+l}{k}\dbinom{l}{k},
\end{align*}
we obtain $\sum\limits_{k=0}^{p-1}\dbinom{cp+l}{k}\dbinom{l}{k}=\dbinom
{cp+l+l}{l}=\dbinom{cp+2l}{l}$ (since $l+l=2l$). This proves Lemma
\ref{lem.vandermonde.3}.
\end{proof}
\end{verlong}

\begin{lemma}
\label{lem.vandermonde.4}Let $p\in\mathbb{N}$. Let $l\in\mathbb{N}$. Then,%
\[
\sum_{i=1}^{l}\dbinom{p}{i}\dbinom{2l}{l-i}=\dbinom{p+2l}{l}-\dbinom{2l}{l}.
\]

\end{lemma}

\begin{vershort}
\begin{proof}
[Proof of Lemma \ref{lem.vandermonde.4}.]Proposition
\ref{prop.vandermonde.consequences} (applied to $x=p$, $y=2l$ and $n=l$)
yields%
\begin{align*}
\dbinom{p+2l}{l}  &  =\sum_{k=0}^{l}\dbinom{p}{k}\dbinom{2l}{l-k}=\sum
_{i=0}^{l}\dbinom{p}{i}\dbinom{2l}{l-i}\\
&  \ \ \ \ \ \ \ \ \ \ \left(  \text{here, we have renamed the summation index
}k\text{ as }i\right) \\
&  =\underbrace{\dbinom{p}{0}}_{=1}\underbrace{\dbinom{2l}{l-0}}_{=\dbinom
{2l}{l}}+\sum_{i=1}^{l}\dbinom{p}{i}\dbinom{2l}{l-i}=\dbinom{2l}{l}+\sum
_{i=1}^{l}\dbinom{p}{i}\dbinom{2l}{l-i}.
\end{align*}
Thus,%
\[
\sum_{i=1}^{l}\dbinom{p}{i}\dbinom{2l}{l-i}=\dbinom{p+2l}{l}-\dbinom{2l}{l}.
\]
This proves Lemma \ref{lem.vandermonde.4}.
\end{proof}
\end{vershort}

\begin{verlong}
\begin{proof}
[Proof of Lemma \ref{lem.vandermonde.4}.]We have $0\in\left\{  0,1,\ldots
,l\right\}  $ (since $l\in\mathbb{N}$).

Proposition \ref{prop.vandermonde.consequences} (applied to $x=p$, $y=2l$ and
$n=l$) yields%
\begin{align*}
\dbinom{p+2l}{l}  &  =\sum_{k=0}^{l}\dbinom{p}{k}\dbinom{2l}{l-k}=\sum
_{i=0}^{l}\dbinom{p}{i}\dbinom{2l}{l-i}\\
&  \ \ \ \ \ \ \ \ \ \ \left(  \text{here, we have renamed the summation index
}k\text{ as }i\right) \\
&  =\underbrace{\dbinom{p}{0}}_{\substack{=1\\\text{(by Proposition
\ref{prop.binom.00}}\\\text{(applied to }m=p\text{))}}}\underbrace{\dbinom
{2l}{l-0}}_{=\dbinom{2l}{l}}+\sum_{i=1}^{l}\dbinom{p}{i}\dbinom{2l}{l-i}\\
&  \ \ \ \ \ \ \ \ \ \ \left(
\begin{array}
[c]{c}%
\text{here, we have split off the addend for }i=0\text{ from the sum}\\
\text{(since }0\in\left\{  0,1,\ldots,l\right\}  \text{)}%
\end{array}
\right) \\
&  =\dbinom{2l}{l}+\sum_{i=1}^{l}\dbinom{p}{i}\dbinom{2l}{l-i}.
\end{align*}
Solving this equation for $\sum_{i=1}^{l}\dbinom{p}{i}\dbinom{2l}{l-i}$, we
find
\[
\sum_{i=1}^{l}\dbinom{p}{i}\dbinom{2l}{l-i}=\dbinom{p+2l}{l}-\dbinom{2l}{l}.
\]
This proves Lemma \ref{lem.vandermonde.4}.
\end{proof}
\end{verlong}

\subsection{A congruence of Bailey's}

Next, we shall prove a modulo-$p^{2}$ congruence for certain binomial
coefficients that can be regarded as a counterpart to Theorem \ref{thm.p2cong}:

\begin{theorem}
\label{thm.bailey}Let $p$ be a prime. Let $N\in\mathbb{Z}$ and $K\in
\mathbb{Z}$ and $i\in\left\{  1,2,\ldots,p-1\right\}  $. Then:

\textbf{(a)} We have%
\[
\dbinom{Np}{Kp+i}\equiv N\dbinom{N-1}{K}\dbinom{p}{i}\operatorname{mod}p^{2}.
\]

\textbf{(b)} We have%
\[
\dbinom{Np}{Kp-i}\equiv N\dbinom{N-1}{K-1}\dbinom{p}{i}\operatorname{mod}%
p^{2}.
\]

\textbf{(c)} We have%
\[
\dbinom{Np}{Kp+i}+\dbinom{Np}{Kp-i}\equiv N\dbinom{N}{K}\dbinom{p}%
{i}\operatorname{mod}p^{2}.
\]

\end{theorem}

Theorem \ref{thm.bailey} \textbf{(a)} is essentially the result \cite[Theorem
4]{Bailey91} by Bailey (see also \cite[(26)]{Mestro14}); in fact, it
transforms into \cite[Theorem 4]{Bailey91} if we rewrite $N\dbinom{N-1}{K}$ as
$\left(  K+1\right)  \dbinom{N}{K+1}$ (using Proposition
\ref{prop.binom.absorb}). We shall nevertheless give our own proof.

\begin{vershort}
\begin{proof}
[Proof of Theorem \ref{thm.bailey}.]From $i\in\left\{  1,2,\ldots,p-1\right\}
$, we conclude that both $i-1$ and $p-i$ are elements of $\left\{
0,1,\ldots,p-1\right\}  $. Notice also that $i$ is not divisible by $p$ (since
$i\in\left\{  1,2,\ldots,p-1\right\}  $); hence, $i$ is coprime to $p$ (since
$p$ is a prime). Therefore, $i$ is also coprime to $p^{2}$.

\textbf{(a)} Proposition \ref{prop.binom.absorb} (applied to $n=Np$ and
$k=Kp+i$) yields%
\begin{align}
\left(  Kp+i\right)  \dbinom{Np}{Kp+i}  &  =Np\dbinom{Np-1}{Kp+i-1}%
=Np\underbrace{\dbinom{\left(  N-1\right)  p+\left(  p-1\right)  }{Kp+\left(
i-1\right)  }}_{\substack{\equiv\dbinom{N-1}{K}\dbinom{p-1}{i-1}%
\operatorname{mod}p\\\text{(by Theorem \ref{thm.lucas}, applied to}%
\\a=N-1\text{, }b=K\text{, }c=p-1\text{ and }d=i-1\text{)}}}\nonumber\\
&  \equiv Np\dbinom{N-1}{K}\dbinom{p-1}{i-1}\operatorname{mod}p^{2}
\label{pf.thm.bailey.a.0}%
\end{align}
(notice that the presence of the $p$ factor has turned a congruence modulo $p$
into a congruence modulo $p^{2}$). Thus,%
\[
\left(  Kp+i\right)  \dbinom{Np}{Kp+i}\equiv Np\dbinom{N-1}{K}\dbinom
{p-1}{i-1}\equiv0\operatorname{mod}p,
\]
so that $0\equiv\underbrace{\left(  Kp+i\right)  }_{\equiv i\operatorname{mod}%
p}\dbinom{Np}{Kp+i}\equiv i\dbinom{Np}{Kp+i}\operatorname{mod}p$. We can
cancel $i$ from this congruence (since $i$ is coprime to $p$), and thus obtain
$0\equiv\dbinom{Np}{Kp+i}\operatorname{mod}p$. Hence, $\dbinom{Np}{Kp+i}$ is
divisible by $p$. Thus, $p\dbinom{Np}{Kp+i}$ is divisible by $p^{2}$. In other
words,%
\begin{equation}
p\dbinom{Np}{Kp+i}\equiv0\operatorname{mod}p^{2}. \label{pf.thm.bailey.a.1}%
\end{equation}
Now,%
\[
\left(  Kp+i\right)  \dbinom{Np}{Kp+i}=K\underbrace{p\dbinom{Np}{Kp+i}%
}_{\substack{\equiv0\operatorname{mod}p^{2}\\\text{(by
(\ref{pf.thm.bailey.a.1}))}}}+i\dbinom{Np}{Kp+i}\equiv i\dbinom{Np}%
{Kp+i}\operatorname{mod}p^{2}.
\]
Hence,%
\begin{align*}
i\dbinom{Np}{Kp+i}  &  \equiv\left(  Kp+i\right)  \dbinom{Np}{Kp+i}\equiv
Np\dbinom{N-1}{K}\dbinom{p-1}{i-1}\ \ \ \ \ \ \ \ \ \ \left(  \text{by
(\ref{pf.thm.bailey.a.0})}\right) \\
&  =N\dbinom{N-1}{K}\underbrace{p\dbinom{p-1}{i-1}}_{\substack{=i\dbinom{p}%
{i}\\\text{(by Proposition \ref{prop.binom.absorb})}}}=N\dbinom{N-1}%
{K}i\dbinom{p}{i}\operatorname{mod}p^{2}.
\end{align*}
We can cancel $i$ from this congruence (since $i$ is coprime to $p^{2}$), and
thus obtain%
\[
\dbinom{Np}{Kp+i}\equiv N\dbinom{N-1}{K}\dbinom{p}{i}\operatorname{mod}p^{2}.
\]
This proves Theorem \ref{thm.bailey} \textbf{(a)}.

\textbf{(b)} We have $i\in\left\{  1,2,\ldots,p-1\right\}  $ and thus
$p-i\in\left\{  1,2,\ldots,p-1\right\}  $. Hence, Theorem \ref{thm.bailey}
\textbf{(a)} (applied to $K-1$ and $p-i$ instead of $K$ and $i$) yields%
\[
\dbinom{Np}{\left(  K-1\right)  p+\left(  p-i\right)  }\equiv N\dbinom
{N-1}{K-1}\underbrace{\dbinom{p}{p-i}}_{\substack{=\dbinom{p}{i}\\\text{(by
Proposition \ref{prop.binom.symm})}}}=N\dbinom{N-1}{K-1}\dbinom{p}%
{i}\operatorname{mod}p^{2}.
\]
In view of $\left(  K-1\right)  p+\left(  p-i\right)  =Kp-i$, this rewrites as%
\[
\dbinom{Np}{Kp-i}\equiv N\dbinom{N-1}{K-1}\dbinom{p}{i}\operatorname{mod}%
p^{2}.
\]
This proves Theorem \ref{thm.bailey} \textbf{(b)}.

\textbf{(c)} We have
\begin{align*}
&  \underbrace{\dbinom{Np}{Kp+i}}_{\substack{\equiv N\dbinom{N-1}{K}\dbinom
{p}{i}\operatorname{mod}p^{2}\\\text{(by Theorem \ref{thm.bailey}
\textbf{(a)})}}}+\underbrace{\dbinom{Np}{Kp-i}}_{\substack{\equiv
N\dbinom{N-1}{K-1}\dbinom{p}{i}\operatorname{mod}p^{2}\\\text{(by Theorem
\ref{thm.bailey} \textbf{(b)})}}}\\
&  \equiv N\dbinom{N-1}{K}\dbinom{p}{i}+N\dbinom{N-1}{K-1}\dbinom{p}{i}\\
&  =N\underbrace{\left(  \dbinom{N-1}{K-1}+\dbinom{N-1}{K}\right)
}_{\substack{=\dbinom{N}{K}\\\text{(by Proposition \ref{prop.binom.rec.neg})}%
}}\dbinom{p}{i}=N\dbinom{N}{K}\dbinom{p}{i}\operatorname{mod}p^{2}.
\end{align*}
This proves Theorem \ref{thm.bailey} \textbf{(c)}.
\end{proof}
\end{vershort}

\begin{verlong}
Before we prove Theorem \ref{thm.bailey}, let us recall two basic facts:

\begin{lemma}
\label{lem.2.1.p44.lem1}Let $p$ be a prime. Let $k\in\left\{  0,1,\ldots
,p-1\right\}  $. Then, $k!$ is coprime to $p$.
\end{lemma}

Lemma \ref{lem.2.1.p44.lem1} is \cite[Lemma 5.2]{fleck}. We can somewhat
improve Lemma \ref{lem.2.1.p44.lem1}:

\begin{lemma}
\label{lem.2.1.p44.lem1p2}Let $p$ be a prime. Let $i\in\mathbb{Z}$ be such
that $i$ is coprime to $p$. Then, $i$ is coprime to $p^{2}$.
\end{lemma}

\begin{proof}
[Proof of Lemma \ref{lem.2.1.p44.lem1p2}.]We assumed that $i$ is coprime to
$p$. In other words, $\gcd\left(  i,p\right)  =1$.

Let $g=\gcd\left(  i,p^{2}\right)  $. Assume (for the sake of contradiction)
that $g\neq1$.

Clearly, $\gcd\left(  i,p^{2}\right)  $ is a positive integer (since $i$ is a
positive integer). In other words, $g$ is a positive integer (since
$g=\gcd\left(  i,p^{2}\right)  $). Hence, $g$ is a positive integer and a
divisor of $p^{2}$ (since $g=\gcd\left(  i,p^{2}\right)  \mid p^{2}$). Thus,
$g$ is a positive divisor of $p^{2}$. But the only positive divisors of
$p^{2}$ are $1$, $p$ and $p^{2}$ (since $p$ is a prime). Hence, every positive
divisor of $p^{2}$ belongs to the set $\left\{  1,p,p^{2}\right\}  $. In other
words, if $r$ is a positive divisor of $p^{2}$, then $r\in\left\{
1,p,p^{2}\right\}  $. Applying this to $r=g$, we obtain $g\in\left\{
1,p,p^{2}\right\}  $ (since $g$ is a positive divisor of $p^{2}$). Combining
this with $g\neq1$, we conclude that $g\in\left\{  1,p,p^{2}\right\}
\setminus\left\{  1\right\}  \subseteq\left\{  p,p^{2}\right\}  $. Hence,
$p\mid g$\ \ \ \ \footnote{\textit{Proof.} We have either $g=p$ or $g=p^{2}$
(since $g\in\left\{  p,p^{2}\right\}  $). Thus, we are in one of the following
two cases:
\par
\textit{Case 1:} We have $g=p$.
\par
\textit{Case 2:} We have $g=p^{2}$.
\par
Let us first consider Case 1. In this case, we have $g=p$. Hence, $p\mid p=g$.
Thus, $p\mid g$ is proven in Case 1.
\par
Let us now consider Case 2. In this case, we have $g=p^{2}$. Hence, $p\mid
pp=p^{2}=g$. Thus, $p\mid g$ is proven in Case 2.
\par
We have now proven $p\mid g$ in each of the two Cases 1 and 2. Since these two
Cases cover all possibilities, we thus conclude that $p\mid g$ always holds.
Qed.}. Thus, $p\mid g=\gcd\left(  i,p^{2}\right)  \mid i$. Also, $p\mid p$.

But one of the most basic properties of the greatest common divisor says the
following: If three positive integers $a$, $b$ and $c$ satisfy $c\mid a$ and
$c\mid b$, then they also satisfy $c\mid\gcd\left(  a,b\right)  $.

Applying this to $a=i$, $b=p$ and $c=p$, we conclude that $p\mid\gcd\left(
i,p\right)  $ (since $p\mid i$ and $p\mid p$). Thus, $p\mid\gcd\left(
i,p\right)  =1$, so that $p\leq1$ (since $p$ and $1$ are positive integers).
This contradicts the fact that $p$ is a prime. This contradiction shows that
our assumption (that $g\neq1$) was false. Hence, we cannot have $g\neq1$.
Thus, we must have $g=1$. Hence, $\gcd\left(  i,p^{2}\right)  =g=1$. In other
words, $i$ is coprime to $p^{2}$. This proves Lemma \ref{lem.2.1.p44.lem1p2}.
\end{proof}

\begin{lemma}
\label{lem.2.1.p44.lem2}Let $b$ and $c$ be integers such that $c$ is nonzero
and such that $b$ is coprime to $c$. Let $a$ and $a^{\prime}$ be integers such
that $ba\equiv ba^{\prime}\operatorname{mod}c$. Then, $a\equiv a^{\prime
}\operatorname{mod}c$.
\end{lemma}

Lemma \ref{lem.2.1.p44.lem2} is \cite[Lemma 5.3]{fleck}.

Here is yet another trivial lemma:

\begin{lemma}
\label{lem.congruence-times-p}Let $a$, $b$ and $c$ be three integers such that
$c$ is nonzero. Assume that $a\equiv b\operatorname{mod}c$. Then, $ca\equiv
cb\operatorname{mod}c^{2}$.
\end{lemma}

\begin{proof}
[Proof of Lemma \ref{lem.congruence-times-p}.]We have $a\equiv
b\operatorname{mod}c$. In other words, $c\mid a-b$. In other words, there
exists some $m\in\mathbb{Z}$ such that $a-b=cm$. Consider this $m$. Now,
$ca-cb=c\underbrace{\left(  a-b\right)  }_{=cm}=ccm=c^{2}m$. But $c^{2}\mid
c^{2}m=ca-cb$. In other words, $ca\equiv cb\operatorname{mod}c^{2}$. This
proves Lemma \ref{lem.congruence-times-p}.
\end{proof}

\begin{proof}
[Proof of Theorem \ref{thm.bailey}.]From $i\in\left\{  1,2,\ldots,p-1\right\}
$, we conclude that $p-i\in\left\{  1,2,\ldots,p-1\right\}  \subseteq\left\{
0,1,\ldots,p-1\right\}  $. Also, $i\in\left\{  1,2,\ldots,p-1\right\}
\subseteq\left\{  0,1,\ldots,p-1\right\}  $. Moreover, $i$ is a positive
integer (since $i\in\left\{  1,2,\ldots,p-1\right\}  $). Also, from
$i\in\left\{  1,2,\ldots,p-1\right\}  $, we obtain $i\leq p-1<p$.

Also, the integer $i$ is not divisible by $p$\ \ \ \ \footnote{\textit{Proof.}
Assume the contrary. Thus, the integer $i$ is divisible by $p$. Since both $i$
and $p$ are positive integers, we thus conclude that $i\geq p$. This
contradicts $i<p$.
\par
This contradiction shows that our assumption was wrong, qed.}. Hence, $i$ is
coprime to $p$ (since $p$ is a prime). Therefore, $i$ is also coprime to
$p^{2}$ (by Lemma \ref{lem.2.1.p44.lem1p2}).

Proposition \ref{prop.binom.absorb} (applied to $n=p$ and $k=i$) yields%
\begin{equation}
i\dbinom{p}{i}=p\dbinom{p-1}{i-1}. \label{pf.thm.bailey.long.absorb-p-i}%
\end{equation}

\textbf{(a)} Recall that $p$ is a prime; thus, $p>1$. Hence, $p-1>0$. Thus,
$p-1$ is a positive integer. Hence, $p-1\in\left\{  0,1,\ldots,p-1\right\}  $.
Also, $i\in\left\{  1,2,\ldots,p-1\right\}  $; thus, $i-1\in\left\{
0,1,\ldots,\left(  p-1\right)  -1\right\}  \subseteq\left\{  0,1,\ldots
,p-1\right\}  $ . Hence, Theorem \ref{thm.lucas} (applied to $a=N-1$, $b=K$,
$c=p-1$ and $d=i-1$) yields%
\[
\dbinom{\left(  N-1\right)  p+\left(  p-1\right)  }{Kp+\left(  i-1\right)
}\equiv\dbinom{N-1}{K}\dbinom{p-1}{i-1}\operatorname{mod}p.
\]
In view of $\left(  N-1\right)  p+\left(  p-1\right)  =Np-1$ and $Kp+\left(
i-1\right)  =Kp+i-1$, this rewrites as%
\[
\dbinom{Np-1}{Kp+i-1}\equiv\dbinom{N-1}{K}\dbinom{p-1}{i-1}\operatorname{mod}%
p.
\]
Thus, Lemma \ref{lem.congruence-times-p} (applied to $a=\dbinom{Np-1}{Kp+i-1}%
$, $b=\dbinom{N-1}{K}\dbinom{p-1}{i-1}$ and $c=p$) yields%
\begin{equation}
p\dbinom{Np-1}{Kp+i-1}\equiv p\dbinom{N-1}{K}\dbinom{p-1}{i-1}%
\operatorname{mod}p^{2}. \label{pf.thm.bailey.long.1}%
\end{equation}

Proposition \ref{prop.binom.absorb} (applied to $n=Np$ and $k=Kp+i$) yields%
\begin{align}
\left(  Kp+i\right)  \dbinom{Np}{Kp+i}  &  =N\underbrace{p\dbinom
{Np-1}{Kp+i-1}}_{\substack{\equiv p\dbinom{N-1}{K}\dbinom{p-1}{i-1}%
\operatorname{mod}p^{2}\\\text{(by (\ref{pf.thm.bailey.long.1}))}}}\nonumber\\
&  \equiv Np\dbinom{N-1}{K}\dbinom{p-1}{i-1}\operatorname{mod}p^{2}.
\label{pf.thm.bailey.long.a.0}%
\end{align}
In other words, $p^{2}\mid\left(  Kp+i\right)  \dbinom{Np}{Kp+i}%
-Np\dbinom{N-1}{K}\dbinom{p-1}{i-1}$. Hence,
\[
p\mid p^{2}\mid\left(  Kp+i\right)  \dbinom{Np}{Kp+i}-Np\dbinom{N-1}{K}%
\dbinom{p-1}{i-1}.
\]
In other words,%
\[
\left(  Kp+i\right)  \dbinom{Np}{Kp+i}\equiv Np\dbinom{N-1}{K}\dbinom
{p-1}{i-1}\equiv0=i\cdot0\operatorname{mod}p.
\]
Therefore,
\[
i\cdot0\equiv\underbrace{\left(  Kp+i\right)  }_{\equiv i\operatorname{mod}%
p}\dbinom{Np}{Kp+i}\equiv i\dbinom{Np}{Kp+i}\operatorname{mod}p.
\]
But $i$ is coprime to $p$. Hence, Lemma \ref{lem.2.1.p44.lem2} (applied to
$b=i$, $c=p$, $a=0$ and $a^{\prime}=\dbinom{Np}{Kp+i}$) shows that
$0\equiv\dbinom{Np}{Kp+i}\operatorname{mod}p$. In other words, $\dbinom
{Np}{Kp+i}\equiv0\operatorname{mod}p$. Thus, Lemma
\ref{lem.congruence-times-p} (applied to $a=\dbinom{Np}{Kp+i}$, $b=0$ and
$c=p$) yields%
\begin{equation}
p\dbinom{Np}{Kp+i}\equiv p\cdot0=0\operatorname{mod}p^{2}.
\label{pf.thm.bailey.long.a.1}%
\end{equation}
Now,%
\[
\left(  Kp+i\right)  \dbinom{Np}{Kp+i}=K\underbrace{p\dbinom{Np}{Kp+i}%
}_{\substack{\equiv0\operatorname{mod}p^{2}\\\text{(by
(\ref{pf.thm.bailey.long.a.1}))}}}+i\dbinom{Np}{Kp+i}\equiv i\dbinom{Np}%
{Kp+i}\operatorname{mod}p^{2}.
\]
Hence,%
\begin{align*}
i\dbinom{Np}{Kp+i}  &  \equiv\left(  Kp+i\right)  \dbinom{Np}{Kp+i}\equiv
Np\dbinom{N-1}{K}\dbinom{p-1}{i-1}\ \ \ \ \ \ \ \ \ \ \left(  \text{by
(\ref{pf.thm.bailey.long.a.0})}\right) \\
&  =N\dbinom{N-1}{K}\underbrace{p\dbinom{p-1}{i-1}}_{\substack{=i\dbinom{p}%
{i}\\\text{(by (\ref{pf.thm.bailey.long.absorb-p-i}))}}}=N\dbinom{N-1}%
{K}i\dbinom{p}{i}\\
&  =iN\dbinom{N-1}{K}\dbinom{p}{i}\operatorname{mod}p^{2}.
\end{align*}
But $i$ is coprime to $p^{2}$. Hence, Lemma \ref{lem.2.1.p44.lem2} (applied to
$b=i$, $c=p^{2}$, $a=\dbinom{Np}{Kp+i}$ and $a^{\prime}=N\dbinom{N-1}%
{K}\dbinom{p}{i}$) shows that%
\[
\dbinom{Np}{Kp+i}\equiv N\dbinom{N-1}{K}\dbinom{p}{i}\operatorname{mod}p^{2}.
\]
This proves Theorem \ref{thm.bailey} \textbf{(a)}.

\textbf{(b)} We have $i\in\left\{  1,2,\ldots,p-1\right\}  $ and thus
$p-i\in\left\{  1,2,\ldots,p-1\right\}  $. Also, from $i\in\left\{
1,2,\ldots,p-1\right\}  $, we obtain $i\leq p-1<p$, so that $p\geq i$.
Finally, $i\in\left\{  1,2,\ldots,p-1\right\}  \subseteq\mathbb{N}$. Thus,
Proposition \ref{prop.binom.symm} (applied to $m=p$ and $n=i$) yields%
\begin{equation}
\dbinom{p}{i}=\dbinom{p}{p-i}. \label{pf.thm.bailey.long.b.symm}%
\end{equation}

But $p-i\in\left\{  1,2,\ldots,p-1\right\}  $. Thus, Theorem \ref{thm.bailey}
\textbf{(a)} (applied to $K-1$ and $p-i$ instead of $K$ and $i$) yields%
\[
\dbinom{Np}{\left(  K-1\right)  p+\left(  p-i\right)  }\equiv N\dbinom
{N-1}{K-1}\underbrace{\dbinom{p}{p-i}}_{\substack{=\dbinom{p}{i}\\\text{(by
(\ref{pf.thm.bailey.long.b.symm}))}}}=N\dbinom{N-1}{K-1}\dbinom{p}%
{i}\operatorname{mod}p^{2}.
\]
In view of $\left(  K-1\right)  p+\left(  p-i\right)  =Kp-i$, this rewrites as%
\[
\dbinom{Np}{Kp-i}\equiv N\dbinom{N-1}{K-1}\dbinom{p}{i}\operatorname{mod}%
p^{2}.
\]
This proves Theorem \ref{thm.bailey} \textbf{(b)}.

\textbf{(c)} Proposition \ref{prop.binom.rec.neg} (applied to $m=N$ and $n=K$)
yields%
\begin{equation}
\dbinom{N}{K}=\dbinom{N-1}{K-1}+\dbinom{N-1}{K}=\dbinom{N-1}{K}+\dbinom
{N-1}{K-1}. \label{pf.thm.bailey.long.c.rec}%
\end{equation}

Now,
\begin{align*}
&  \underbrace{\dbinom{Np}{Kp+i}}_{\substack{\equiv N\dbinom{N-1}{K}\dbinom
{p}{i}\operatorname{mod}p^{2}\\\text{(by Theorem \ref{thm.bailey}
\textbf{(a)})}}}+\underbrace{\dbinom{Np}{Kp-i}}_{\substack{\equiv
N\dbinom{N-1}{K-1}\dbinom{p}{i}\operatorname{mod}p^{2}\\\text{(by Theorem
\ref{thm.bailey} \textbf{(b)})}}}\\
&  \equiv N\dbinom{N-1}{K}\dbinom{p}{i}+N\dbinom{N-1}{K-1}\dbinom{p}{i}\\
&  =N\underbrace{\left(  \dbinom{N-1}{K}+\dbinom{N-1}{K-1}\right)
}_{\substack{=\dbinom{N}{K}\\\text{(by (\ref{pf.thm.bailey.long.c.rec}))}%
}}\dbinom{p}{i}=N\dbinom{N}{K}\dbinom{p}{i}\operatorname{mod}p^{2}.
\end{align*}
This proves Theorem \ref{thm.bailey} \textbf{(c)}.
\end{proof}
\end{verlong}

\subsection{Two congruences for polynomials}

Now, we recall that $\mathbb{Z}\left[  X\right]  $ is the ring of all
polynomials in one indeterminate $X$ with integer coefficients.

\begin{lemma}
\label{lem.A3}Let $p$ be a prime. Let $c\in\mathbb{Z}$. Let $P\in
\mathbb{Z}\left[  X\right]  $ be a polynomial of degree $<2p-1$. Then,
$\sum\limits_{l=0}^{p-1}\left(  P\left(  cp+l\right)  -P\left(  l\right)
\right)  \equiv0\operatorname{mod}p^{2}$.
\end{lemma}

\begin{vershort}
\begin{proof}
[Proof of Lemma \ref{lem.A3}.]WLOG assume that $P=X^{k}$ for some
$k\in\left\{  0,1,\ldots,2p-2\right\}  $ (since the congruence we are proving
depends $\mathbb{Z}$-linearly on $P$). If $k=0$, then Lemma \ref{lem.A3} is
easily checked (because in this case, $P$ is constant). Thus, WLOG assume that
$k\neq0$. Hence, $k$ is a positive integer (since $k\in\mathbb{N}$). Thus,
$k-1\in\mathbb{N}$.

Each $l\in\left\{  0,1,\ldots,p-1\right\}  $ satisfies%
\begin{align*}
P\left(  cp+l\right)   &  =\left(  cp+l\right)  ^{k}%
\ \ \ \ \ \ \ \ \ \ \left(  \text{since }P=X^{k}\right) \\
&  =\sum_{i=0}^{k}\dbinom{k}{i}\left(  cp\right)  ^{i}l^{k-i}%
\ \ \ \ \ \ \ \ \ \ \left(  \text{by the binomial formula}\right) \\
&  =\underbrace{\left(  cp\right)  ^{0}l^{k-0}}_{=l^{k}}+k\underbrace{\left(
cp\right)  ^{1}}_{=cp}l^{k-1}+\sum_{i=2}^{k}\dbinom{k}{i}\underbrace{\left(
cp\right)  ^{i}}_{\substack{\equiv0\operatorname{mod}p^{2}\\\text{(since
}i\geq2\text{)}}}l^{k-i}\\
&  \equiv l^{k}+kcpl^{k-1}+\underbrace{\sum_{i=2}^{k}\dbinom{k}{i}0l^{k-i}%
}_{=0}=l^{k}+kcpl^{k-1}\operatorname{mod}p^{2}%
\end{align*}
and $P\left(  l\right)  =l^{k}$ (since $P=X^{k}$). Thus,%
\[
\sum\limits_{l=0}^{p-1}\left(  \underbrace{P\left(  cp+l\right)  }_{\equiv
l^{k}+kcpl^{k-1}\operatorname{mod}p^{2}}-\underbrace{P\left(  l\right)
}_{=l^{k}}\right)  \equiv\sum_{l=0}^{p-1}\underbrace{\left(  l^{k}%
+kcpl^{k-1}-l^{k}\right)  }_{=kcpl^{k-1}}=kcp\sum\limits_{l=0}^{p-1}%
l^{k-1}\operatorname{mod}p^{2}.
\]
The claim of Lemma \ref{lem.A3} now becomes obvious if $k=p$ (because if
$k=p$, then $kcp$ is already divisible by $p^{2}$); thus, we WLOG assume that
$k\neq p$. Hence, $k-1\neq p-1$.

If $k-1$ was a positive multiple of $p-1$, then we would have $k-1=p-1$ (since
$k\in\left\{  0,1,\ldots,2p-2\right\}  $), which would contradict $k-1\neq
p-1$. Hence, $k-1$ is not a positive multiple of $p-1$. Thus, Theorem
\ref{thm.powersums.p} (applied to $k-1$ instead of $k$) yields $\sum
_{l=0}^{p-1}l^{k-1}\equiv0\operatorname{mod}p$. Thus, $p\sum\limits_{l=0}%
^{p-1}l^{k-1}\equiv0\operatorname{mod}p^{2}$, so that%
\[
\sum\limits_{l=0}^{p-1}\left(  P\left(  cp+l\right)  -P\left(  l\right)
\right)  \equiv kc\underbrace{p\sum\limits_{l=0}^{p-1}l^{k-1}}_{\equiv
0\operatorname{mod}p^{2}}\equiv0\operatorname{mod}p^{2}.
\]
This proves Lemma \ref{lem.A3}.
\end{proof}
\end{vershort}

\begin{verlong}
Before we prove Lemma \ref{lem.A3}, let us state some basic facts.

\begin{proposition}
\label{prop.binom3}Let $k\in\mathbb{N}$. Let $a\in\mathbb{Z}$ and
$b\in\mathbb{Z}$. Then,%
\[
a^{k}-b^{k}=\left(  a-b\right)  \sum_{i=0}^{k-1}a^{i}b^{k-1-i}.
\]

\end{proposition}

\begin{proof}
[Proof of Proposition \ref{prop.binom3}.]We have $k\in\mathbb{N}$. Hence,
$0\in\left\{  0,1,\ldots,k\right\}  $. Thus, we can split off the addend for
$i=0$ from the sum $\sum_{i=0}^{k}a^{i}b^{k-i+1}$. We thus obtain
\[
\sum_{i=0}^{k}a^{i}b^{k-i}=\underbrace{a^{0}}_{=1}\underbrace{b^{k-0}}%
_{=b^{k}}+\sum_{i=1}^{k}a^{i}b^{k-i}=b^{k}+\sum_{i=1}^{k}a^{i}b^{k-i}.
\]
Hence,%
\begin{equation}
\sum_{i=1}^{k}a^{i}b^{k-i}=\sum_{i=0}^{k}a^{i}b^{k-i}-b^{k}.
\label{pf.prop.binom3.1}%
\end{equation}

Clearly, $k-k=0$, so that $b^{k-k}=b^{0}=1$.

We have $k\in\mathbb{N}$. Thus, $k\in\left\{  0,1,\ldots,k\right\}  $. Thus,
we can split off the addend for $i=k$ from the sum $\sum_{i=0}^{k}%
a^{i}b^{k-i+1}$. We thus obtain
\[
\sum_{i=0}^{k}a^{i}b^{k-i}=a^{k}\underbrace{b^{k-k}}_{=1}+\sum_{i=0}%
^{k-1}a^{i}b^{k-i}=a^{k}+\sum_{i=0}^{k-1}a^{i}b^{k-i}.
\]
Hence,%
\begin{equation}
\sum_{i=0}^{k-1}a^{i}b^{k-i}=\sum_{i=0}^{k}a^{i}b^{k-i}-a^{k}.
\label{pf.prop.binom3.2}%
\end{equation}

Now,%
\begin{align*}
\left(  a-b\right)  \sum_{i=0}^{k-1}a^{i}b^{k-1-i}  &  =\underbrace{a\sum
_{i=0}^{k-1}a^{i}b^{k-1-i}}_{=\sum_{i=0}^{k-1}aa^{i}b^{k-1-i}}%
-\underbrace{b\sum_{i=0}^{k-1}a^{i}b^{k-1-i}}_{=\sum_{i=0}^{k-1}%
a^{i}bb^{k-1-i}}=\sum_{i=0}^{k-1}\underbrace{aa^{i}}_{=a^{i+1}}b^{k-1-i}%
-\sum_{i=0}^{k-1}a^{i}\underbrace{bb^{k-1-i}}_{\substack{=b^{\left(
k-1-i\right)  +1}=b^{k-i}\\\text{(since }\left(  k-1-i\right)  +1=k-i\text{)}%
}}\\
&  =\underbrace{\sum_{i=0}^{k-1}a^{i+1}b^{k-1-i}}_{\substack{=\sum_{i=1}%
^{k}a^{\left(  i-1\right)  +1}b^{k-1-\left(  i+1\right)  }\\\text{(here, we
have substituted }i-1\text{ for }i\text{ in the sum)}}}-\sum_{i=0}^{k-1}%
a^{i}b^{k-i}\\
&  =\sum_{i=1}^{k}\underbrace{a^{\left(  i-1\right)  +1}}_{\substack{=a^{i}%
\\\text{(since }\left(  i-1\right)  +1=i\text{)}}}\underbrace{b^{k-1-\left(
i+1\right)  }}_{\substack{=b^{k-i}\\\text{(since }k-1-\left(  i+1\right)
=k-i\text{)}}}-\sum_{i=0}^{k-1}a^{i}b^{k-i}\\
&  =\underbrace{\sum_{i=1}^{k}a^{i}b^{k-i}}_{\substack{=\sum_{i=0}^{k}%
a^{i}b^{k-i}-b^{k}\\\text{(by (\ref{pf.prop.binom3.1}))}}}-\underbrace{\sum
_{i=0}^{k-1}a^{i}b^{k-i}}_{\substack{=\sum_{i=0}^{k}a^{i}b^{k-i}%
-a^{k}\\\text{(by (\ref{pf.prop.binom3.2}))}}}\\
&  =\left(  \sum_{i=0}^{k}a^{i}b^{k-i}-b^{k}\right)  -\left(  \sum_{i=0}%
^{k}a^{i}b^{k-i}-a^{k}\right)  =a^{k}-b^{k}.
\end{align*}
This proves Proposition \ref{prop.binom3}.
\end{proof}

\begin{lemma}
\label{lem.A3.Xk}Let $k$ be a positive integer. Let $p$ be a nonzero integer.
Let $c$ and $l$ be any integers. Then,%
\[
\left(  cp+l\right)  ^{k}-l^{k}\equiv kcpl^{k-1}\operatorname{mod}p^{2}.
\]

\end{lemma}

\begin{proof}
[Proof of Lemma \ref{lem.A3.Xk}.]Define an integer $a\in\mathbb{Z}$ by
$a=cp+l$. Define an integer $b\in\mathbb{Z}$ by $b=l$. Then,%
\[
\underbrace{a}_{=cp+l}-\underbrace{b}_{=l}=cp+l-l=cp\equiv0\operatorname{mod}%
p.
\]
In other words, $a\equiv b\operatorname{mod}p$. Now,%
\[
\sum_{i=0}^{k-1}\underbrace{a^{i}}_{\substack{\equiv b^{i}\operatorname{mod}%
p\\\text{(since }a\equiv b\operatorname{mod}p\text{)}}}b^{k-1-i}\equiv
\sum_{i=0}^{k-1}\underbrace{b^{i}b^{k-1-i}}_{\substack{=b^{i+\left(
k-1-i\right)  }=b^{k-1}\\\text{(since }i+\left(  k-1-i\right)  =k-1\text{)}%
}}=\sum_{i=0}^{k-1}b^{k-1}=kb^{k-1}\operatorname{mod}p.
\]
Hence, Lemma \ref{lem.congruence-times-p} (applied to $\sum_{i=0}^{k-1}%
a^{i}b^{k-1-i}$, $kb^{k-1}$ and $p$ instead of $a$, $b$ and $c$) yields%
\[
p\sum_{i=0}^{k-1}a^{i}b^{k-1-i}\equiv pkb^{k-1}\operatorname{mod}p^{2}.
\]

But Proposition \ref{prop.binom3} yields
\begin{align*}
a^{k}-b^{k}  &  =\underbrace{\left(  a-b\right)  }_{=cp}\sum_{i=0}^{k-1}%
a^{i}b^{k-1-i}=c\underbrace{p\sum_{i=0}^{k-1}a^{i}b^{k-1-i}}_{\equiv
pkb^{k-1}\operatorname{mod}p^{2}}\equiv cpkb^{k-1}\\
&  =kcpb^{k-1}\operatorname{mod}p^{2}.
\end{align*}
In view of $a=cp+l$ and $b=l$, this rewrites as $\left(  cp+l\right)
^{k}-l^{k}\equiv kcpl^{k-1}\operatorname{mod}p^{2}$. This proves Lemma
\ref{lem.A3.Xk}.
\end{proof}

\begin{lemma}
\label{lem.A3.sumXk}Let $p$ be a prime. Let $k\in\left\{  0,1,\ldots
,2p-2\right\}  $. Then,%
\[
\sum_{l=0}^{p-1}\left(  \left(  cp+l\right)  ^{k}-l^{k}\right)  \equiv
0\operatorname{mod}p^{2}.
\]

\end{lemma}

\begin{proof}
[Proof of Lemma \ref{lem.A3.sumXk}.]If $k=0$, then Lemma \ref{lem.A3.sumXk}
holds\footnote{\textit{Proof.} Assume that $k=0$. We must now prove that Lemma
\ref{lem.A3.sumXk} holds.
\par
We have%
\[
\sum_{l=0}^{p-1}\underbrace{\left(  \left(  cp+l\right)  ^{k}-l^{k}\right)
}_{\substack{=\left(  cp+l\right)  ^{0}-l^{0}\\\text{(since }k=0\text{)}%
}}=\sum_{l=0}^{p-1}\left(  \underbrace{\left(  cp+l\right)  ^{0}}%
_{=1}-\underbrace{l^{0}}_{=1}\right)  =\sum_{l=0}^{p-1}\underbrace{\left(
1-1\right)  }_{=0}=\sum_{l=0}^{p-1}0=0\equiv0\operatorname{mod}p^{2}.
\]
Thus, Lemma \ref{lem.A3.sumXk} holds. Qed.}. Hence, for the rest of this
proof, we can WLOG assume that we don't have $k=0$. Assume this.

We have $k\neq0$ (since we don't have $k=0$). Combining this with
$k\in\left\{  0,1,\ldots,2p-2\right\}  $, we obtain $k\in\left\{
0,1,\ldots,2p-2\right\}  \setminus\left\{  0\right\}  =\left\{  1,2,\ldots
,2p-2\right\}  \subseteq\left\{  1,2,3,\ldots\right\}  $. Hence, $k$ is a
positive integer. Thus, $k-1\in\mathbb{N}$. Also, $k\leq2p-2$ (since
$k\in\left\{  0,1,\ldots,2p-2\right\}  $).

If $k=p$, then Lemma \ref{lem.A3.sumXk} holds\footnote{\textit{Proof.} Assume
that $k=p$. We must now prove that Lemma \ref{lem.A3.sumXk} holds.
\par
We have%
\begin{align*}
\sum_{l=0}^{p-1}\underbrace{\left(  \left(  cp+l\right)  ^{k}-l^{k}\right)
}_{\substack{\equiv kcpl^{k-1}\operatorname{mod}p^{2}\\\text{(by Lemma
\ref{lem.A3.Xk})}}}  &  \equiv\sum_{l=0}^{p-1}kcpl^{k-1}=\underbrace{k}%
_{=p}cp\sum_{l=0}^{p-1}l^{k-1}=pcp\sum_{l=0}^{p-1}l^{k-1}\\
&  =p^{2}c\sum_{l=0}^{p-1}l^{k-1}\equiv0\operatorname{mod}p^{2}.
\end{align*}
Thus, Lemma \ref{lem.A3.sumXk} holds. Qed.}. Hence, for the rest of this
proof, we can WLOG assume that we don't have $k=p$. Assume this.

We have $k\neq p$ (since we don't have $k=p$). Hence, $k-1$ is not a positive
multiple of $p-1$\ \ \ \ \footnote{\textit{Proof.} Assume the contrary. Thus,
$k-1$ is a positive multiple of $p-1$. In other words, $k-1$ is a positive
integer and a multiple of $p-1$.
\par
There exists some $g\in\mathbb{Z}$ such that $k-1=g\left(  p-1\right)  $
(since $k-1$ is a multiple of $p-1$). Consider this $g$.
\par
We have $p>1$ (since $p$ is a prime), thus $p-1>0$. Hence, $p-1\geq1$ (since
$p-1$ is an integer). But $k-1<k\leq2p-2=2\left(  p-1\right)  $.
\par
If we had $g\geq2$, then we would have $\underbrace{g}_{\geq2}\left(
p-1\right)  \geq2\left(  p-1\right)  $ (because $p-1>0$), which would
contradict $g\left(  p-1\right)  =k-1<2\left(  p-1\right)  $. Hence, we cannot
have $g\geq2$. We thus have $g<2$. Thus, $g\leq2-1$ (since $g$ is an integer).
In other words, $g\leq1$.
\par
But $k-1>0$ (since $k-1$ is a positive integer). Hence, $g\left(  p-1\right)
=k-1>0$. We can divide this inequality by $p-1$ (since $p-1>0$), and obtain
$g>0$. Hence, $g\geq1$ (since $g$ is an integer). Combined with $g\leq1$, this
yields $g=1$. Thus, $k-1=\underbrace{g}_{=1}\left(  p-1\right)  =p-1$, so that
$k=p$. This contradicts $k\neq p$.
\par
This contradiction proves that our assumption was wrong, qed.}.

Hence, Theorem \ref{thm.powersums.p} (applied to $k-1$ instead of $k$) yields
\[
\sum_{l=0}^{p-1}l^{k-1}\equiv0\operatorname{mod}p.
\]
Thus, Lemma \ref{lem.congruence-times-p} (applied to $a=\sum_{l=0}%
^{p-1}l^{k-1}$, $b=0$ and $c=p$) yields
\[
p\sum_{l=0}^{p-1}l^{k-1}\equiv p\cdot0=0\operatorname{mod}p^{2}.
\]

But
\[
\sum_{l=0}^{p-1}\underbrace{\left(  \left(  cp+l\right)  ^{k}-l^{k}\right)
}_{\substack{\equiv kcpl^{k-1}\operatorname{mod}p^{2}\\\text{(by Lemma
\ref{lem.A3.Xk})}}}\equiv\sum_{l=0}^{p-1}kcpl^{k-1}=kc\underbrace{p\sum
_{l=0}^{p-1}l^{k-1}}_{\equiv0\operatorname{mod}p^{2}}\equiv0\operatorname{mod}%
p^{2}.
\]
This proves Lemma \ref{lem.A3.sumXk}.
\end{proof}

\begin{proof}
[Proof of Lemma \ref{lem.A3}.]We know that $P$ is a polynomial of degree
$<2p-1$. Hence, $\deg P<2p-1$. Since $\deg P$ and $2p-1$ are integers (unless
$\deg P=-\infty$, in which case this is also clear), we thus have $\deg
P\leq\left(  2p-1\right)  -1=2p-2$. Thus, $P$ is a polynomial in
$\mathbb{Z}\left[  X\right]  $ having degree $\leq2p-2$. Hence, we can write
$P$ in the form $P=\sum_{k=0}^{2p-2}a_{k}X^{k}$ for some integers $a_{0}%
,a_{1},\ldots,a_{2p-2}$. Consider these integers $a_{0},a_{1},\ldots,a_{2p-2}$.

Now, for each $m\in\mathbb{Z}$, we have%
\begin{equation}
P\left(  m\right)  =\sum_{k=0}^{2p-2}a_{k}m^{k}. \label{pf.lem.A3.Pm=}%
\end{equation}
(Indeed, this follows by substituting $m$ for $X$ into the equality
$P=\sum_{k=0}^{2p-2}a_{k}X^{k}$.)

Now,%
\begin{align*}
&  \sum\limits_{l=0}^{p-1}\left(  \underbrace{P\left(  cp+l\right)
}_{\substack{=\sum_{k=0}^{2p-2}a_{k}\left(  cp+l\right)  ^{k}\\\text{(by
(\ref{pf.lem.A3.Pm=}) (applied to }m=cp+l\text{))}}}-\underbrace{P\left(
l\right)  }_{\substack{=\sum_{k=0}^{2p-2}a_{k}l^{k}\\\text{(by
(\ref{pf.lem.A3.Pm=}) (applied to }m=l\text{))}}}\right) \\
&  =\sum\limits_{l=0}^{p-1}\underbrace{\left(  \sum_{k=0}^{2p-2}a_{k}\left(
cp+l\right)  ^{k}-\sum_{k=0}^{2p-2}a_{k}l^{k}\right)  }_{=\sum_{k=0}%
^{2p-2}a_{k}\left(  \left(  cp+l\right)  ^{k}-l^{k}\right)  }=\underbrace{\sum
\limits_{l=0}^{p-1}\sum_{k=0}^{2p-2}}_{=\sum_{k=0}^{2p-2}\sum\limits_{l=0}%
^{p-1}}a_{k}\left(  \left(  cp+l\right)  ^{k}-l^{k}\right) \\
&  =\sum_{k=0}^{2p-2}\underbrace{\sum\limits_{l=0}^{p-1}a_{k}\left(  \left(
cp+l\right)  ^{k}-l^{k}\right)  }_{=a_{k}\sum\limits_{l=0}^{p-1}\left(
\left(  cp+l\right)  ^{k}-l^{k}\right)  }=\sum_{k=0}^{2p-2}a_{k}%
\underbrace{\sum\limits_{l=0}^{p-1}\left(  \left(  cp+l\right)  ^{k}%
-l^{k}\right)  }_{\substack{\equiv0\operatorname{mod}p^{2}\\\text{(by Lemma
\ref{lem.A3.sumXk})}}}\\
&  \equiv\sum_{k=0}^{2p-2}a_{k}0=0\operatorname{mod}p^{2}.
\end{align*}
This proves Lemma \ref{lem.A3}.
\end{proof}
\end{verlong}

\begin{lemma}
\label{lem.Aab}Let $p$, $a$ and $b$ be three integers such that $a-b$ is
divisible by $p$. Then, $a^{2}-b^{2}\equiv2\left(  a-b\right)
b\operatorname{mod}p^{2}$.
\end{lemma}

\begin{vershort}
\begin{proof}
[Proof of Lemma \ref{lem.Aab}.]The difference $\left(  a^{2}-b^{2}\right)
-2\left(  a-b\right)  b=\left(  a-b\right)  ^{2}$ is divisible by $p^{2}$
(since $a-b$ is divisible by $p$). In other words, $a^{2}-b^{2}\equiv2\left(
a-b\right)  b\operatorname{mod}p^{2}$. Lemma \ref{lem.Aab} is proven.
\end{proof}
\end{vershort}

\begin{verlong}
\begin{proof}
[Proof of Lemma \ref{lem.Aab}.]There exists some $g\in\mathbb{Z}$ such that
$a-b=pg$ (since $a-b$ is divisible by $p$). Consider this $g$. Now,%
\[
\left(  a^{2}-b^{2}\right)  -2\left(  a-b\right)  b=\left(  \underbrace{a-b}%
_{=pg}\right)  ^{2}=\left(  pg\right)  ^{2}=p^{2}g^{2}\equiv
0\operatorname{mod}p^{2}.
\]
In other words, $a^{2}-b^{2}\equiv2\left(  a-b\right)  b\operatorname{mod}%
p^{2}$. This proves Lemma \ref{lem.Aab}.
\end{proof}
\end{verlong}

\begin{lemma}
\label{lem.A4}Let $p$ be an odd prime. Let $c\in\mathbb{Z}$. Let
$P\in\mathbb{Z}\left[  X\right]  $ be a polynomial of degree $\leq p-1$.
Then,
\[
\sum\limits_{l=0}^{p-1}\left(  P\left(  cp+l\right)  -P\left(  l\right)
\right)  P\left(  l\right)  \equiv0\operatorname{mod}p^{2}.
\]

\end{lemma}

\begin{vershort}
\begin{proof}
[Proof of Lemma \ref{lem.A4}.]Fix $l\in\mathbb{Z}$. We have $P\in
\mathbb{Z}\left[  X\right]  $. Thus, $P\left(  u\right)  -P\left(  v\right)  $
is divisible by $u-v$ whenever $u$ and $v$ are two integers\footnote{This is a
well-known fact. It can be proven as follows: WLOG assume that $P=X^{k}$ for
some $k\in\mathbb{N}$ (this is a valid assumption, since the claim is
$\mathbb{Z}$-linear in $P$); then, $P\left(  u\right)  -P\left(  v\right)
=u^{k}-v^{k}=\left(  u-v\right)  \sum_{i=0}^{k-1}u^{i}v^{k-i}$ is clearly
divisible by $u-v$.}. Applying this to $u=cp+l$ and $v=l$, we conclude that
$P\left(  cp+l\right)  -P\left(  l\right)  $ is divisible by $\left(
cp+l\right)  -l=cp$, and thus also divisible by $p$.

Hence, Lemma \ref{lem.Aab} (applied to $a=P\left(  cp+l\right)  $ and
$b=P\left(  l\right)  $) shows that%
\begin{equation}
\left(  P\left(  cp+l\right)  \right)  ^{2}-\left(  P\left(  l\right)
\right)  ^{2}\equiv2\left(  P\left(  cp+l\right)  -P\left(  l\right)  \right)
P\left(  l\right)  \operatorname{mod}p^{2}. \label{pf.lem.A4.first}%
\end{equation}

Now, forget that we fixed $l$. We thus have proven (\ref{pf.lem.A4.first}) for
each $l\in\mathbb{Z}$.

The polynomial $P$ has degree $\leq p-1$. Hence, the polynomial $P^{2}$ has
degree $\leq2\left(  p-1\right)  <2p-1$. Thus, Lemma \ref{lem.A3} (applied to
$P^{2}$ instead of $P$) shows that%
\[
\sum\limits_{l=0}^{p-1}\left(  P^{2}\left(  cp+l\right)  -P^{2}\left(
l\right)  \right)  \equiv0\operatorname{mod}p^{2}.
\]
Thus,%
\[
0\equiv\sum\limits_{l=0}^{p-1}\underbrace{\left(  P^{2}\left(  cp+l\right)
-P^{2}\left(  l\right)  \right)  }_{\substack{=\left(  P\left(  cp+l\right)
\right)  ^{2}-\left(  P\left(  l\right)  \right)  ^{2}\\\equiv2\left(
P\left(  cp+l\right)  -P\left(  l\right)  \right)  P\left(  l\right)
\operatorname{mod}p^{2}\\\text{(by (\ref{pf.lem.A4.first}))}}}\equiv
2\sum\limits_{l=0}^{p-1}\left(  P\left(  cp+l\right)  -P\left(  l\right)
\right)  P\left(  l\right)  \operatorname{mod}p^{2}.
\]
We can cancel $2$ from this congruence (since $p$ is odd), and conclude that
\[
0\equiv\sum\limits_{l=0}^{p-1}\left(  P\left(  cp+l\right)  -P\left(
l\right)  \right)  P\left(  l\right)  \operatorname{mod}p^{2}.
\]
This proves Lemma \ref{lem.A4}.
\end{proof}
\end{vershort}

\begin{verlong}
Before we prove Lemma \ref{lem.A4}, we need one further lemma:

\begin{lemma}
\label{lem.binom3cor}Let $P\in\mathbb{Z}\left[  X\right]  $. Let $u$ and $v$
be two integers. Then, the integer $P\left(  u\right)  -P\left(  v\right)  $
is divisible by $u-v$.
\end{lemma}

\begin{proof}
[Proof of Lemma \ref{lem.binom3cor}.]We know that $P\in\mathbb{Z}\left[
X\right]  $. In other words, $P$ is a polynomial with integer coefficients.
Thus, we can write $P$ in the form $P=\sum_{k=0}^{n}a_{k}X^{k}$ for some
$n\in\mathbb{N}$ and some integers $a_{0},a_{1},\ldots,a_{n}$. Consider this
$n$ and these $a_{0},a_{1},\ldots,a_{n}$.

We have $n\in\mathbb{N}$ and thus $0\in\left\{  0,1,\ldots,n\right\}  $.

Substituting $u$ for $X$ in the equality $P=\sum_{k=0}^{n}a_{k}X^{k}$, we find
$P\left(  u\right)  =\sum_{k=0}^{n}a_{k}u^{k}$.

Substituting $v$ for $X$ in the equality $P=\sum_{k=0}^{n}a_{k}X^{k}$, we find
$P\left(  v\right)  =\sum_{k=0}^{n}a_{k}v^{k}$.

Now,%
\begin{align}
&  \underbrace{P\left(  u\right)  }_{=\sum_{k=0}^{n}a_{k}u^{k}}%
-\underbrace{P\left(  v\right)  }_{=\sum_{k=0}^{n}a_{k}v^{k}}\nonumber\\
&  =\sum_{k=0}^{n}a_{k}u^{k}-\sum_{k=0}^{n}a_{k}v^{k}=\sum_{k=0}^{n}%
a_{k}\underbrace{\left(  u^{k}-v^{k}\right)  }_{\substack{=\left(  u-v\right)
\sum_{i=0}^{k-1}u^{i}v^{k-1-i}\\\text{(by Proposition \ref{prop.binom3}%
}\\\text{(applied to }a=u\text{ and }b=v\text{))}}}\nonumber\\
&  =\sum_{k=0}^{n}a_{k}\left(  u-v\right)  \sum_{i=0}^{k-1}u^{i}%
v^{k-1-i}=\left(  u-v\right)  \sum_{k=0}^{n}a_{k}\sum_{i=0}^{k-1}%
u^{i}v^{k-1-i}. \label{pf.lem.binom3cor.1}%
\end{align}
Now, define an integer $g\in\mathbb{Z}$ by $g=\sum_{k=0}^{n}a_{k}\sum
_{i=0}^{k-1}u^{i}v^{k-1-i}$. Then, (\ref{pf.lem.binom3cor.1}) becomes%
\[
P\left(  u\right)  -P\left(  v\right)  =\left(  u-v\right)  \underbrace{\sum
_{k=0}^{n}a_{k}\sum_{i=0}^{k-1}u^{i}v^{k-1-i}}_{=g}=\left(  u-v\right)  g.
\]
Hence, $P\left(  u\right)  -P\left(  v\right)  $ is divisible by $u-v$. This
proves Lemma \ref{lem.binom3cor}.
\end{proof}

\begin{proof}
[Proof of Lemma \ref{lem.A4}.]We have assumed that $p$ is odd. Thus, $p^{2}$
is odd (since the square of any odd integer is odd). In other words, $p^{2}$
is coprime to $2$. In other words, $2$ is coprime to $p^{2}$.

Fix $l\in\mathbb{Z}$. Applying Lemma \ref{lem.binom3cor} to $u=cp+l$ and
$v=l$, we conclude that $P\left(  cp+l\right)  -P\left(  l\right)  $ is
divisible by $\left(  cp+l\right)  -l$. In other words, $\left(  cp+l\right)
-l\mid P\left(  cp+l\right)  -P\left(  l\right)  $. In view of $\left(
cp+l\right)  -l=cp$, this rewrites as $cp\mid P\left(  cp+l\right)  -P\left(
l\right)  $. Now, $p\mid cp\mid P\left(  cp+l\right)  -P\left(  l\right)  $.
In other words, $P\left(  cp+l\right)  -P\left(  l\right)  $ is divisible by
$p$.

Hence, Lemma \ref{lem.Aab} (applied to $a=P\left(  cp+l\right)  $ and
$b=P\left(  l\right)  $) shows that%
\begin{equation}
\left(  P\left(  cp+l\right)  \right)  ^{2}-\left(  P\left(  l\right)
\right)  ^{2}\equiv2\left(  P\left(  cp+l\right)  -P\left(  l\right)  \right)
P\left(  l\right)  \operatorname{mod}p^{2}. \label{pf.lem.A4.long.first}%
\end{equation}

Now, forget that we fixed $l$. We thus have proven (\ref{pf.lem.A4.long.first}%
) for each $l\in\mathbb{Z}$.

The polynomial $P$ has degree $\leq p-1$. In other words, $\deg P\leq p-1$.
Hence,%
\begin{align*}
\deg\left(  \underbrace{P^{2}}_{=PP}\right)   &  =\deg\left(  PP\right)  =\deg
P+\deg P\\
&  =2\underbrace{\deg P}_{\leq p-1}\leq2\left(  p-1\right)  =2p-2<2p-1.
\end{align*}
In other words, the polynomial $P^{2}$ has degree $<2p-1$. Thus, Lemma
\ref{lem.A3} (applied to $P^{2}$ instead of $P$) shows that%
\[
\sum\limits_{l=0}^{p-1}\left(  P^{2}\left(  cp+l\right)  -P^{2}\left(
l\right)  \right)  \equiv0\operatorname{mod}p^{2}.
\]
Thus,%
\begin{align*}
0  &  \equiv\sum\limits_{l=0}^{p-1}\left(  \underbrace{P^{2}\left(
cp+l\right)  }_{=\left(  P\left(  cp+l\right)  \right)  ^{2}}%
-\underbrace{P^{2}\left(  l\right)  }_{=\left(  P\left(  l\right)  \right)
^{2}}\right)  =\sum\limits_{l=0}^{p-1}\underbrace{\left(  \left(  P\left(
cp+l\right)  \right)  ^{2}-\left(  P\left(  l\right)  \right)  ^{2}\right)
}_{\substack{\equiv2\left(  P\left(  cp+l\right)  -P\left(  l\right)  \right)
P\left(  l\right)  \operatorname{mod}p^{2}\\\text{(by
(\ref{pf.lem.A4.long.first}))}}}\\
&  \equiv\sum\limits_{l=0}^{p-1}2\left(  P\left(  cp+l\right)  -P\left(
l\right)  \right)  P\left(  l\right)  =2\sum\limits_{l=0}^{p-1}\left(
P\left(  cp+l\right)  -P\left(  l\right)  \right)  P\left(  l\right)
\operatorname{mod}p^{2}.
\end{align*}
Hence,%
\[
2\sum\limits_{l=0}^{p-1}\left(  P\left(  cp+l\right)  -P\left(  l\right)
\right)  P\left(  l\right)  \equiv0=2\cdot0\operatorname{mod}p^{2}.
\]
Hence, Lemma \ref{lem.2.1.p44.lem2} (applied to $2$, $p^{2}$, $\sum
\limits_{l=0}^{p-1}\left(  P\left(  cp+l\right)  -P\left(  l\right)  \right)
P\left(  l\right)  $ and $0$ instead of $b$, $c$, $a$ and $a^{\prime}$) shows
that%
\[
\sum\limits_{l=0}^{p-1}\left(  P\left(  cp+l\right)  -P\left(  l\right)
\right)  P\left(  l\right)  \equiv0\operatorname{mod}p^{2}%
\]
(since $2$ is coprime to $p^{2}$). This proves Lemma \ref{lem.A4}.
\end{proof}
\end{verlong}

\subsection{Proving Theorem \ref{thm.Z1}}

Now, let us prepare for the proofs of our results by showing several lemmas.

\begin{lemma}
\label{lem.A5}Let $p$ be an odd prime. Let $c\in\mathbb{Z}$. Let $k\in\left\{
0,1,\ldots,p-1\right\}  $. Then,
\[
\sum_{l=0}^{p-1}\left(  \dbinom{cp+l}{k}-\dbinom{l}{k}\right)  \dbinom{l}%
{k}\equiv0\operatorname{mod}p^{2}.
\]

\end{lemma}

\begin{vershort}
\begin{proof}
[Proof of Lemma \ref{lem.A5}.]Notice that $k!$ is coprime to $p$ (since $k\leq
p-1$), and thus $k!^{2}$ is coprime to $p^{2}$.

Define a polynomial $P\in\mathbb{Z}\left[  X\right]  $ by
\[
P=X\left(  X-1\right)  \cdots\left(  X-k+1\right)  .
\]
Then, $P$ has degree $k\leq p-1$. Thus, Lemma \ref{lem.A4} yields%
\[
\sum\limits_{l=0}^{p-1}\left(  P\left(  cp+l\right)  -P\left(  l\right)
\right)  P\left(  l\right)  \equiv0\operatorname{mod}p^{2}.
\]
Since each $n\in\mathbb{Z}$ satisfies $P\left(  n\right)  =n\left(
n-1\right)  \cdots\left(  n-k+1\right)  =k!\dbinom{n}{k}$, this rewrites as
\[
\sum\limits_{l=0}^{p-1}\left(  k!\dbinom{cp+l}{k}-k!\dbinom{l}{k}\right)
k!\dbinom{l}{k}\equiv0\operatorname{mod}p^{2}.
\]
We can cancel $k!^{2}$ from this congruence (since $k!^{2}$ is coprime to
$p^{2}$), and thus obtain%
\[
\sum_{l=0}^{p-1}\left(  \dbinom{cp+l}{k}-\dbinom{l}{k}\right)  \dbinom{l}%
{k}\equiv0\operatorname{mod}p^{2}.
\]
This proves Lemma \ref{lem.A5}.
\end{proof}
\end{vershort}

\begin{verlong}
\begin{proof}
[Proof of Lemma \ref{lem.A5}.]Lemma \ref{lem.2.1.p44.lem1} shows that $k!$ is
coprime to $p$. Hence, Lemma \ref{lem.2.1.p44.lem1p2} (applied to $i=k!$)
shows that $k!$ is coprime to $p^{2}$.

Define a polynomial $P\in\mathbb{Z}\left[  X\right]  $ by
\[
P=X\left(  X-1\right)  \cdots\left(  X-k+1\right)  .
\]
Then,
\[
\deg\underbrace{P}_{\substack{=X\left(  X-1\right)  \cdots\left(
X-k+1\right)  \\=\prod_{j=0}^{k-1}\left(  X-j\right)  }}=\deg\left(
\prod_{j=0}^{k-1}\left(  X-j\right)  \right)  =\sum_{j=0}^{k-1}%
\underbrace{\deg\left(  X-j\right)  }_{=1}=\sum_{j=0}^{k-1}1=k\leq p-1
\]
(since $k\in\left\{  0,1,\ldots,p-1\right\}  $). In other words, the
polynomial $P$ has degree $\leq p-1$. Thus, Lemma \ref{lem.A4} yields%
\begin{equation}
\sum\limits_{l=0}^{p-1}\left(  P\left(  cp+l\right)  -P\left(  l\right)
\right)  P\left(  l\right)  \equiv0\operatorname{mod}p^{2}.
\label{pf.lem.A5.long.1}%
\end{equation}
But each $n\in\mathbb{Z}$ satisfies
\begin{equation}
P\left(  n\right)  =k!\dbinom{n}{k} \label{pf.lem.A5.long.2}%
\end{equation}
\footnote{\textit{Proof of (\ref{pf.lem.A5.long.2}):} Let $n\in\mathbb{Z}$.
Recall that $P=X\left(  X-1\right)  \cdots\left(  X-k+1\right)  $.
Substituting $n$ for $X$ in this equality, we obtain $P\left(  n\right)
=n\left(  n-1\right)  \cdots\left(  n-k+1\right)  $.
\par
But $k\in\left\{  0,1,\ldots,p-1\right\}  \subseteq\mathbb{N}$. Hence, the
definition of the binomial coefficient $\dbinom{n}{k}$ shows that $\dbinom
{n}{k}=\dfrac{n\left(  n-1\right)  \cdots\left(  n-k+1\right)  }{k!}$. Thus,
$k!\dbinom{n}{k}=n\left(  n-1\right)  \cdots\left(  n-k+1\right)  $. Comparing
this with $P\left(  n\right)  =n\left(  n-1\right)  \cdots\left(
n-k+1\right)  $, we obtain $P\left(  n\right)  =k!\dbinom{n}{k}$. This proves
(\ref{pf.lem.A5.long.2}).}. Hence,%
\begin{align*}
&  \sum\limits_{l=0}^{p-1}\left(  \underbrace{P\left(  cp+l\right)
}_{\substack{=k!\dbinom{cp+l}{k}\\\text{(by (\ref{pf.lem.A5.long.2}%
)}\\\text{(applied to }n=cp+l\text{))}}}-\underbrace{P\left(  l\right)
}_{\substack{=k!\dbinom{l}{k}\\\text{(by (\ref{pf.lem.A5.long.2}%
)}\\\text{(applied to }n=l\text{))}}}\right)  \underbrace{P\left(  l\right)
}_{\substack{=k!\dbinom{l}{k}\\\text{(by (\ref{pf.lem.A5.long.2}%
)}\\\text{(applied to }n=l\text{))}}}\\
&  =\sum\limits_{l=0}^{p-1}\underbrace{\left(  k!\dbinom{cp+l}{k}-k!\dbinom
{l}{k}\right)  }_{=k!\left(  \dbinom{cp+l}{k}-\dbinom{l}{k}\right)  }%
k!\dbinom{l}{k}=\sum\limits_{l=0}^{p-1}k!\left(  \dbinom{cp+l}{k}-\dbinom
{l}{k}\right)  k!\dbinom{l}{k}\\
&  =k!k!\sum\limits_{l=0}^{p-1}\left(  \dbinom{cp+l}{k}-\dbinom{l}{k}\right)
\dbinom{l}{k}.
\end{align*}
Therefore,%
\begin{align*}
&  k!k!\sum\limits_{l=0}^{p-1}\left(  \dbinom{cp+l}{k}-\dbinom{l}{k}\right)
\dbinom{l}{k}\\
&  =\sum\limits_{l=0}^{p-1}\left(  P\left(  cp+l\right)  -P\left(  l\right)
\right)  P\left(  l\right)  \equiv0\\
&  =k!k!0\operatorname{mod}p^{2}\ \ \ \ \ \ \ \ \ \ \left(  \text{by
(\ref{pf.lem.A5.long.1})}\right)  .
\end{align*}
Hence, Lemma \ref{lem.2.1.p44.lem2} (applied to $k!$, $p^{2}$, $k!\sum
\limits_{l=0}^{p-1}\left(  \dbinom{cp+l}{k}-\dbinom{l}{k}\right)  \dbinom
{l}{k}$ and $k!0$ instead of $b$, $c$, $a$ and $a^{\prime}$) shows that%
\[
k!\sum\limits_{l=0}^{p-1}\left(  \dbinom{cp+l}{k}-\dbinom{l}{k}\right)
\dbinom{l}{k}\equiv k!0\operatorname{mod}p^{2}.
\]
Thus, Lemma \ref{lem.2.1.p44.lem2} (applied to $k!$, $p^{2}$, $\sum
\limits_{l=0}^{p-1}\left(  \dbinom{cp+l}{k}-\dbinom{l}{k}\right)  \dbinom
{l}{k}$ and $0$ instead of $b$, $c$, $a$ and $a^{\prime}$) shows that%
\[
\sum_{l=0}^{p-1}\left(  \dbinom{cp+l}{k}-\dbinom{l}{k}\right)  \dbinom{l}%
{k}\equiv0\operatorname{mod}p^{2}.
\]
This proves Lemma \ref{lem.A5}.
\end{proof}
\end{verlong}

\begin{lemma}
\label{lem.A6}Let $p$ be an odd prime. Let $c\in\mathbb{Z}$. Then,%
\[
\sum_{l=0}^{p-1}\left(  \dbinom{cp+2l}{l}-\dbinom{2l}{l}\right)
\equiv0\operatorname{mod}p^{2}.
\]

\end{lemma}

\begin{vershort}
\begin{proof}
[Proof of Lemma \ref{lem.A6}.]For each $l\in\left\{  0,1,\ldots,p-1\right\}
$, we have%
\begin{align*}
&  \underbrace{\dbinom{cp+2l}{l}}_{\substack{=\sum\limits_{k=0}^{p-1}%
\dbinom{cp+l}{k}\dbinom{l}{k}\\\text{(by Lemma \ref{lem.vandermonde.3})}%
}}-\underbrace{\dbinom{2l}{l}}_{\substack{=\sum\limits_{k=0}^{p-1}\dbinom
{l}{k}\dbinom{l}{k}\\\text{(by Lemma \ref{lem.vandermonde.3},}\\\text{applied
to }0\text{ instead of }c\text{)}}}\\
&  =\sum\limits_{k=0}^{p-1}\dbinom{cp+l}{k}\dbinom{l}{k}-\sum\limits_{k=0}%
^{p-1}\dbinom{l}{k}\dbinom{l}{k}=\sum\limits_{k=0}^{p-1}\left(  \dbinom
{cp+l}{k}-\dbinom{l}{k}\right)  \dbinom{l}{k}.
\end{align*}
Summing these equalities over all $l\in\left\{  0,1,\ldots,p-1\right\}  $, we
find%
\begin{align*}
\sum_{l=0}^{p-1}\left(  \dbinom{cp+2l}{l}-\dbinom{2l}{l}\right)   &
=\sum_{l=0}^{p-1}\sum\limits_{k=0}^{p-1}\left(  \dbinom{cp+l}{k}-\dbinom{l}%
{k}\right)  \dbinom{l}{k}\\
&  =\sum\limits_{k=0}^{p-1}\underbrace{\sum_{l=0}^{p-1}\left(  \dbinom
{cp+l}{k}-\dbinom{l}{k}\right)  \dbinom{l}{k}}_{\substack{\equiv
0\operatorname{mod}p^{2}\\\text{(by Lemma \ref{lem.A5})}}}\equiv\sum
_{k=0}^{p-1}0=0\operatorname{mod}p^{2}.
\end{align*}
This proves Lemma \ref{lem.A6}.
\end{proof}
\end{vershort}

\begin{verlong}
\begin{proof}
[Proof of Lemma \ref{lem.A6}.]Lemma \ref{lem.vandermonde.3} (applied to $0$
instead of $c$) yields%
\[
\dbinom{0p+2l}{l}=\sum\limits_{k=0}^{p-1}\underbrace{\dbinom{0p+l}{k}%
}_{\substack{=\dbinom{l}{k}\\\text{(since }0p+l=l\text{)}}}\dbinom{l}{k}%
=\sum\limits_{k=0}^{p-1}\dbinom{l}{k}\dbinom{l}{k}.
\]
In view of $0p+2l=2l$, this rewrites as
\begin{equation}
\dbinom{2l}{l}=\sum\limits_{k=0}^{p-1}\dbinom{l}{k}\dbinom{l}{k}.
\label{pf.lem.A6.long.1}%
\end{equation}

We have%
\begin{align*}
&  \sum_{l=0}^{p-1}\left(  \underbrace{\dbinom{cp+2l}{l}}_{\substack{=\sum
\limits_{k=0}^{p-1}\dbinom{cp+l}{k}\dbinom{l}{k}\\\text{(by Lemma
\ref{lem.vandermonde.3})}}}-\underbrace{\dbinom{2l}{l}}_{\substack{=\sum
\limits_{k=0}^{p-1}\dbinom{l}{k}\dbinom{l}{k}\\\text{(by
(\ref{pf.lem.A6.long.1}))}}}\right) \\
&  =\sum_{l=0}^{p-1}\underbrace{\left(  \sum\limits_{k=0}^{p-1}\dbinom
{cp+l}{k}\dbinom{l}{k}-\sum\limits_{k=0}^{p-1}\dbinom{l}{k}\dbinom{l}%
{k}\right)  }_{=\sum\limits_{k=0}^{p-1}\left(  \dbinom{cp+l}{k}-\dbinom{l}%
{k}\right)  \dbinom{l}{k}}\\
&  =\underbrace{\sum_{l=0}^{p-1}\sum\limits_{k=0}^{p-1}}_{=\sum\limits_{k=0}%
^{p-1}\sum_{l=0}^{p-1}}\left(  \dbinom{cp+l}{k}-\dbinom{l}{k}\right)
\dbinom{l}{k}\\
&  =\sum\limits_{k=0}^{p-1}\underbrace{\sum_{l=0}^{p-1}\left(  \dbinom
{cp+l}{k}-\dbinom{l}{k}\right)  \dbinom{l}{k}}_{\substack{\equiv
0\operatorname{mod}p^{2}\\\text{(by Lemma \ref{lem.A5})}}}\equiv\sum
_{k=0}^{p-1}0=0\operatorname{mod}p^{2}.
\end{align*}
This proves Lemma \ref{lem.A6}.
\end{proof}
\end{verlong}

\begin{vershort}
\begin{proof}
[Proof of Theorem \ref{thm.Z1}.]Lemma \ref{lem.A6} (applied to $c=-1$) yields%
\[
\sum_{l=0}^{p-1}\left(  \dbinom{-p+2l}{l}-\dbinom{2l}{l}\right)
\equiv0\operatorname{mod}p^{2}.
\]
Thus,%
\[
0\equiv\sum_{l=0}^{p-1}\left(  \dbinom{-p+2l}{l}-\dbinom{2l}{l}\right)
=\sum_{l=0}^{p-1}\dbinom{-p+2l}{l}-\sum_{l=0}^{p-1}\dbinom{2l}{l}%
\operatorname{mod}p^{2},
\]
so that%
\begin{equation}
\sum_{l=0}^{p-1}\dbinom{2l}{l}\equiv\sum_{l=0}^{p-1}\dbinom{-p+2l}%
{l}\operatorname{mod}p^{2}. \label{pf.thm.Z1.1}%
\end{equation}
Now,%
\begin{align*}
\sum_{n=0}^{p-1}\dbinom{2n}{n}  &  =\sum_{l=0}^{p-1}\dbinom{2l}{l}\equiv
\sum_{l=0}^{p-1}\underbrace{\dbinom{-p+2l}{l}}_{\substack{=\left(  -1\right)
^{l}\dbinom{l-\left(  -p+2l\right)  -1}{l}\\\text{(by Proposition
\ref{prop.binom.upper-neg})}}}\ \ \ \ \ \ \ \ \ \ \left(  \text{by
(\ref{pf.thm.Z1.1})}\right) \\
&  =\sum_{l=0}^{p-1}\left(  -1\right)  ^{l}\underbrace{\dbinom{l-\left(
-p+2l\right)  -1}{l}}_{=\dbinom{p-1-l}{l}}=\sum_{l=0}^{p-1}\left(  -1\right)
^{l}\dbinom{p-1-l}{l}\\
&  =\sum_{i=0}^{p-1}\left(  -1\right)  ^{i}\dbinom{p-1-i}{i}%
=\underbrace{\left(  -1\right)  ^{p}}_{\substack{=-1\\\text{(since }p\text{ is
odd)}}}\cdot%
\begin{cases}
0, & \text{if }p\equiv0\operatorname{mod}3;\\
-1, & \text{if }p\equiv1\operatorname{mod}3;\\
1, & \text{if }p\equiv2\operatorname{mod}3
\end{cases}
\\
&  \ \ \ \ \ \ \ \ \ \ \left(  \text{by Corollary \ref{cor.ps2.2.S.-1},
applied to }n=p\right) \\
&  =-%
\begin{cases}
0, & \text{if }p\equiv0\operatorname{mod}3;\\
-1, & \text{if }p\equiv1\operatorname{mod}3;\\
1, & \text{if }p\equiv2\operatorname{mod}3
\end{cases}
=%
\begin{cases}
0, & \text{if }p\equiv0\operatorname{mod}3;\\
1, & \text{if }p\equiv1\operatorname{mod}3;\\
-1, & \text{if }p\equiv2\operatorname{mod}3
\end{cases}
=\eta_{p}\operatorname{mod}p^{2}.
\end{align*}
This proves Theorem \ref{thm.Z1}.
\end{proof}
\end{vershort}

\begin{verlong}
\begin{proof}
[Proof of Theorem \ref{thm.Z1}.]Lemma \ref{lem.A6} (applied to $c=-1$) yields%
\[
\sum_{l=0}^{p-1}\left(  \dbinom{\left(  -1\right)  p+2l}{l}-\dbinom{2l}%
{l}\right)  \equiv0\operatorname{mod}p^{2}.
\]
In view of $\left(  -1\right)  p=-p$, this rewrites as%
\[
\sum_{l=0}^{p-1}\left(  \dbinom{-p+2l}{l}-\dbinom{2l}{l}\right)
\equiv0\operatorname{mod}p^{2}.
\]
Thus,%
\[
0\equiv\sum_{l=0}^{p-1}\left(  \dbinom{-p+2l}{l}-\dbinom{2l}{l}\right)
=\sum_{l=0}^{p-1}\dbinom{-p+2l}{l}-\sum_{l=0}^{p-1}\dbinom{2l}{l}%
\operatorname{mod}p^{2},
\]
so that%
\begin{equation}
\sum_{l=0}^{p-1}\dbinom{2l}{l}\equiv\sum_{l=0}^{p-1}\dbinom{-p+2l}%
{l}\operatorname{mod}p^{2}. \label{pf.thm.Z1.long.1}%
\end{equation}
Now,%
\begin{align*}
\sum_{n=0}^{p-1}\dbinom{2n}{n}  &  =\sum_{l=0}^{p-1}\dbinom{2l}{l}%
\ \ \ \ \ \ \ \ \ \ \left(  \text{here, we have renamed the summation index
}n\text{ as }l\right) \\
&  \equiv\sum_{l=0}^{p-1}\underbrace{\dbinom{-p+2l}{l}}_{\substack{=\left(
-1\right)  ^{l}\dbinom{l-\left(  -p+2l\right)  -1}{l}\\\text{(by Proposition
\ref{prop.binom.upper-neg}}\\\text{(applied to }m=-p+2l\text{ and
}n=l\text{))}}}\ \ \ \ \ \ \ \ \ \ \left(  \text{by (\ref{pf.thm.Z1.long.1}%
)}\right) \\
&  =\sum_{l=0}^{p-1}\left(  -1\right)  ^{l}\underbrace{\dbinom{l-\left(
-p+2l\right)  -1}{l}}_{\substack{=\dbinom{p-1-l}{l}\\\text{(since }l-\left(
-p+2l\right)  -1=p-1-l\text{)}}}=\sum_{l=0}^{p-1}\left(  -1\right)
^{l}\dbinom{p-1-l}{l}\\
&  =\sum_{i=0}^{p-1}\left(  -1\right)  ^{i}\dbinom{p-1-i}{i}%
\ \ \ \ \ \ \ \ \ \ \left(
\begin{array}
[c]{c}%
\text{here, we have renamed the}\\
\text{summation index }l\text{ as }i
\end{array}
\right) \\
&  =\underbrace{\left(  -1\right)  ^{p}}_{\substack{=-1\\\text{(since }p\text{
is odd)}}}\cdot%
\begin{cases}
0, & \text{if }p\equiv0\operatorname{mod}3;\\
-1, & \text{if }p\equiv1\operatorname{mod}3;\\
1, & \text{if }p\equiv2\operatorname{mod}3
\end{cases}
\\
&  \ \ \ \ \ \ \ \ \ \ \left(  \text{by Corollary \ref{cor.ps2.2.S.-1}
(applied to }n=p\text{)}\right) \\
&  =\left(  -1\right)  \cdot%
\begin{cases}
0, & \text{if }p\equiv0\operatorname{mod}3;\\
-1, & \text{if }p\equiv1\operatorname{mod}3;\\
1, & \text{if }p\equiv2\operatorname{mod}3
\end{cases}
=-%
\begin{cases}
0, & \text{if }p\equiv0\operatorname{mod}3;\\
-1, & \text{if }p\equiv1\operatorname{mod}3;\\
1, & \text{if }p\equiv2\operatorname{mod}3
\end{cases}
\\
&  =%
\begin{cases}
-0, & \text{if }p\equiv0\operatorname{mod}3;\\
-\left(  -1\right)  , & \text{if }p\equiv1\operatorname{mod}3;\\
-1, & \text{if }p\equiv2\operatorname{mod}3
\end{cases}
=%
\begin{cases}
0, & \text{if }p\equiv0\operatorname{mod}3;\\
1, & \text{if }p\equiv1\operatorname{mod}3;\\
-1, & \text{if }p\equiv2\operatorname{mod}3
\end{cases}
\\
&  \ \ \ \ \ \ \ \ \ \ \left(  \text{since }-0=0\text{ and }-\left(
-1\right)  =1\right) \\
&  =\eta_{p}\operatorname{mod}p^{2}\ \ \ \ \ \ \ \ \ \ \left(  \text{since
}\eta_{p}=%
\begin{cases}
0, & \text{if }p\equiv0\operatorname{mod}3;\\
1, & \text{if }p\equiv1\operatorname{mod}3;\\
-1, & \text{if }p\equiv2\operatorname{mod}3
\end{cases}
\right)  .
\end{align*}
This proves Theorem \ref{thm.Z1}.
\end{proof}
\end{verlong}

\subsection{Proving Theorem \ref{thm.Z2}}

\begin{lemma}
\label{lem.A1}Let $N\in\mathbb{Z}$ and $K\in\mathbb{N}$. Let $p$ be a prime.
Let $l\in\left\{  0,1,\ldots,p-1\right\}  $. Then,%
\[
\dbinom{Np+2l}{Kp+l}-\dbinom{N}{K}\dbinom{2l}{l}\equiv N\dbinom{N}{K}\left(
\dbinom{p+2l}{l}-\dbinom{2l}{l}\right)  \operatorname{mod}p^{2}.
\]

\end{lemma}

\begin{vershort}
\begin{proof}
[Proof of Lemma \ref{lem.A1}.]Theorem \ref{thm.p2cong} yields $\dbinom{Np}%
{Kp}\equiv\dbinom{N}{K}\operatorname{mod}p^{2}$.

Lemma \ref{lem.A1a} (applied to $u=Np$ and $w=Kp$) yields%
\begin{align*}
\dbinom{Np+2l}{Kp+l}  &  =\underbrace{\dbinom{Np}{Kp}}_{\equiv\dbinom{N}%
{K}\operatorname{mod}p^{2}}\dbinom{2l}{l}+\sum_{i=1}^{l}\underbrace{\left(
\dbinom{Np}{Kp+i}+\dbinom{Np}{Kp-i}\right)  }_{\substack{\equiv N\dbinom{N}%
{K}\dbinom{p}{i}\operatorname{mod}p^{2}\\\text{(by Theorem \ref{thm.bailey}
\textbf{(c)})}}}\dbinom{2l}{l-i}\\
&  \equiv\dbinom{N}{K}\dbinom{2l}{l}+\sum\limits_{i=1}^{l}N\dbinom{N}%
{K}\dbinom{p}{i}\dbinom{2l}{l-i}\\
&  =\dbinom{N}{K}\dbinom{2l}{l}+N\dbinom{N}{K}\underbrace{\sum\limits_{i=1}%
^{l}\dbinom{p}{i}\dbinom{2l}{l-i}}_{\substack{=\dbinom{p+2l}{l}-\dbinom{2l}%
{l}\\\text{(by Lemma \ref{lem.vandermonde.4})}}}\\
&  =\dbinom{N}{K}\dbinom{2l}{l}+N\dbinom{N}{K}\left(  \dbinom{p+2l}{l}%
-\dbinom{2l}{l}\right)  \operatorname{mod}p^{2}.
\end{align*}
Subtracting $\dbinom{N}{K}\dbinom{2l}{l}$ from both sides of this congruence,
we obtain%
\[
\dbinom{Np+2l}{Kp+l}-\dbinom{N}{K}\dbinom{2l}{l}\equiv N\dbinom{N}{K}\left(
\dbinom{p+2l}{l}-\dbinom{2l}{l}\right)  \operatorname{mod}p^{2}.
\]
This proves Lemma \ref{lem.A1}.
\end{proof}
\end{vershort}

\begin{verlong}
\begin{proof}
[Proof of Lemma \ref{lem.A1}.]Theorem \ref{thm.p2cong} (applied to $a=N$ and
$b=K$) yields $\dbinom{Np}{Kp}\equiv\dbinom{N}{K}\operatorname{mod}p^{2}$.

Lemma \ref{lem.A1a} (applied to $u=Np$ and $w=Kp$) yields%
\begin{align*}
\dbinom{Np+2l}{Kp+l}  &  =\underbrace{\dbinom{Np}{Kp}}_{\equiv\dbinom{N}%
{K}\operatorname{mod}p^{2}}\dbinom{2l}{l}+\sum_{i=1}^{l}\underbrace{\left(
\dbinom{Np}{Kp+i}+\dbinom{Np}{Kp-i}\right)  }_{\substack{\equiv N\dbinom{N}%
{K}\dbinom{p}{i}\operatorname{mod}p^{2}\\\text{(by Theorem \ref{thm.bailey}
\textbf{(c)})}}}\dbinom{2l}{l-i}\\
&  \equiv\dbinom{N}{K}\dbinom{2l}{l}+\underbrace{\sum\limits_{i=1}^{l}%
N\dbinom{N}{K}\dbinom{p}{i}\dbinom{2l}{l-i}}_{=N\dbinom{N}{K}\sum
\limits_{i=1}^{l}\dbinom{p}{i}\dbinom{2l}{l-i}}\\
&  =\dbinom{N}{K}\dbinom{2l}{l}+N\dbinom{N}{K}\underbrace{\sum\limits_{i=1}%
^{l}\dbinom{p}{i}\dbinom{2l}{l-i}}_{\substack{=\dbinom{p+2l}{l}-\dbinom{2l}%
{l}\\\text{(by Lemma \ref{lem.vandermonde.4})}}}\\
&  =\dbinom{N}{K}\dbinom{2l}{l}+N\dbinom{N}{K}\left(  \dbinom{p+2l}{l}%
-\dbinom{2l}{l}\right)  \operatorname{mod}p^{2}.
\end{align*}
Subtracting $\dbinom{N}{K}\dbinom{2l}{l}$ from both sides of this congruence,
we obtain%
\[
\dbinom{Np+2l}{Kp+l}-\dbinom{N}{K}\dbinom{2l}{l}\equiv N\dbinom{N}{K}\left(
\dbinom{p+2l}{l}-\dbinom{2l}{l}\right)  \operatorname{mod}p^{2}.
\]
This proves Lemma \ref{lem.A1}.
\end{proof}
\end{verlong}

\begin{lemma}
\label{lem.C1}Let $p$ be an odd prime. Let $N\in\mathbb{Z}$ and $K\in
\mathbb{N}$. Then,%
\[
\sum_{l=0}^{p-1}\dbinom{Np+2l}{Kp+l}\equiv\dbinom{N}{K}\eta_{p}%
\operatorname{mod}p^{2}.
\]

\end{lemma}

\begin{vershort}
\begin{proof}
[Proof of Lemma \ref{lem.C1}.]For any $l\in\left\{  0,1,\ldots,p-1\right\}  $,
we have%
\[
\dbinom{Np+2l}{Kp+l}\equiv\dbinom{N}{K}\dbinom{2l}{l}+N\dbinom{N}{K}\left(
\dbinom{p+2l}{l}-\dbinom{2l}{l}\right)  \operatorname{mod}p^{2}%
\]
(by Lemma \ref{lem.A1}). Summing these congruences over all $l\in\left\{
0,1,\ldots,p-1\right\}  $, we find%
\begin{align*}
\sum_{l=0}^{p-1}\dbinom{Np+2l}{Kp+l}  &  \equiv\sum_{l=0}^{p-1}\left(
\dbinom{N}{K}\dbinom{2l}{l}+N\dbinom{N}{K}\left(  \dbinom{p+2l}{l}-\dbinom
{2l}{l}\right)  \right) \\
&  =\dbinom{N}{K}\sum_{l=0}^{p-1}\dbinom{2l}{l}+N\dbinom{N}{K}\underbrace{\sum
_{l=0}^{p-1}\left(  \dbinom{p+2l}{l}-\dbinom{2l}{l}\right)  }%
_{\substack{\equiv0\operatorname{mod}p^{2}\\\text{(by Lemma \ref{lem.A6},
applied to }c=1\text{)}}}\\
&  \equiv\dbinom{N}{K}\underbrace{\sum_{l=0}^{p-1}\dbinom{2l}{l}%
}_{\substack{=\sum_{n=0}^{p-1}\dbinom{2n}{n}\equiv\eta_{p}\operatorname{mod}%
p^{2}\\\text{(by Theorem \ref{thm.Z1})}}}\equiv\dbinom{N}{K}\eta
_{p}\operatorname{mod}p^{2}.
\end{align*}
This proves Lemma \ref{lem.C1}.
\end{proof}
\end{vershort}

\begin{verlong}
\begin{proof}
[Proof of Lemma \ref{lem.C1}.]We have%
\begin{align}
\sum_{l=0}^{p-1}\dbinom{2l}{l}  &  =\sum_{n=0}^{p-1}\dbinom{2n}{n}%
\ \ \ \ \ \ \ \ \ \ \left(
\begin{array}
[c]{c}%
\text{here, we have renamed the}\\
\text{summation index }l\text{ as }n
\end{array}
\right) \nonumber\\
&  \equiv\eta_{p}\operatorname{mod}p^{2}\ \ \ \ \ \ \ \ \ \ \left(  \text{by
Theorem \ref{thm.Z1}}\right)  . \label{pf.lem.C1.long.1}%
\end{align}

Lemma \ref{lem.A6} (applied to $c=1$) yields%
\[
\sum_{l=0}^{p-1}\left(  \dbinom{1p+2l}{l}-\dbinom{2l}{l}\right)
\equiv0\operatorname{mod}p^{2}.
\]
In view of $1p=p$, this rewrites as
\begin{equation}
\sum_{l=0}^{p-1}\left(  \dbinom{p+2l}{l}-\dbinom{2l}{l}\right)  \equiv
0\operatorname{mod}p^{2}. \label{pf.lem.C1.long.0}%
\end{equation}

Let $l\in\left\{  0,1,\ldots,p-1\right\}  $. Then, Lemma \ref{lem.A1} yields%
\[
\dbinom{Np+2l}{Kp+l}-\dbinom{N}{K}\dbinom{2l}{l}\equiv N\dbinom{N}{K}\left(
\dbinom{p+2l}{l}-\dbinom{2l}{l}\right)  \operatorname{mod}p^{2}.
\]
Adding $\dbinom{N}{K}\dbinom{2l}{l}$ to both sides of this congruence, we
obtain%
\begin{equation}
\dbinom{Np+2l}{Kp+l}\equiv\dbinom{N}{K}\dbinom{2l}{l}+N\dbinom{N}{K}\left(
\dbinom{p+2l}{l}-\dbinom{2l}{l}\right)  \operatorname{mod}p^{2}.
\label{pf.lem.C1.long.2}%
\end{equation}

Now, forget that we fixed $l$. We thus have proven (\ref{pf.lem.C1.long.2})
for each $l\in\left\{  0,1,\ldots,p-1\right\}  $.

Summing up the congruences (\ref{pf.lem.C1.long.2}) over all $l\in\left\{
0,1,\ldots,p-1\right\}  $, we find%
\begin{align*}
\sum_{l=0}^{p-1}\dbinom{Np+2l}{Kp+l}  &  \equiv\sum_{l=0}^{p-1}\left(
\dbinom{N}{K}\dbinom{2l}{l}+N\dbinom{N}{K}\left(  \dbinom{p+2l}{l}-\dbinom
{2l}{l}\right)  \right) \\
&  =\dbinom{N}{K}\underbrace{\sum_{l=0}^{p-1}\dbinom{2l}{l}}_{\substack{\equiv
\eta_{p}\operatorname{mod}p^{2}\\\text{(by (\ref{pf.lem.C1.long.1}))}%
}}+N\dbinom{N}{K}\underbrace{\sum_{l=0}^{p-1}\left(  \dbinom{p+2l}{l}%
-\dbinom{2l}{l}\right)  }_{\substack{\equiv0\operatorname{mod}p^{2}\\\text{(by
(\ref{pf.lem.C1.long.0}))}}}\\
&  \equiv\dbinom{N}{K}\eta_{p}+\underbrace{N\dbinom{N}{K}0}_{=0}=\dbinom{N}%
{K}\eta_{p}\operatorname{mod}p^{2}.
\end{align*}
This proves Lemma \ref{lem.C1}.
\end{proof}
\end{verlong}

\begin{vershort}
\begin{proof}
[Proof of Theorem \ref{thm.Z2}.]The map%
\begin{align*}
\left\{  0,1,\ldots,p-1\right\}  \times\left\{  0,1,\ldots,r-1\right\}   &
\rightarrow\left\{  0,1,\ldots,rp-1\right\}  ,\\
\left(  l,K\right)   &  \mapsto Kp+l
\end{align*}
is a bijection (since each element of $\left\{  0,1,\ldots,rp-1\right\}  $ can
be uniquely divided by $p$ with remainder, and said remainder will belong to
$\left\{  0,1,\ldots,r-1\right\}  $). Thus, we can substitute $Kp+l$ for $n$
in the sum $\sum_{n=0}^{rp-1}\dbinom{2n}{n}$. This sum thus rewrites as
follows:%
\begin{align*}
\sum_{n=0}^{rp-1}\dbinom{2n}{n}  &  =\underbrace{\sum_{\left(  l,K\right)
\in\left\{  0,1,\ldots,p-1\right\}  \times\left\{  0,1,\ldots,r-1\right\}  }%
}_{=\sum_{K=0}^{r-1}\sum_{l=0}^{p-1}}\underbrace{\dbinom{2\left(  Kp+l\right)
}{Kp+l}}_{=\dbinom{2Kp+2l}{Kp+l}}=\sum_{K=0}^{r-1}\underbrace{\sum_{l=0}%
^{p-1}\dbinom{2Kp+2l}{Kp+l}}_{\substack{\equiv\dbinom{2K}{K}\eta
_{p}\operatorname{mod}p^{2}\\\text{(by Lemma \ref{lem.C1},}\\\text{applied to
}N=2K\text{)}}}\\
&  \equiv\underbrace{\sum_{K=0}^{r-1}\dbinom{2K}{K}}_{=\sum_{n=0}^{r-1}%
\dbinom{2n}{n}=\alpha_{r}}\eta_{p}=\alpha_{r}\eta_{p}=\eta_{p}\alpha
_{r}\operatorname{mod}p^{2}.
\end{align*}
This proves Theorem \ref{thm.Z2}.
\end{proof}
\end{vershort}

\begin{verlong}
Next, we state a well-known combinatorial lemma:

\begin{lemma}
\label{lem.pigeonhole.inj}Let $U$ and $V$ be two finite sets such that
$\left\vert U\right\vert =\left\vert V\right\vert $. Let $f:U\rightarrow V$ be
any injective map. Then, the map $f$ is bijective.
\end{lemma}

Another simple lemma is the following:

\begin{lemma}
\label{lem.redressys}Let $p$ be a positive integer. Let $u$ and $v$ be two
elements of $\left\{  0,1,\ldots,p-1\right\}  $ such that $u\equiv
v\operatorname{mod}p$. Then, $u=v$.
\end{lemma}

\begin{proof}
[Proof of Lemma \ref{lem.redressys}.]We can WLOG assume that $u\leq v$ (since
otherwise, we can simply switch $u$ with $v$). Assume this.

We must prove that $u=v$. Indeed, assume the contrary. Thus, $u\neq v$.
Combining this with $u\leq v$, we obtain $u<v$. Thus, $v-u$ is a positive integer.

But $v\equiv u\operatorname{mod}p$ (since $u\equiv v\operatorname{mod}p$).
Thus, $v-u$ is divisible by $p$. Hence, $v-u\geq p$ (because $v-u$ and $p$ are
positive integers).

From $u\in\left\{  0,1,\ldots,p-1\right\}  $, we obtain $u\geq0$. From
$v\in\left\{  0,1,\ldots,p-1\right\}  $, we obtain $v\leq p-1$. Thus,
$\underbrace{v}_{\leq p-1}-\underbrace{u}_{\geq0}\leq p-1-0=p-1<p$. This
contradicts $v-u\geq p$. This contradiction shows that our assumption was
wrong. Hence, $u=v$ is proven. Thus, we have proven Lemma \ref{lem.redressys}.
\end{proof}

We can use this to prove the following basic fact:

\begin{lemma}
\label{lem.pr-bij}Let $p$ be a positive integer. Let $r\in\mathbb{N}$. Then,
the map%
\begin{align*}
\left\{  0,1,\ldots,p-1\right\}  \times\left\{  0,1,\ldots,r-1\right\}   &
\rightarrow\left\{  0,1,\ldots,rp-1\right\}  ,\\
\left(  l,K\right)   &  \mapsto Kp+l
\end{align*}
is well-defined and is a bijection.
\end{lemma}

\begin{proof}
[Proof of Lemma \ref{lem.pr-bij}.]We have $p\in\mathbb{N}$ (since $p$ is a
positive integer).

For each $\left(  l,K\right)  \in\left\{  0,1,\ldots,p-1\right\}
\times\left\{  0,1,\ldots,r-1\right\}  $, we have $Kp+l\in\left\{
0,1,\ldots,rp-1\right\}  $\ \ \ \ \footnote{\textit{Proof.} Let $\left(
l,K\right)  \in\left\{  0,1,\ldots,p-1\right\}  \times\left\{  0,1,\ldots
,r-1\right\}  $. We must prove that $Kp+l\in\left\{  0,1,\ldots,rp-1\right\}
$.
\par
We have $\left(  l,K\right)  \in\left\{  0,1,\ldots,p-1\right\}
\times\left\{  0,1,\ldots,r-1\right\}  $. In other words, $l\in\left\{
0,1,\ldots,p-1\right\}  $ and $K\in\left\{  0,1,\ldots,r-1\right\}  $.
\par
From $K\in\left\{  0,1,\ldots,r-1\right\}  \subseteq\mathbb{N}$ and
$p\in\mathbb{N}$ (since $p$ is a positive integer), we conclude that
$Kp\in\mathbb{N}$. Combining this with $l\in\left\{  0,1,\ldots,p-1\right\}
\subseteq\mathbb{N}$, we obtain $Kp+l\in\mathbb{N}$.
\par
From $K\in\left\{  0,1,\ldots,r-1\right\}  $, we obtain $K\leq r-1$. We can
multiply this inequality with $p$ (since $p\geq0$); thus, we obtain
$Kp\leq\left(  r-1\right)  p$. From $l\in\left\{  0,1,\ldots,p-1\right\}  $,
we obtain $l\leq p-1$. Thus, $\underbrace{Kp}_{\leq\left(  r-1\right)
p}+\underbrace{l}_{\leq p-1}\leq\left(  r-1\right)  p+\left(  p-1\right)
=rp-1$. Hence, $Kp+l\in\left\{  0,1,\ldots,rp-1\right\}  $ (since
$Kp+l\in\mathbb{N}$). Qed.}. Hence, the map%
\begin{align*}
\left\{  0,1,\ldots,p-1\right\}  \times\left\{  0,1,\ldots,r-1\right\}   &
\rightarrow\left\{  0,1,\ldots,rp-1\right\}  ,\\
\left(  l,K\right)   &  \mapsto Kp+l
\end{align*}
is well-defined. Denote this map by $A$.

Now, let $u_{1}$ and $u_{2}$ be two elements of $\left\{  0,1,\ldots
,p-1\right\}  \times\left\{  0,1,\ldots,r-1\right\}  $ satisfying $A\left(
u_{1}\right)  =A\left(  u_{2}\right)  $. We shall prove that $u_{1}=u_{2}$.

We have $u_{1}\in\left\{  0,1,\ldots,p-1\right\}  \times\left\{
0,1,\ldots,r-1\right\}  $. Hence, we can write $u_{1}$ in the form
$u_{1}=\left(  l_{1},K_{1}\right)  $ for some $l_{1}\in\left\{  0,1,\ldots
,p-1\right\}  $ and $K_{1}\in\left\{  0,1,\ldots,r-1\right\}  $. Consider
these $l_{1}$ and $K_{1}$.

We have $u_{2}\in\left\{  0,1,\ldots,p-1\right\}  \times\left\{
0,1,\ldots,r-1\right\}  $. Hence, we can write $u_{2}$ in the form
$u_{2}=\left(  l_{2},K_{2}\right)  $ for some $l_{2}\in\left\{  0,1,\ldots
,p-1\right\}  $ and $K_{2}\in\left\{  0,1,\ldots,r-1\right\}  $. Consider
these $l_{2}$ and $K_{2}$.

From $u_{1}=\left(  l_{1},K_{1}\right)  $, we obtain%
\begin{align*}
A\left(  u_{1}\right)   &  =A\left(  l_{1},K_{1}\right)  =K_{1}p+l_{1}%
\ \ \ \ \ \ \ \ \ \ \left(  \text{by the definition of }A\right) \\
&  \equiv l_{1}\operatorname{mod}p.
\end{align*}
From $u_{2}=\left(  l_{2},K_{2}\right)  $, we obtain%
\begin{align*}
A\left(  u_{2}\right)   &  =A\left(  l_{2},K_{2}\right)  =K_{2}p+l_{2}%
\ \ \ \ \ \ \ \ \ \ \left(  \text{by the definition of }A\right) \\
&  \equiv l_{2}\operatorname{mod}p.
\end{align*}
But $A\left(  u_{1}\right)  \equiv l_{1}\operatorname{mod}p$, so that
$l_{1}\equiv A\left(  u_{1}\right)  =A\left(  u_{2}\right)  \equiv
l_{2}\operatorname{mod}p$. Hence, Lemma \ref{lem.redressys} (applied to
$l_{1}$ and $l_{2}$ instead of $u$ and $v$) yields $l_{1}=l_{2}$. Now,
$A\left(  u_{1}\right)  =K_{1}p+\underbrace{l_{1}}_{=l_{2}}=K_{1}p+l_{2}$, so
that%
\[
K_{1}p+l_{2}=A\left(  u_{1}\right)  =A\left(  u_{2}\right)  =K_{2}p+l_{2}.
\]
Subtracting $l_{2}$ from both sides of this equality, we obtain $K_{1}%
p=K_{2}p$. Dividing both sides of this equality by $p$ (this is allowed, since
$p$ is a positive integer), we obtain $K_{1}=K_{2}$.

Now, $u_{1}=\left(  \underbrace{l_{1}}_{=l_{2}},\underbrace{K_{1}}_{=K_{2}%
}\right)  =\left(  l_{2},K_{2}\right)  =u_{2}$.

Now, forget that we fixed $u_{1}$ and $u_{2}$. We thus have shown that if
$u_{1}$ and $u_{2}$ are two elements of $\left\{  0,1,\ldots,p-1\right\}
\times\left\{  0,1,\ldots,r-1\right\}  $ satisfying $A\left(  u_{1}\right)
=A\left(  u_{2}\right)  $, then $u_{1}=u_{2}$.

In other words, the map $A$ is injective.

But
\[
\left\vert \left\{  0,1,\ldots,p-1\right\}  \times\left\{  0,1,\ldots
,r-1\right\}  \right\vert =\underbrace{\left\vert \left\{  0,1,\ldots
,p-1\right\}  \right\vert }_{=p}\cdot\underbrace{\left\vert \left\{
0,1,\ldots,r-1\right\}  \right\vert }_{=r}=p\cdot r=rp.
\]
Thus, in particular, $\left\{  0,1,\ldots,p-1\right\}  \times\left\{
0,1,\ldots,r-1\right\}  $ is a finite set. Also, $\left\{  0,1,\ldots
,rp-1\right\}  $ is a finite set with size $\left\vert \left\{  0,1,\ldots
,rp-1\right\}  \right\vert =rp$. Hence,%
\[
\left\vert \left\{  0,1,\ldots,p-1\right\}  \times\left\{  0,1,\ldots
,r-1\right\}  \right\vert =rp=\left\vert \left\{  0,1,\ldots,rp-1\right\}
\right\vert .
\]
Hence, Lemma \ref{lem.pigeonhole.inj} (applied to $U=\left\{  0,1,\ldots
,p-1\right\}  \times\left\{  0,1,\ldots,r-1\right\}  $, $V=\left\{
0,1,\ldots,rp-1\right\}  $ and $f=A$) shows that the map $A$ is bijective. In
other words, $A$ is a bijection. In other words, the map%
\begin{align*}
\left\{  0,1,\ldots,p-1\right\}  \times\left\{  0,1,\ldots,r-1\right\}   &
\rightarrow\left\{  0,1,\ldots,rp-1\right\}  ,\\
\left(  l,K\right)   &  \mapsto Kp+l
\end{align*}
is a bijection\footnote{since $A$ is the map
\begin{align*}
\left\{  0,1,\ldots,p-1\right\}  \times\left\{  0,1,\ldots,r-1\right\}   &
\rightarrow\left\{  0,1,\ldots,rp-1\right\}  ,\\
\left(  l,K\right)   &  \mapsto Kp+l
\end{align*}
}. This completes the proof of Lemma \ref{lem.pr-bij}.
\end{proof}

\begin{proof}
[Proof of Theorem \ref{thm.Z2}.]We have%
\begin{equation}
\alpha_{r}=\sum_{n=0}^{r-1}\dbinom{2n}{n}=\sum_{K=0}^{r-1}\dbinom{2K}{K}
\label{pf.thm.Z2.long.alphar=}%
\end{equation}
(here, we have renamed the summation index $n$ as $K$).

Lemma \ref{lem.pr-bij} shows that The map%
\begin{align*}
\left\{  0,1,\ldots,p-1\right\}  \times\left\{  0,1,\ldots,r-1\right\}   &
\rightarrow\left\{  0,1,\ldots,rp-1\right\}  ,\\
\left(  l,K\right)   &  \mapsto Kp+l
\end{align*}
is a bijection. Thus, we can substitute $Kp+l$ for $n$ in the sum $\sum
_{n\in\left\{  0,1,\ldots,rp-1\right\}  }\dbinom{2n}{n}$. We thus obtain%
\[
\sum_{n\in\left\{  0,1,\ldots,rp-1\right\}  }\dbinom{2n}{n}=\sum_{\left(
l,K\right)  \in\left\{  0,1,\ldots,p-1\right\}  \times\left\{  0,1,\ldots
,r-1\right\}  }\dbinom{2\left(  Kp+l\right)  }{Kp+l}.
\]

Now,%
\begin{align*}
&  \underbrace{\sum_{n=0}^{rp-1}}_{=\sum_{n\in\left\{  0,1,\ldots
,rp-1\right\}  }}\dbinom{2n}{n}\\
&  =\sum_{n\in\left\{  0,1,\ldots,rp-1\right\}  }\dbinom{2n}{n}%
=\underbrace{\sum_{\left(  l,K\right)  \in\left\{  0,1,\ldots,p-1\right\}
\times\left\{  0,1,\ldots,r-1\right\}  }}_{\substack{=\sum_{l\in\left\{
0,1,\ldots,p-1\right\}  }\sum_{K\in\left\{  0,1,\ldots,r-1\right\}  }%
\\=\sum_{K\in\left\{  0,1,\ldots,r-1\right\}  }\sum_{l\in\left\{
0,1,\ldots,p-1\right\}  }}}\underbrace{\dbinom{2\left(  Kp+l\right)  }{Kp+l}%
}_{\substack{=\dbinom{2Kp+2l}{Kp+l}\\\text{(since }2\left(  Kp+l\right)
=2Kp+2l\text{)}}}\\
&  =\underbrace{\sum_{K\in\left\{  0,1,\ldots,r-1\right\}  }}_{=\sum
_{K=0}^{r-1}}\underbrace{\sum_{l\in\left\{  0,1,\ldots,p-1\right\}  }}%
_{=\sum_{l=0}^{p-1}}\dbinom{2Kp+2l}{Kp+l}=\sum_{K=0}^{r-1}\underbrace{\sum
_{l=0}^{p-1}\dbinom{2Kp+2l}{Kp+l}}_{\substack{\equiv\dbinom{2K}{K}\eta
_{p}\operatorname{mod}p^{2}\\\text{(by Lemma \ref{lem.C1} (applied to
}N=2K\text{))}}}\\
&  \equiv\sum_{K=0}^{r-1}\dbinom{2K}{K}\eta_{p}=\eta_{p}\underbrace{\sum
_{K=0}^{r-1}\dbinom{2K}{K}}_{\substack{=\alpha_{r}\\\text{(by
(\ref{pf.thm.Z2.long.alphar=}))}}}=\eta_{p}\alpha_{r}\operatorname{mod}p^{2}.
\end{align*}
This proves Theorem \ref{thm.Z2}.
\end{proof}
\end{verlong}

\subsection{Proving Theorem \ref{thm.Z3}}

\begin{lemma}
\label{lem.A7}Let $p$ be an odd prime. Let $N\in\mathbb{Z}$ and $K\in
\mathbb{N}$. Then,%
\[
\sum_{l=0}^{p-1}\sum_{m=0}^{l}\left(  \dbinom{Np+l}{Kp+m}-\dbinom{N}{K}%
\dbinom{l}{m}\right)  \dbinom{l}{m}\equiv0\operatorname{mod}p^{2}.
\]

\end{lemma}

\begin{vershort}
\begin{proof}
[Proof of Lemma \ref{lem.A7}.]We have%
\begin{align*}
&  \sum_{l=0}^{p-1}\sum_{m=0}^{l}\left(  \dbinom{Np+l}{Kp+m}-\dbinom{N}%
{K}\dbinom{l}{m}\right)  \dbinom{l}{m}\\
&  =\sum_{l=0}^{p-1}\underbrace{\sum_{m=0}^{l}\dbinom{Np+l}{Kp+m}\dbinom{l}%
{m}}_{\substack{=\dbinom{Np+2l}{Kp+l}\\\text{(by Corollary
\ref{cor.vandermonde.2},}\\\text{applied to }u=Np+l\text{ and }w=Kp\text{)}%
}}-\dbinom{N}{K}\sum_{l=0}^{p-1}\underbrace{\sum_{m=0}^{l}\dbinom{l}{m}%
\dbinom{l}{m}}_{\substack{=\dbinom{2l}{l}\\\text{(by Corollary
\ref{cor.vandermonde.2},}\\\text{applied to }u=l\text{ and }w=0\text{)}}}\\
&  =\sum_{l=0}^{p-1}\dbinom{Np+2l}{Kp+l}-\dbinom{N}{K}\sum_{l=0}^{p-1}%
\dbinom{2l}{l}=\sum_{l=0}^{p-1}\underbrace{\left(  \dbinom{Np+2l}%
{Kp+l}-\dbinom{N}{K}\dbinom{2l}{l}\right)  }_{\substack{\equiv N\dbinom{N}%
{K}\left(  \dbinom{p+2l}{l}-\dbinom{2l}{l}\right)  \operatorname{mod}%
p^{2}\\\text{(by Lemma \ref{lem.A1})}}}\\
&  \equiv N\dbinom{N}{K}\underbrace{\sum_{l=0}^{p-1}\left(  \dbinom{p+2l}%
{l}-\dbinom{2l}{l}\right)  }_{\substack{\equiv0\operatorname{mod}%
p^{2}\\\text{(by Lemma \ref{lem.A6}, applied to }c=1\text{)}}}\equiv
0\operatorname{mod}p^{2}.
\end{align*}
This proves Lemma \ref{lem.A7}.
\end{proof}
\end{vershort}

\begin{verlong}
\begin{proof}
[Proof of Lemma \ref{lem.A7}.]Lemma \ref{lem.A6} (applied to $c=1$) yields%
\[
\sum_{l=0}^{p-1}\left(  \dbinom{1p+2l}{l}-\dbinom{2l}{l}\right)
\equiv0\operatorname{mod}p^{2}.
\]
In view of $1p=p$, this rewrites as
\begin{equation}
\sum_{l=0}^{p-1}\left(  \dbinom{p+2l}{l}-\dbinom{2l}{l}\right)  \equiv
0\operatorname{mod}p^{2}. \label{pf.lem.A7.long.0}%
\end{equation}

Let $l\in\left\{  0,1,\ldots,p-1\right\}  $. Then, Corollary
\ref{cor.vandermonde.2} (applied to $u=Np+l$ and $w=Kp$) yields%
\begin{equation}
\sum_{m=0}^{l}\dbinom{Np+l}{Kp+m}\dbinom{l}{m}=\dbinom{Np+l+l}{Kp+l}%
=\dbinom{Np+2l}{Kp+l} \label{pf.lem.A7.long.1}%
\end{equation}
(since $l+l=2l$). But Corollary \ref{cor.vandermonde.2} (applied to $u=l$ and
$w=0$) yields%
\[
\sum_{m=0}^{l}\dbinom{l}{0+m}\dbinom{l}{m}=\dbinom{l+l}{0+l}=\dbinom{2l}{l}%
\]
(since $l+l=2l$ and $0+l=l$). Thus,%
\begin{equation}
\dbinom{2l}{l}=\sum_{m=0}^{l}\underbrace{\dbinom{l}{0+m}}_{\substack{=\dbinom
{l}{m}\\\text{(since }0+m=m\text{)}}}\dbinom{l}{m}=\sum_{m=0}^{l}\dbinom{l}%
{m}\dbinom{l}{m}. \label{pf.lem.A7.long.2}%
\end{equation}

Now,%
\begin{align}
&  \sum_{m=0}^{l}\underbrace{\left(  \dbinom{Np+l}{Kp+m}-\dbinom{N}{K}%
\dbinom{l}{m}\right)  \dbinom{l}{m}}_{=\dbinom{Np+l}{Kp+m}\dbinom{l}%
{m}-\dbinom{N}{K}\dbinom{l}{m}\dbinom{l}{m}}\nonumber\\
&  =\sum_{m=0}^{l}\left(  \dbinom{Np+l}{Kp+m}\dbinom{l}{m}-\dbinom{N}%
{K}\dbinom{l}{m}\dbinom{l}{m}\right) \nonumber\\
&  =\underbrace{\sum_{m=0}^{l}\dbinom{Np+l}{Kp+m}\dbinom{l}{m}}%
_{\substack{_{\substack{=\dbinom{Np+2l}{Kp+l}}}\\\text{(by
(\ref{pf.lem.A7.long.1}))}}}-\dbinom{N}{K}\underbrace{\sum_{m=0}^{l}\dbinom
{l}{m}\dbinom{l}{m}}_{\substack{=\dbinom{2l}{l}\\\text{(by
(\ref{pf.lem.A7.long.2}))}}}\nonumber\\
&  =\dbinom{Np+2l}{Kp+l}-\dbinom{N}{K}\dbinom{2l}{l}\nonumber\\
&  \equiv N\dbinom{N}{K}\left(  \dbinom{p+2l}{l}-\dbinom{2l}{l}\right)
\operatorname{mod}p^{2} \label{pf.lem.A7.long.4}%
\end{align}
(by Lemma \ref{lem.A1}).

Now, forget that we fixed $l$. We thus have proven (\ref{pf.lem.A7.long.4})
for each $l\in\left\{  0,1,\ldots,p-1\right\}  $.

We have%
\begin{align*}
&  \sum_{l=0}^{p-1}\underbrace{\sum_{m=0}^{l}\left(  \dbinom{Np+l}%
{Kp+m}-\dbinom{N}{K}\dbinom{l}{m}\right)  \dbinom{l}{m}}_{\substack{\equiv
N\dbinom{N}{K}\left(  \dbinom{p+2l}{l}-\dbinom{2l}{l}\right)
\operatorname{mod}p^{2}\\\text{(by (\ref{pf.lem.A7.long.4}))}}}\\
&  \equiv\sum_{l=0}^{p-1}N\dbinom{N}{K}\left(  \dbinom{p+2l}{l}-\dbinom{2l}%
{l}\right) \\
&  =N\dbinom{N}{K}\underbrace{\sum_{l=0}^{p-1}\left(  \dbinom{p+2l}{l}%
-\dbinom{2l}{l}\right)  }_{\substack{\equiv0\operatorname{mod}p^{2}\\\text{(by
(\ref{pf.lem.A7.long.0}))}}}\equiv N\dbinom{N}{K}0=0\operatorname{mod}p^{2}.
\end{align*}
This proves Lemma \ref{lem.A7}.
\end{proof}
\end{verlong}

\begin{lemma}
\label{lem.A8}Let $p$ be an odd prime. Let $N\in\mathbb{Z}$ and $K\in
\mathbb{N}$. Then,%
\begin{equation}
\sum_{l=0}^{p-1}\sum_{m=0}^{l}\dbinom{Np+l}{Kp+m}^{2}\equiv\dbinom{N}{K}%
^{2}\eta_{p}\operatorname{mod}p^{2}.\nonumber
\end{equation}

\end{lemma}

\begin{vershort}
\begin{proof}
[Proof of Lemma \ref{lem.A8}.]Fix $l\in\left\{  0,1,\ldots,p-1\right\}  $ and
$m\in\left\{  0,1,\ldots,p-1\right\}  $. Then, Theorem \ref{thm.lucas}
(applied to $a=N$, $b=K$, $c=l$ and $d=m$) yields that $\dbinom{Np+l}%
{Kp+m}\equiv\dbinom{N}{K}\dbinom{l}{m}\operatorname{mod}p$. In other words,
$\dbinom{Np+l}{Kp+m}-\dbinom{N}{K}\dbinom{l}{m}$ is divisible by $p$. Hence,
Lemma \ref{lem.Aab} (applied to $a=\dbinom{Np+l}{Kp+m}$ and $b=\dbinom{N}%
{K}\dbinom{l}{m}$) shows that%
\begin{align}
&  \dbinom{Np+l}{Kp+m}^{2}-\left(  \dbinom{N}{K}\dbinom{l}{m}\right)
^{2}\nonumber\\
&  \equiv2\left(  \dbinom{Np+l}{Kp+m}-\dbinom{N}{K}\dbinom{l}{m}\right)
\dbinom{N}{K}\dbinom{l}{m}\operatorname{mod}p^{2}. \label{pf.lem.A8.1}%
\end{align}

Now, forget that we fixed $l$ and $m$. We thus have proven (\ref{pf.lem.A8.1})
for all $l\in\left\{  0,1,\ldots,p-1\right\}  $ and $m\in\left\{
0,1,\ldots,p-1\right\}  $. Now,%
\begin{align*}
&  \sum_{l=0}^{p-1}\sum_{m=0}^{l}\dbinom{Np+l}{Kp+m}^{2}-\sum_{l=0}^{p-1}%
\sum_{m=0}^{l}\left(  \dbinom{N}{K}\dbinom{l}{m}\right)  ^{2}\\
&  =\sum_{l=0}^{p-1}\sum_{m=0}^{l}\underbrace{\left(  \dbinom{Np+l}{Kp+m}%
^{2}-\left(  \dbinom{N}{K}\dbinom{l}{m}\right)  ^{2}\right)  }%
_{\substack{\equiv2\left(  \dbinom{Np+l}{Kp+m}-\dbinom{N}{K}\dbinom{l}%
{m}\right)  \dbinom{N}{K}\dbinom{l}{m}\operatorname{mod}p^{2}\\\text{(by
(\ref{pf.lem.A8.1}))}}}\\
&  \equiv2\dbinom{N}{K}\underbrace{\sum_{l=0}^{p-1}\sum_{m=0}^{l}\left(
\dbinom{Np+l}{Kp+m}-\dbinom{N}{K}\dbinom{l}{m}\right)  \dbinom{l}{m}%
}_{\substack{\equiv0\operatorname{mod}p^{2}\\\text{(by Lemma \ref{lem.A7})}%
}}\equiv0\operatorname{mod}p^{2}.
\end{align*}
Thus,%
\begin{align*}
\sum_{l=0}^{p-1}\sum_{m=0}^{l}\dbinom{Np+l}{Kp+m}^{2}  &  \equiv\sum
_{l=0}^{p-1}\sum_{m=0}^{l}\left(  \dbinom{N}{K}\dbinom{l}{m}\right)
^{2}=\dbinom{N}{K}^{2}\sum_{l=0}^{p-1}\underbrace{\sum_{m=0}^{l}\dbinom{l}%
{m}^{2}}_{\substack{=\sum_{m=0}^{l}\dbinom{l}{m}\dbinom{l}{m}=\dbinom{2l}%
{l}\\\text{(by Corollary \ref{cor.vandermonde.2},}\\\text{applied to
}u=l\text{ and }w=0\text{)}}}\\
&  =\dbinom{N}{K}^{2}\underbrace{\sum_{l=0}^{p-1}\dbinom{2l}{l}}%
_{\substack{=\sum_{n=0}^{p-1}\dbinom{2n}{n}\equiv\eta_{p}\operatorname{mod}%
p^{2}\\\text{(by Theorem \ref{thm.Z1})}}}\equiv\dbinom{N}{K}^{2}\eta
_{p}\operatorname{mod}p^{2}.
\end{align*}
This proves Lemma \ref{lem.A8}.
\end{proof}
\end{vershort}

\begin{verlong}
\begin{proof}
[Proof of Lemma \ref{lem.A8}.]We have%
\begin{align}
\sum_{l=0}^{p-1}\dbinom{2l}{l}  &  =\sum_{n=0}^{p-1}\dbinom{2n}{n}%
\ \ \ \ \ \ \ \ \ \ \left(
\begin{array}
[c]{c}%
\text{here, we have renamed the}\\
\text{summation index }l\text{ as }n
\end{array}
\right) \nonumber\\
&  \equiv\eta_{p}\operatorname{mod}p^{2}\ \ \ \ \ \ \ \ \ \ \left(  \text{by
Theorem \ref{thm.Z1}}\right)  . \label{pf.lem.A8.long.Z1}%
\end{align}

Fix $l\in\mathbb{N}$. Corollary \ref{cor.vandermonde.2} (applied to $u=l$ and
$w=0$) yields%
\[
\sum_{m=0}^{l}\dbinom{l}{0+m}\dbinom{l}{m}=\dbinom{l+l}{0+l}=\dbinom{2l}{l}%
\]
(since $l+l=2l$ and $0+l=l$). Thus,%
\begin{equation}
\dbinom{2l}{l}=\sum_{m=0}^{l}\underbrace{\dbinom{l}{0+m}}_{\substack{=\dbinom
{l}{m}\\\text{(since }0+m=m\text{)}}}\dbinom{l}{m}=\sum_{m=0}^{l}%
\underbrace{\dbinom{l}{m}\dbinom{l}{m}}_{=\dbinom{l}{m}^{2}}=\sum_{m=0}%
^{l}\dbinom{l}{m}^{2}. \label{pf.lem.A8.long.0}%
\end{equation}

Now, forget that we fixed $l$. We thus have proven (\ref{pf.lem.A8.long.0})
for each $l\in\mathbb{N}$.

Fix $l\in\left\{  0,1,\ldots,p-1\right\}  $ and $m\in\left\{  0,1,\ldots
,p-1\right\}  $. Then, Theorem \ref{thm.lucas} (applied to $a=N$, $b=K$, $c=l$
and $d=m$) yields that $\dbinom{Np+l}{Kp+m}\equiv\dbinom{N}{K}\dbinom{l}%
{m}\operatorname{mod}p$. In other words, $\dbinom{Np+l}{Kp+m}-\dbinom{N}%
{K}\dbinom{l}{m}$ is divisible by $p$. Hence, Lemma \ref{lem.Aab} (applied to
$a=\dbinom{Np+l}{Kp+m}$ and $b=\dbinom{N}{K}\dbinom{l}{m}$) shows that%
\begin{align}
&  \dbinom{Np+l}{Kp+m}^{2}-\left(  \dbinom{N}{K}\dbinom{l}{m}\right)
^{2}\nonumber\\
&  \equiv2\left(  \dbinom{Np+l}{Kp+m}-\dbinom{N}{K}\dbinom{l}{m}\right)
\dbinom{N}{K}\dbinom{l}{m}\operatorname{mod}p^{2}. \label{pf.lem.A8.long.1}%
\end{align}

Now, forget that we fixed $l$ and $m$. We thus have proven
(\ref{pf.lem.A8.long.1}) for all $l\in\left\{  0,1,\ldots,p-1\right\}  $ and
$m\in\left\{  0,1,\ldots,p-1\right\}  $. Now,%
\begin{align*}
&  \sum_{l=0}^{p-1}\sum_{m=0}^{l}\dbinom{Np+l}{Kp+m}^{2}-\sum_{l=0}^{p-1}%
\sum_{m=0}^{l}\left(  \dbinom{N}{K}\dbinom{l}{m}\right)  ^{2}\\
&  =\sum_{l=0}^{p-1}\underbrace{\left(  \sum_{m=0}^{l}\dbinom{Np+l}{Kp+m}%
^{2}-\sum_{m=0}^{l}\left(  \dbinom{N}{K}\dbinom{l}{m}\right)  ^{2}\right)
}_{=\sum_{m=0}^{l}\left(  \dbinom{Np+l}{Kp+m}^{2}-\left(  \dbinom{N}{K}%
\dbinom{l}{m}\right)  ^{2}\right)  }\\
&  =\sum_{l=0}^{p-1}\sum_{m=0}^{l}\underbrace{\left(  \dbinom{Np+l}{Kp+m}%
^{2}-\left(  \dbinom{N}{K}\dbinom{l}{m}\right)  ^{2}\right)  }%
_{\substack{\equiv2\left(  \dbinom{Np+l}{Kp+m}-\dbinom{N}{K}\dbinom{l}%
{m}\right)  \dbinom{N}{K}\dbinom{l}{m}\operatorname{mod}p^{2}\\\text{(by
(\ref{pf.lem.A8.long.1}))}}}\\
&  \equiv\sum_{l=0}^{p-1}\sum_{m=0}^{l}2\left(  \dbinom{Np+l}{Kp+m}-\dbinom
{N}{K}\dbinom{l}{m}\right)  \dbinom{N}{K}\dbinom{l}{m}\\
&  \equiv2\dbinom{N}{K}\underbrace{\sum_{l=0}^{p-1}\sum_{m=0}^{l}\left(
\dbinom{Np+l}{Kp+m}-\dbinom{N}{K}\dbinom{l}{m}\right)  \dbinom{l}{m}%
}_{\substack{\equiv0\operatorname{mod}p^{2}\\\text{(by Lemma \ref{lem.A7})}%
}}\\
&  \equiv2\dbinom{N}{K}\cdot0=0\operatorname{mod}p^{2}.
\end{align*}
Thus,%
\begin{align*}
\sum_{l=0}^{p-1}\sum_{m=0}^{l}\dbinom{Np+l}{Kp+m}^{2}  &  \equiv\sum
_{l=0}^{p-1}\sum_{m=0}^{l}\underbrace{\left(  \dbinom{N}{K}\dbinom{l}%
{m}\right)  ^{2}}_{=\dbinom{N}{K}^{2}\dbinom{l}{m}^{2}}=\sum_{l=0}^{p-1}%
\sum_{m=0}^{l}\dbinom{N}{K}^{2}\dbinom{l}{m}^{2}\\
&  =\dbinom{N}{K}^{2}\sum_{l=0}^{p-1}\underbrace{\sum_{m=0}^{l}\dbinom{l}%
{m}^{2}}_{\substack{=\dbinom{2l}{l}\\\text{(by (\ref{pf.lem.A8.long.0}))}%
}}=\dbinom{N}{K}^{2}\underbrace{\sum_{l=0}^{p-1}\dbinom{2l}{l}}%
_{\substack{\equiv\eta_{p}\operatorname{mod}p^{2}\\\text{(by
(\ref{pf.lem.A8.long.Z1}))}}}\equiv\dbinom{N}{K}^{2}\eta_{p}\operatorname{mod}%
p^{2}.
\end{align*}
This proves Lemma \ref{lem.A8}.
\end{proof}
\end{verlong}

\begin{lemma}
\label{lem.A9a}Let $p$ be a prime. Let $N\in\mathbb{Z}$ and $K\in\mathbb{Z}$.
Let $u$ and $v$ be two elements of $\left\{  0,1,\ldots,p-1\right\}  $
satisfying $u+v\geq p$. Then, $p\mid\dbinom{Np+u+v}{Kp+u}$.
\end{lemma}

\begin{vershort}
\begin{proof}
[Proof of Lemma \ref{lem.A9a}.]We have $u+v\geq p$. Thus, $u+v=p+c$ for some
$c\in\mathbb{N}$. Consider this $c$. From $v\in\left\{  0,1,\ldots
,p-1\right\}  $, we obtain $v<p$. Thus, $c+p=p+c=u+\underbrace{v}_{<p}<u+p$,
so that $c<u\leq p-1$ (since $u\in\left\{  0,1,\ldots,p-1\right\}  $). Thus,
$c\in\left\{  0,1,\ldots,p-1\right\}  $ (since $c\in\mathbb{N}$). Also, $c<u$.
Hence, Proposition \ref{prop.binom.0} (applied to $m=c$ and $n=u$) yields
$\dbinom{c}{u}=0$.

Now, $u+v=p+c$, so that $Np+u+v=Np+p+c=\left(  N+1\right)  p+c$. Hence,%
\begin{align*}
\dbinom{Np+u+v}{Kp+u}  &  =\dbinom{\left(  N+1\right)  p+c}{Kp+u}\equiv
\dbinom{N+1}{K}\underbrace{\dbinom{c}{u}}_{=0}\\
&  \ \ \ \ \ \ \ \ \ \ \left(  \text{by Theorem \ref{thm.lucas}, applied to
}a=N+1\text{, }b=K\text{ and }d=u\right) \\
&  =0\operatorname{mod}p.
\end{align*}
In other words, $p\mid\dbinom{Np+u+v}{Kp+u}$. This proves Lemma \ref{lem.A9a}.
\end{proof}
\end{vershort}

\begin{verlong}
\begin{proof}
[Proof of Lemma \ref{lem.A9a}.]We have $u+v\geq p$. Thus, $u+v-p\in\mathbb{N}%
$. Hence, we can define $c\in\mathbb{N}$ by $c=u+v-p$. Consider this $c$.

From $c=u+v-p$, we obtain $c+p=u+v$. From $u\in\left\{  0,1,\ldots
,p-1\right\}  $, we obtain $u\leq p-1<p$. From $v\in\left\{  0,1,\ldots
,p-1\right\}  $, we obtain $v\leq p-1<p$. Thus, $c+p=u+\underbrace{v}%
_{<p}<u+p$. Subtracting $p$ from both sides of this inequality, we obtain
$c<u\leq p-1$. Thus, $c\in\left\{  0,1,\ldots,p-1\right\}  $ (since
$c\in\mathbb{N}$). Also, $c<u$. Hence, Proposition \ref{prop.binom.0} (applied
to $c$ and $u$ instead of $m$ and $n$) yields $\dbinom{c}{u}=0$ (since
$u\in\left\{  0,1,\ldots,p-1\right\}  \subseteq\mathbb{N}$).

Now, $Np+\underbrace{u+v}_{=c+p}=Np+c+p=\left(  N+1\right)  p+c$. Thus,%
\begin{align*}
\dbinom{Np+u+v}{Kp+u}  &  =\dbinom{\left(  N+1\right)  p+c}{Kp+u}\\
&  \equiv\dbinom{N+1}{K}\underbrace{\dbinom{c}{u}}_{=0}%
\ \ \ \ \ \ \ \ \ \ \left(
\begin{array}
[c]{c}%
\text{by Theorem \ref{thm.lucas} (applied to }a=N+1\text{,}\\
b=K\text{ and }d=u\text{)}%
\end{array}
\right) \\
&  =0\operatorname{mod}p.
\end{align*}
In other words, $p\mid\dbinom{Np+u+v}{Kp+u}$. This proves Lemma \ref{lem.A9a}.
\end{proof}
\end{verlong}

\begin{lemma}
\label{lem.A9}Let $p$ be an odd prime. Let $N\in\mathbb{Z}$ and $K\in
\mathbb{N}$. Then,%
\begin{equation}
\sum_{u=0}^{p-1}\sum_{v=0}^{p-1}\dbinom{Np+u+v}{Kp+u}^{2}\equiv\dbinom{N}%
{K}^{2}\eta_{p}\operatorname{mod}p^{2}.\nonumber
\end{equation}

\end{lemma}

\begin{vershort}
\begin{proof}
[Proof of Lemma \ref{lem.A9}.]If $u$ and $v$ are two elements of $\left\{
0,1,\ldots,p-1\right\}  $ satisfying $v\geq p-u$, then%
\begin{equation}
\dbinom{Np+u+v}{Kp+u}^{2}\equiv0\operatorname{mod}p^{2} \label{pf.lem.A9.1}%
\end{equation}
\footnote{\textit{Proof of (\ref{pf.lem.A9.1}):} Let $u$ and $v$ be two
elements of $\left\{  0,1,\ldots,p-1\right\}  $ satisfying $v\geq p-u$. From
$v\geq p-u$, we obtain $u+v\geq p$. Thus, Lemma \ref{lem.A9a} yields
$p\mid\dbinom{Np+u+v}{Kp+u}$. Hence, $p^{2}\mid\dbinom{Np+u+v}{Kp+u}^{2}$.
This proves (\ref{pf.lem.A9.1}).}.

Hence, any $u\in\left\{  0,1,\ldots,p-1\right\}  $ satisfies%
\begin{align*}
\sum_{v=0}^{p-1}\dbinom{Np+u+v}{Kp+u}^{2}  &  =\sum_{v=0}^{p-u-1}%
\dbinom{Np+u+v}{Kp+u}^{2}+\sum_{v=p-u}^{p-1}\underbrace{\dbinom{Np+u+v}%
{Kp+u}^{2}}_{\substack{\equiv0\operatorname{mod}p^{2}\\\text{(by
(\ref{pf.lem.A9.1}))}}}\\
&  \ \ \ \ \ \ \ \ \ \ \left(  \text{here, we have split the sum at
}v=p-u\right) \\
&  \equiv\sum_{v=0}^{p-u-1}\dbinom{Np+u+v}{Kp+u}^{2}=\sum_{l=u}^{p-1}%
\dbinom{Np+l}{Kp+u}^{2}\operatorname{mod}p^{2}%
\end{align*}
(here, we have substituted $l$ for $u+v$ in the sum). Summing up these
congruences for all $u\in\left\{  0,1,\ldots,p-1\right\}  $, we obtain%
\begin{align*}
&  \sum_{u=0}^{p-1}\sum_{v=0}^{p-1}\dbinom{Np+u+v}{Kp+u}^{2}\\
&  \equiv\underbrace{\sum_{u=0}^{p-1}\sum_{l=u}^{p-1}}_{=\sum_{l=0}^{p-1}%
\sum_{u=0}^{l}}\dbinom{Np+l}{Kp+u}^{2}=\sum_{l=0}^{p-1}\sum_{u=0}^{l}%
\dbinom{Np+l}{Kp+u}^{2}=\sum_{l=0}^{p-1}\sum_{m=0}^{l}\dbinom{Np+l}{Kp+m}%
^{2}\\
&  \ \ \ \ \ \ \ \ \ \ \left(  \text{here, we have renamed the index }u\text{
as }m\text{ in the second sum}\right) \\
&  \equiv\dbinom{N}{K}^{2}\eta_{p}\operatorname{mod}p^{2}%
\end{align*}
(by Lemma \ref{lem.A8}). This proves Lemma \ref{lem.A9}.
\end{proof}
\end{vershort}

\begin{verlong}
\begin{proof}
[Proof of Lemma \ref{lem.A9}.]If $u\in\left\{  0,1,\ldots,p-1\right\}  $ and
$l\in\mathbb{N}$ are such that $l\leq u-1$, then%
\begin{equation}
\dbinom{Np+l}{Kp+u}^{2}\equiv0\operatorname{mod}p^{2} \label{pf.lem.A9.long.3}%
\end{equation}
\footnote{\textit{Proof of (\ref{pf.lem.A9.long.3}):} Let $l\in\mathbb{N}$ be
such that $l\leq u-1$. Thus, $l\leq u-1<u\leq p-1$ (since $u\in\left\{
0,1,\ldots,p-1\right\}  $). From $l<u$, we conclude that $u-l>0$, so that
$u-l\geq1$ (since $u-l$ is an integer). Combined with $u-\underbrace{l}%
_{\geq0}\leq u-0=u\leq p-1$, this yields $1\leq u-l\leq p-1$. Hence,
$u-l\in\left\{  1,2,\ldots,p-1\right\}  $. Therefore, $p-\left(  u-l\right)
\in\left\{  1,2,\ldots,p-1\right\}  \subseteq\left\{  1,2,\ldots,p\right\}  $.
Also, $u+\left(  p-\left(  u-l\right)  \right)  =p+\underbrace{l}_{\geq0}\geq
p$.
\par
Hence, Lemma \ref{lem.A9a} (applied to $N-1$ and $p-\left(  u-l\right)  $
instead of $N$ and $v$) yields $p\mid\dbinom{\left(  N-1\right)  p+u+\left(
p-\left(  u-l\right)  \right)  }{Kp+u}=\dbinom{Np+l}{Kp+u}$ (since $\left(
N-1\right)  p+u+\left(  p-\left(  u-l\right)  \right)  =Np+l$). In other
words, there exists some $z\in\mathbb{Z}$ such that $\dbinom{Np+l}{Kp+u}=zp$.
Consider this $z$. From $\dbinom{Np+l}{Kp+u}=zp$, we obtain%
\[
\dbinom{Np+l}{Kp+u}^{2}=\left(  zp\right)  ^{2}=z^{2}p^{2}\equiv
0\operatorname{mod}p^{2}.
\]
This proves (\ref{pf.lem.A9.long.3}).}.

Fix $u\in\left\{  0,1,\ldots,p-1\right\}  $. Thus, $u\leq p-1$.

If $v\in\left\{  0,1,\ldots,p-1\right\}  $ is such that $v\geq p-u$, then%
\begin{equation}
\dbinom{Np+u+v}{Kp+u}^{2}\equiv0\operatorname{mod}p^{2}
\label{pf.lem.A9.long.1}%
\end{equation}
\footnote{\textit{Proof of (\ref{pf.lem.A9.long.1}):} Let $v\in\left\{
0,1,\ldots,p-1\right\}  $ be such that $v\geq p-u$. From $v\geq p-u$, we
obtain $u+v\geq p$. Thus, Lemma \ref{lem.A9a} yields $p\mid\dbinom
{Np+u+v}{Kp+u}$. In other words, there exists some $z\in\mathbb{Z}$ such that
$\dbinom{Np+u+v}{Kp+u}=zp$. Consider this $z$. From $\dbinom{Np+u+v}{Kp+u}%
=zp$, we obtain%
\[
\dbinom{Np+u+v}{Kp+u}^{2}=\left(  zp\right)  ^{2}=z^{2}p^{2}\equiv
0\operatorname{mod}p^{2}.
\]
This proves (\ref{pf.lem.A9.long.1}).}.

But $u\in\left\{  0,1,\ldots,p-1\right\}  $, so that $p-u\in\left\{
1,2,\ldots,p\right\}  \subseteq\left\{  0,1,\ldots,p\right\}  $. Hence, $0\leq
p-u\leq p$. Thus, we can split the sum $\sum_{v=0}^{p-1}\dbinom{Np+u+v}%
{Kp+u}^{2}$ at $v=p-u$. We thus obtain
\begin{align}
\sum_{v=0}^{p-1}\dbinom{Np+u+v}{Kp+u}^{2}  &  =\sum_{v=0}^{p-u-1}%
\dbinom{Np+u+v}{Kp+u}^{2}+\sum_{v=p-u}^{p-1}\underbrace{\dbinom{Np+u+v}%
{Kp+u}^{2}}_{\substack{\equiv0\operatorname{mod}p^{2}\\\text{(by
(\ref{pf.lem.A9.long.1})}\\\text{(since }v\geq p-u\text{))}}}\nonumber\\
&  \equiv\sum_{v=0}^{p-u-1}\dbinom{Np+u+v}{Kp+u}^{2}+\underbrace{\sum
_{v=p-u}^{p-1}0}_{=0}=\sum_{v=0}^{p-u-1}\dbinom{Np+u+v}{Kp+u}^{2}\nonumber\\
&  =\sum_{l=u}^{p-1}\dbinom{Np+l}{Kp+u}^{2}\operatorname{mod}p^{2}
\label{pf.lem.A9.long.2}%
\end{align}
(here, we have substituted $l$ for $u+v$ in the sum).

On the other hand, $u\in\left\{  0,1,\ldots,p-1\right\}  \subseteq\left\{
0,1,\ldots,p\right\}  $. Thus, $0\leq u\leq p$. Hence, we can split the sum
$\sum_{l=0}^{p-1}\dbinom{Np+l}{Kp+u}^{2}$ at $l=u$. We thus obtain%
\begin{align}
\sum_{l=0}^{p-1}\dbinom{Np+l}{Kp+u}^{2}  &  =\sum_{l=0}^{u-1}%
\underbrace{\dbinom{Np+l}{Kp+u}^{2}}_{\substack{\equiv0\operatorname{mod}%
p^{2}\\\text{(by (\ref{pf.lem.A9.long.3})}\\\text{(since }l\leq u-1\text{))}%
}}+\sum_{l=u}^{p-1}\dbinom{Np+l}{Kp+u}^{2}\nonumber\\
&  \equiv\underbrace{\sum_{l=0}^{u-1}0}_{=0}+\sum_{l=u}^{p-1}\dbinom
{Np+l}{Kp+u}^{2}=\sum_{l=u}^{p-1}\dbinom{Np+l}{Kp+u}^{2}\nonumber\\
&  \equiv\sum_{v=0}^{p-1}\dbinom{Np+u+v}{Kp+u}^{2}\operatorname{mod}%
p^{2}\ \ \ \ \ \ \ \ \ \ \left(  \text{by (\ref{pf.lem.A9.long.2})}\right)  .
\label{pf.lem.A9.long.4}%
\end{align}

Now, forget that we fixed $u$. We thus have proven (\ref{pf.lem.A9.long.4})
for each $u\in\left\{  0,1,\ldots,p-1\right\}  $.

Now,%
\begin{align}
\sum_{u=0}^{p-1}\underbrace{\sum_{v=0}^{p-1}\dbinom{Np+u+v}{Kp+u}^{2}%
}_{\substack{\equiv\sum_{l=0}^{p-1}\dbinom{Np+l}{Kp+u}^{2}\operatorname{mod}%
p^{2}\\\text{(by (\ref{pf.lem.A9.long.4}))}}}  &  \equiv\underbrace{\sum
_{u=0}^{p-1}\sum_{l=0}^{p-1}}_{=\sum_{l=0}^{p-1}\sum_{u=0}^{p-1}}\dbinom
{Np+l}{Kp+u}^{2}\nonumber\\
&  =\sum_{l=0}^{p-1}\sum_{u=0}^{p-1}\dbinom{Np+l}{Kp+u}^{2}\operatorname{mod}%
p^{2}. \label{pf.lem.A9.long.6}%
\end{align}

Now, let us fix $l\in\left\{  0,1,\ldots,p-1\right\}  $. Then, $l\in\left\{
0,1,\ldots,p-1\right\}  $. Hence, we can split the sum $\sum_{u=0}%
^{p-1}\dbinom{Np+l}{Kp+u}^{2}$ at $u=l$. We thus obtain%
\begin{align}
\sum_{u=0}^{p-1}\dbinom{Np+l}{Kp+u}^{2}  &  =\sum_{u=0}^{l}\dbinom{Np+l}%
{Kp+u}^{2}+\sum_{u=l+1}^{p-1}\underbrace{\dbinom{Np+l}{Kp+u}^{2}%
}_{\substack{\equiv0\operatorname{mod}p^{2}\\\text{(by (\ref{pf.lem.A9.long.3}%
)}\\\text{(since }l\leq u-1\text{ (since }u\geq l+1\text{)))}}}\nonumber\\
&  \equiv\sum_{u=0}^{l}\dbinom{Np+l}{Kp+u}^{2}+\underbrace{\sum_{u=l+1}%
^{p-1}0}_{=0}=\sum_{u=0}^{l}\dbinom{Np+l}{Kp+u}^{2}\nonumber\\
&  =\sum_{m=0}^{l}\dbinom{Np+l}{Kp+m}^{2}\operatorname{mod}p^{2}
\label{pf.lem.A9.long.9}%
\end{align}
(here, we have renamed the index $u$ as $m$ in the sum).

Now, forget that we fixed $l$. We thus have proven (\ref{pf.lem.A9.long.9})
for each $l\in\left\{  0,1,\ldots,p-1\right\}  $.

Now, (\ref{pf.lem.A9.long.6}) becomes%
\begin{align*}
\sum_{u=0}^{p-1}\sum_{v=0}^{p-1}\dbinom{Np+u+v}{Kp+u}^{2}  &  \equiv\sum
_{l=0}^{p-1}\underbrace{\sum_{u=0}^{p-1}\dbinom{Np+l}{Kp+u}^{2}}%
_{\substack{\equiv\sum_{m=0}^{l}\dbinom{Np+l}{Kp+m}^{2}\operatorname{mod}%
p^{2}\\\text{(by (\ref{pf.lem.A9.long.9}))}}}\\
&  \equiv\sum_{l=0}^{p-1}\sum_{m=0}^{l}\dbinom{Np+l}{Kp+m}^{2}\equiv\dbinom
{N}{K}^{2}\eta_{p}\operatorname{mod}p^{2}%
\end{align*}
(by Lemma \ref{lem.A8}). This proves Lemma \ref{lem.A9}.
\end{proof}
\end{verlong}

\begin{vershort}
\begin{proof}
[\textit{Proof of Theorem~\ref{thm.Z3}.}]First, let us observe that%
\begin{align}
\epsilon_{r,s}  &  =\sum_{m=0}^{r-1}\sum_{n=0}^{s-1}\dbinom{n+m}{m}^{2}%
=\sum_{n=0}^{s-1}\sum_{m=0}^{r-1}\dbinom{n+m}{m}^{2}=\sum_{K=0}^{s-1}%
\sum_{L=0}^{r-1}\dbinom{K+L}{L}^{2}\nonumber\\
&  =\sum_{K=0}^{s-1}\sum_{L=0}^{r-1}\dbinom{K+L}{K}^{2} \label{pf.thm.Z3.ers}%
\end{align}
(since Proposition \ref{prop.binom.symm} yields $\dbinom{K+L}{L}=\dbinom
{K+L}{K}$ for all $K\in\mathbb{N}$ and $L\in\mathbb{N}$).

Each $n\in\mathbb{N}$ satisfies%
\[
\sum_{m=0}^{sp-1}\dbinom{n+m}{m}^{2}=\sum_{u=0}^{p-1}\sum_{K=0}^{s-1}%
\dbinom{n+Kp+u}{Kp+u}^{2}%
\]
(here, we have substituted $Kp+u$ for $m$ in the sum, since the map%
\begin{align*}
\left\{  0,1,\ldots,p-1\right\}  \times\left\{  0,1,\ldots,s-1\right\}   &
\rightarrow\left\{  0,1,\ldots,sp-1\right\}  ,\\
\left(  u,K\right)   &  \mapsto Kp+u
\end{align*}
is a bijection). Summing up this equality over all $n\in\left\{
0,1,\ldots,rp-1\right\}  $, we obtain%
\begin{align*}
\sum_{n=0}^{rp-1}\sum_{m=0}^{sp-1}\dbinom{n+m}{m}^{2}  &  =\sum_{n=0}%
^{rp-1}\sum_{u=0}^{p-1}\sum_{K=0}^{s-1}\dbinom{n+Kp+u}{Kp+u}^{2}\\
&  =\sum_{v=0}^{p-1}\sum_{L=0}^{r-1}\sum_{u=0}^{p-1}\sum_{K=0}^{s-1}%
\dbinom{Lp+v+Kp+u}{Kp+u}^{2}%
\end{align*}
(here, we have substituted $Lp+v$ for $n$ in the sum, since the map%
\begin{align*}
\left\{  0,1,\ldots,p-1\right\}  \times\left\{  0,1,\ldots,r-1\right\}   &
\rightarrow\left\{  0,1,\ldots,rp-1\right\}  ,\\
\left(  v,L\right)   &  \mapsto Lp+v
\end{align*}
is a bijection).

Thus,%
\begin{align*}
\sum_{n=0}^{rp-1}\sum_{m=0}^{sp-1}\dbinom{n+m}{m}^{2}  &  =\underbrace{\sum
_{v=0}^{p-1}\sum_{L=0}^{r-1}\sum_{u=0}^{p-1}\sum_{K=0}^{s-1}}_{=\sum
_{K=0}^{s-1}\sum_{L=0}^{r-1}\sum_{u=0}^{p-1}\sum_{v=0}^{p-1}}%
\underbrace{\dbinom{Lp+v+Kp+u}{Kp+u}^{2}}_{=\dbinom{\left(  K+L\right)
p+u+v}{Kp+u}^{2}}\\
&  =\sum_{K=0}^{s-1}\sum_{L=0}^{r-1}\underbrace{\sum_{u=0}^{p-1}\sum
_{v=0}^{p-1}\dbinom{\left(  K+L\right)  p+u+v}{Kp+u}^{2}}_{\substack{\equiv
\dbinom{K+L}{K}^{2}\eta_{p}\operatorname{mod}p^{2}\\\text{(by Lemma
\ref{lem.A9}, applied to }N=K+L\text{)}}}\\
&  \equiv\underbrace{\sum_{K=0}^{s-1}\sum_{L=0}^{r-1}\dbinom{K+L}{K}^{2}%
}_{\substack{=\epsilon_{r,s}\\\text{(by (\ref{pf.thm.Z3.ers}))}}}\eta
_{p}=\epsilon_{r,s}\eta_{p}=\eta_{p}\epsilon_{r,s}\operatorname{mod}p^{2}.
\end{align*}
This proves Theorem~\ref{thm.Z3}.
\end{proof}
\end{vershort}

\begin{verlong}
\begin{proof}
[\textit{Proof of Theorem~\ref{thm.Z3}.}]First, let us observe that every
$K\in\mathbb{N}$ and $L\in\mathbb{N}$ satisfy%
\begin{equation}
\dbinom{K+L}{L}^{2}=\dbinom{K+L}{K}^{2} \label{pf.thm.Z3.long.K+L}%
\end{equation}
\footnote{\textit{Proof of (\ref{pf.thm.Z3.long.K+L}):} Let $K\in\mathbb{N}$
and $L\in\mathbb{N}$. Thus, $K+L\in\mathbb{N}$. Also, $\underbrace{K}_{\geq
0}+L\geq L$. Thus, Proposition \ref{prop.binom.symm} (applied to $K+L$ and $L$
instead of $m$ and $n$) yields $\dbinom{K+L}{L}=\dbinom{K+L}{\left(
K+L\right)  -L}=\dbinom{K+L}{K}$ (since $\left(  K+L\right)  -L=K$). Hence,
$\dbinom{K+L}{L}^{2}=\dbinom{K+L}{K}^{2}$. This proves
(\ref{pf.thm.Z3.long.K+L}).}.

Now,%
\begin{align}
\epsilon_{r,s}  &  =\underbrace{\sum_{m=0}^{r-1}\sum_{n=0}^{s-1}}_{=\sum
_{n=0}^{s-1}\sum_{m=0}^{r-1}}\dbinom{n+m}{m}^{2}=\sum_{n=0}^{s-1}%
\underbrace{\sum_{m=0}^{r-1}\dbinom{n+m}{m}^{2}}_{\substack{=\sum_{L=0}%
^{r-1}\dbinom{n+L}{L}^{2}\\\text{(here, we have renamed the}\\\text{summation
index }m\text{ as }L\text{)}}}=\sum_{n=0}^{s-1}\sum_{L=0}^{r-1}\dbinom{n+L}%
{L}^{2}\nonumber\\
&  =\sum_{K=0}^{s-1}\sum_{L=0}^{r-1}\underbrace{\dbinom{K+L}{L}^{2}%
}_{\substack{=\dbinom{K+L}{K}^{2}\\\text{(by (\ref{pf.thm.Z3.long.K+L}))}%
}}\ \ \ \ \ \ \ \ \ \ \left(
\begin{array}
[c]{c}%
\text{here, we have renamed the}\\
\text{summation index }n\text{ as }K
\end{array}
\right) \nonumber\\
&  =\sum_{K=0}^{s-1}\sum_{L=0}^{r-1}\dbinom{K+L}{K}^{2}.
\label{pf.thm.Z3.long.ers}%
\end{align}

Now, let $n\in\mathbb{Z}$. Lemma \ref{lem.pr-bij} (applied to $s$ instead of
$r$) shows that the map%
\begin{align*}
\left\{  0,1,\ldots,p-1\right\}  \times\left\{  0,1,\ldots,s-1\right\}   &
\rightarrow\left\{  0,1,\ldots,sp-1\right\}  ,\\
\left(  l,K\right)   &  \mapsto Kp+l
\end{align*}
is well-defined and is a bijection. Thus, we can substitute $Kp+l$ for $m$ in
the sum $\sum_{m\in\left\{  0,1,\ldots,sp-1\right\}  }\dbinom{n+m}{m}^{2}$. We
thus obtain%
\[
\sum_{m\in\left\{  0,1,\ldots,sp-1\right\}  }\dbinom{n+m}{m}^{2}=\sum_{\left(
l,K\right)  \in\left\{  0,1,\ldots,p-1\right\}  \times\left\{  0,1,\ldots
,s-1\right\}  }\dbinom{n+Kp+l}{Kp+l}^{2}.
\]

We have%
\begin{align}
&  \underbrace{\sum_{m=0}^{sp-1}}_{=\sum_{m\in\left\{  0,1,\ldots
,sp-1\right\}  }}\dbinom{n+m}{m}^{2}\nonumber\\
&  =\sum_{m\in\left\{  0,1,\ldots,sp-1\right\}  }\dbinom{n+m}{m}%
^{2}=\underbrace{\sum_{\left(  l,K\right)  \in\left\{  0,1,\ldots,p-1\right\}
\times\left\{  0,1,\ldots,s-1\right\}  }}_{=\sum_{l\in\left\{  0,1,\ldots
,p-1\right\}  }\sum_{K\in\left\{  0,1,\ldots,s-1\right\}  }}\dbinom
{n+Kp+l}{Kp+l}^{2}\nonumber\\
&  =\underbrace{\sum_{l\in\left\{  0,1,\ldots,p-1\right\}  }}_{=\sum
_{l=0}^{p-1}}\underbrace{\sum_{K\in\left\{  0,1,\ldots,s-1\right\}  }}%
_{=\sum_{K=0}^{s-1}}\dbinom{n+Kp+l}{Kp+l}^{2}=\sum_{l=0}^{p-1}\sum_{K=0}%
^{s-1}\dbinom{n+Kp+l}{Kp+l}^{2}\nonumber\\
&  =\sum_{u=0}^{p-1}\sum_{K=0}^{s-1}\dbinom{n+Kp+u}{Kp+u}^{2}
\label{pf.thm.Z3.long.3}%
\end{align}
(here, we have renamed the summation index $l$ as $u$ in the first sum).

Now, forget that we fixed $n$. We thus have proven (\ref{pf.thm.Z3.long.3})
for each $n\in\mathbb{Z}$.

On the other hand, Lemma \ref{lem.pr-bij} shows that the map%
\begin{align*}
\left\{  0,1,\ldots,p-1\right\}  \times\left\{  0,1,\ldots,r-1\right\}   &
\rightarrow\left\{  0,1,\ldots,rp-1\right\}  ,\\
\left(  l,K\right)   &  \mapsto Kp+l
\end{align*}
is well-defined and is a bijection. Hence, we can substitute $Kp+l$ for $n$ in
the sum $\sum_{n\in\left\{  0,1,\ldots,rp-1\right\}  }\sum_{m=0}^{sp-1}%
\dbinom{n+m}{m}^{2}$. We thus obtain%
\[
\sum_{n\in\left\{  0,1,\ldots,rp-1\right\}  }\sum_{m=0}^{sp-1}\dbinom{n+m}%
{m}^{2}=\sum_{\left(  l,K\right)  \in\left\{  0,1,\ldots,p-1\right\}
\times\left\{  0,1,\ldots,r-1\right\}  }\sum_{m=0}^{sp-1}\dbinom{Kp+l+m}%
{m}^{2}.
\]

Now,%
\begin{align*}
&  \underbrace{\sum_{n=0}^{rp-1}}_{=\sum_{n\in\left\{  0,1,\ldots
,rp-1\right\}  }}\sum_{m=0}^{sp-1}\dbinom{n+m}{m}^{2}\\
&  =\sum_{n\in\left\{  0,1,\ldots,rp-1\right\}  }\sum_{m=0}^{sp-1}\dbinom
{n+m}{m}^{2}\\
&  =\sum_{\left(  l,K\right)  \in\left\{  0,1,\ldots,p-1\right\}
\times\left\{  0,1,\ldots,r-1\right\}  }\sum_{m=0}^{sp-1}\dbinom{Kp+l+m}%
{m}^{2}\\
&  =\underbrace{\sum_{\left(  v,L\right)  \in\left\{  0,1,\ldots,p-1\right\}
\times\left\{  0,1,\ldots,r-1\right\}  }}_{=\sum_{v\in\left\{  0,1,\ldots
,p-1\right\}  }\sum_{L\in\left\{  0,1,\ldots,r-1\right\}  }}\sum_{m=0}%
^{sp-1}\dbinom{Lp+v+m}{m}^{2}\\
&  \ \ \ \ \ \ \ \ \ \ \left(
\begin{array}
[c]{c}%
\text{here, we have renamed the}\\
\text{summation index }\left(  l,K\right)  \text{ as }\left(  v,L\right)
\end{array}
\right) \\
&  =\underbrace{\sum_{v\in\left\{  0,1,\ldots,p-1\right\}  }}_{=\sum
_{v=0}^{p-1}}\underbrace{\sum_{L\in\left\{  0,1,\ldots,r-1\right\}  }}%
_{=\sum_{L=0}^{r-1}}\underbrace{\sum_{m=0}^{sp-1}\dbinom{Lp+v+m}{m}^{2}%
}_{\substack{=\sum_{u=0}^{p-1}\sum_{K=0}^{s-1}\dbinom{Lp+v+Kp+u}{Kp+u}%
^{2}\\\text{(by (\ref{pf.thm.Z3.long.3}) (applied to }n=Lp+v\text{))}}}\\
&  =\underbrace{\sum_{v=0}^{p-1}\sum_{L=0}^{r-1}\sum_{u=0}^{p-1}\sum
_{K=0}^{s-1}}_{=\sum_{K=0}^{s-1}\sum_{L=0}^{r-1}\sum_{u=0}^{p-1}\sum
_{v=0}^{p-1}}\underbrace{\dbinom{Lp+v+Kp+u}{Kp+u}^{2}}_{\substack{=\dbinom
{\left(  K+L\right)  p+u+v}{Kp+u}^{2}\\\text{(since }Lp+v+Kp+u=\left(
K+L\right)  p+u+v\text{)}}}\\
&  =\sum_{K=0}^{s-1}\sum_{L=0}^{r-1}\underbrace{\sum_{u=0}^{p-1}\sum
_{v=0}^{p-1}\dbinom{\left(  K+L\right)  p+u+v}{Kp+u}^{2}}_{\substack{\equiv
\dbinom{K+L}{K}^{2}\eta_{p}\operatorname{mod}p^{2}\\\text{(by Lemma
\ref{lem.A9} (applied to }N=K+L\text{))}}}\\
&  \equiv\sum_{K=0}^{s-1}\sum_{L=0}^{r-1}\dbinom{K+L}{K}^{2}\eta_{p}=\eta
_{p}\underbrace{\sum_{K=0}^{s-1}\sum_{L=0}^{r-1}\dbinom{K+L}{K}^{2}%
}_{\substack{=\epsilon_{r,s}\\\text{(by (\ref{pf.thm.Z3.long.ers}))}}%
}=\eta_{p}\epsilon_{r,s}\operatorname{mod}p^{2}.
\end{align*}
This proves Theorem~\ref{thm.Z3}.
\end{proof}
\end{verlong}

\subsection{Acknowledgments}

Thanks to Doron Zeilberger and Roberto Tauraso for alerting me to
\cite{AmdTau16} and \cite{SunTau11}.


\begin{thebibliography}{999999999}                                                                                        %


\bibitem[AmdTau16]{AmdTau16}%
\href{https://dx.doi.org/10.1016/j.jnt.2016.11.014}{Tewodros Amdeberhan,
Roberto Tauraso, \textit{Two triple binomial sum supercongruences}, Journal of
Number Theory \textbf{175} (2017), pp. 140--157}. A preprint is
\href{https://arxiv.org/abs/1607.02483v1}{arXiv:1607.02483v1}.

\bibitem[AnBeRo05]{AnBeRo05}Peter G. Anderson, Arthur T. Benjamin and Jeremy
A. Rouse, \textit{Combinatorial Proofs of Fermat's, Lucas's, and Wilson's
Theorems}, The American Mathematical Monthly, Vol. \textbf{112}, No. 3 (Mar.,
2005), pp. 266--268.

\bibitem[ApaZei16]{ApaZei16}%
\href{https://dx.doi.org/10.4169/amer.math.monthly.124.7.597}{Moa Apagodu,
Doron Zeilberger, \textit{Using the \textquotedblleft Freshman's
Dream\textquotedblright\ to Prove Combinatorial Congruences}, The American
Mathematical Monthly, Vol. \textbf{124}, No. 7 (August-September 2017), pp.
597--608}.\newline(A preprint can be found at
\href{https://arxiv.org/abs/1606.03351v2}{arXiv:1606.03351v2}, but is less
up-to-date and uses a different numbering of the conjectures.)

\bibitem[Bailey91]{Bailey91}D. F. Bailey, \textit{Some Binomial Coefficient
Congruences}, Applied Mathematics Letters, Volume \textbf{4}, Issue 4, 1991,
pp. 1--5.\newline\url{https://doi.org/10.1016/0893-9659(91)90043-U}

\bibitem[BenQui03]{BenQui03}Arthur T. Benjamin and Jennifer J. Quinn,
\textit{Proofs that Really Count: The Art of Combinatorial Proof}, The
Mathematical Association of America, 2003.

\bibitem[BenQui08]{BenQui08}%
\href{https://www.math.hmc.edu/~benjamin/papers/DIE.pdf}{Arthur T. Benjamin
and Jennifer J. Quinn, \textit{An Alternate Approach to Alternating Sums: A
Method to DIE for}, The College Mathematics Journal, Volume 39, Number 3, May
2008, pp. 191-202(12).}

\bibitem[GrKnPa94]{GKP}Ronald L. Graham, Donald E. Knuth, Oren Patashnik,
\textit{Concrete Mathematics, Second Edition}, Addison-Wesley 1994.

\begin{verlong}


\bibitem[Grinbe16]{fleck}Darij Grinberg, \textit{Fleck's binomial congruence
using circulant matrices}, 3 October 2018.\newline\url{http://www.cip.ifi.lmu.de/~grinberg/fleck.pdf}
\end{verlong}

\bibitem[Grinbe17]{detnotes}Darij Grinberg, \textit{Notes on the combinatorial
fundamentals of algebra}, 4 October 2018. \newline%
\url{https://github.com/darijgr/detnotes/releases/tag/2018-10-04} \newline See
also \url{http://www.cip.ifi.lmu.de/~grinberg/primes2015/sols.pdf} for a
version that is getting updates.

\bibitem[Grinbe17b]{lucas}Darij Grinberg, \textit{The Lucas and Babbage
congruences}, 3 October 2018.\newline\url{http://www.cip.ifi.lmu.de/~grinberg/lucascong.pdf}

\bibitem[Hausne83]{Hausne83}Melvin Hausner, \textit{Applications of a Simple
Counting Technique}, The American Mathematical Monthly, Vol. 90, No. 2 (Feb.,
1983), pp. 127--129.

\bibitem[Liu16]{Liu16}\href{https://arxiv.org/abs/1606.08432v3}{Ji-Cai Liu,
\textit{On two conjectural supercongruences of Apagodu and Zeilberger},
arXiv:1606.08432v3.}

\bibitem[MacSon10]{MacSon10}\href{https://arxiv.org/abs/1011.0076v1}{Kieren
MacMillan and Jonathan Sondow, \textit{Proofs of Power Sum and Binomial
Coefficient Congruences Via Pascal's Identity}, arXiv:1011.0076v1.}\newline
Published in: The American Mathematical Monthly, Vol. \textbf{118} (2011), pp. 549--551.

\bibitem[Mestro14]{Mestro14}\href{https://arxiv.org/abs/1409.3820v1}{Romeo
Me\v{s}trovi\'{c}, \textit{Lucas' theorem: its generalizations, extensions and
applications (1878--2014)}, arXiv:1409.3820v1.}

\bibitem[Stanle11]{Stanley-EC1}Richard Stanley, \textit{Enumerative
Combinatorics, volume 1}, Second edition, version of 15 July 2011. Available
at \url{http://math.mit.edu/~rstan/ec/} .

\bibitem[Stanto16]{Stanto16}Dennis Stanton, \textit{Addendum to
\textquotedblleft Using the \textquotedblleft Freshman's
Dream\textquotedblright\ to Prove Combinatorial Congruences\textquotedblright%
},\newline\url{http://www.math.rutgers.edu/~zeilberg/mamarim/mamarimhtml/freshmanDennisStanton.pdf}

\bibitem[SunTau11]{SunTau11}%
\href{https://dx.doi.org/10.1142/S1793042111004393}{Zhi-Wei Sun, Roberto
Tauraso, \textit{On some new congruences for binomial coefficients},
International Journal of Number Theory, Vol. \textbf{7}, No. 3 (2011), pp.
645--662}. A preprint is
\href{https://arxiv.org/abs/0709.1665v10}{arXiv:0709.1665v10}.
\end{thebibliography}
\end{document}